\documentclass[12pt,reqno]{amsart}

\usepackage{amsmath,amsthm,amssymb, amscd, amsfonts,enumerate,array,bbm,bm,
geometry}
\usepackage[sorted,compressed-cites,sorted-cites,initials]{amsrefs}
\usepackage[all,cmtip]{xy}
\usepackage{pifont,mathrsfs,mathtools,stmaryrd}
\usepackage[colorlinks]{hyperref}
\hypersetup{
bookmarksnumbered,
pdfstartview={FitH},
breaklinks=true,
linkcolor=blue,
urlcolor=blue,
citecolor=blue,
bookmarksdepth=2
}

\DeclareMathOperator{\can}{can}
\newcommand{\prpref}[2]{\hyperref[#1]{($#2_{\ref*{#1}}$)}}
\newcommand{\fgfrm}[2]{\llparenthesis#1,#2\rrparenthesis}

\allowdisplaybreaks[3]
\makeatletter
\def\mhline{\noalign{\ifnum0=`}\fi\hrule height 4\arrayrulewidth \futurelet
   \@tempa\oxhline}
\def\oxhline{\ifx\@tempa\hline\vskip \doublerulesep\fi
      \ifnum0=`{\fi}}

\makeatother

\numberwithin{equation}{section}

\newtheorem{theorem}{Theorem}[section]
\newtheorem{proposition}[theorem]{Proposition}

\newtheorem{corollary}[theorem]{Corollary}
\newtheorem{lemma}[theorem]{Lemma}

\theoremstyle{definition}

\newtheorem{remark}[theorem]{Remark}

\newtheorem{definition}[theorem]{Definition}

\newcommand{\qbinom}[3][q]{\genfrac[]{0pt}0{#2}{#3}_{#1}}
\def\QBin[#1,#2,#3]{\qbinom[#3]{#1}{#2}}

\newcommand{\lie}[1]{\mathfrak{#1}}

\DeclareMathOperator{\ad}{ad}

\DeclareMathOperator{\sign}{sign}

\newcommand{\bigobj}{\let\objectstyle=\displaystyle}
\newcommand{\smallobj}{\let\objectstyle=\scriptstyle}
\smallobj

\newcommand{\ul}[1]{\underline{#1}}

\newcommand{\tensor}{\otimes}
\renewcommand{\eqref}[1]{{\rm (\ref{#1})}}

\def\ad{\operatorname{ad}}
\DeclareMathOperator{\End}{End}

\def\AA{\mathbb{A}}

\def\ZZ{\mathbb{Z}}

\def\QQ{\mathbb{Q}}

\def\QF(#1){[#1]_q!}

\DeclareMathOperator{\Hom}{Hom}

\def\id{\operatorname{id}}

\def\kk{\Bbbk}

\def\bb{\mathfrak{b}}

\def\ii{\mathbf{i}}

\DeclareMathOperator{\sgn}{sgn}

\newcommand{\vertar}{\ar@{->>}[u]}
\newcommand{\horar}{\ar@{>->}[l]}
\newcommand{\la}{\langle}
\newcommand{\ra}{\rangle}

\DeclareRobustCommand{\SkipTocEntry}[5]{}
\begin{document}
\newgeometry{margin=2cm}

\title{Canonical bases of quantum Schubert cells and their symmetries}

\author[]{Arkady Berenstein}
\address{\noindent Department of Mathematics, University of Oregon,
Eugene, OR 97403, USA} \email{arkadiy@math.uoregon.edu}

\begin{abstract}
The goal of this work is to provide an elementary construction of the canonical basis $\mathbf B(w)$ in each quantum Schubert
cell~$U_q(w)$ and to establish its invariance under modified Lusztig's symmetries. To that effect, we obtain a direct characterization 
of the upper global basis $\mathbf B^{up}$ in terms of a suitable bilinear form and show that $\mathbf B(w)$ is contained 
in $\mathbf B^{up}$ and its large part is preserved by modified Lusztig's symmetries. 
\end{abstract}

\thanks{This work was partially supported by the NSF grant~DMS-1403527 (A.~B.),
by the Simons foundation collaboration grant no.~245735~(J.~G.), and by the ERC grant MODFLAT and the NCCR SwissMAP of the Swiss National Science Foundation
(A.~B. and J.~G.)}

\author[]{Jacob Greenstein}
\address{Department of Mathematics, University of
California, Riverside, CA 92521.} 
 \email{jacob.greenstein@ucr.edu}
\maketitle

\def\bI{{\boldsymbol{I}}}
\def\bB{{\boldsymbol{B}}}
\def\bb{{\boldsymbol{b}}}

\tableofcontents
\newpage
\section{Introduction and main results}
Let $\lie g=\lie b^-\oplus \lie n^+$ be a symmetrizable Kac-Moody Lie algebra.
For any $w$ in its Weyl group~$W$, 
define the algebra $U_q(w)$ by 
\begin{equation}\label{eq:quantum-Schubert-cell}
U_q(w):=T_w(U_q(\lie b^-))\cap U_q(\lie n^+) 
\end{equation}
which we refer to as a quantum Schubert cell (see \S\ref{subs:prelim} for notation). This terminology is justified in Remark~\ref{rem:2}.
The definition~\eqref{eq:quantum-Schubert-cell} for an infinite (affine) type
first appeared in~\cite{BCP}*{Proposition~2.3}. In~\cite{BG-dcb} we conjectured that this definition coincides with Lusztig's one 
which was proved for all Kac-Moody algebras
by Tanisaki in \cite{T}*{Proposition~2.10} and independently by Kimura (\cite{Kim1}*{Theorem~1.3}).

In a remarkable paper~\cite{Kim} Kimura proved that each $U_q(w)$
is compatible with the upper global basis $\mathbf B^{up}$ of~$U_q(\lie n^+)$. The aim of the present 
work is twofold:
\begin{itemize}
 \item to construct the basis $\mathbf B(w)$ of~$U_q(w)$ explicitly using a generalization of Lusztig's Lemma.
 \item to compute the action of Lusztig symmetries on these bases, thus partially verifying Conjecture~1.16 from~\cite{BG-dcb}.
\end{itemize}
To achieve the first goal, first we provide an independent definition (see~\S\ref{subs:defn-basis} and~\S\ref{subs:chos basis}) of the global crystal basis $\mathbf B^{up}$ (which coincides with the dual canonical basis). 
For reader's convenience, we put all necessary definitions and results in Section~\ref{sec:char bas}.

Let $\bar\cdot$ be the anti-linear anti-involution of~$U_q(\lie g)$ which maps $q^{\frac12}$ to~$q^{-\frac12}$ and fixes Chevalley generators. 
It should be noted that we use a slightly different presentation of~$U_q(\lie g)$ (see~\cite{BG-dcb} and \S\ref{subs:prelim}) and accordingly modified~$T_w$ 
so that they commute with $\bar\cdot$ (\cite{BG-dcb}*{}). 

Let $\ii=(i_1,\dots,i_m)\in R(w)$.
Generalizing~\cites{DKP,LS}, we show (see~\S\ref{subs:some prop q S c}) that the $\AA:=\ZZ[q^{\frac12},q^{-\frac12}]$-subalgebra 
of~$U_q(\lie n^+)$ generated by the $X_{\mathbf i,k}:=T_{s_{i_1}\cdots s_{i_{k-1}}}(E_{i_k})$, $1\le k\le m$  of~$U_q(\lie n^+)$
is in fact independent of~$\ii$, hence is denoted $U^\AA(w)$, and by Lemma~\ref{lem:lus PBW elt} has an $\AA$-basis $\{X_\ii^{\mathbf a}\,:\, \mathbf a\in \ZZ_{\ge 0}^m\}$ where 
$X_\ii^{\mathbf a}=q_{\mathbf i,\mathbf a}X_{\ii,1}^{a_1}\cdots X_{\ii,m}^{a_m}$ and $q_{\mathbf i,\mathbf a}\in q^{\frac12\ZZ}$ is defined in~\eqref{eq:q_i a defn}.
The importance of this choice of the $q_{\ii,\mathbf a}$ is highlighted by the following version of Lusztig's Lemma.
\begin{theorem}\label{thm:bas-from-Lusztig-lemma}
Let $w\in W$ and $\ii\in R(w)$.
For every $\mathbf a\in \ZZ_{\ge 0}^m$ there exists a unique $b_{\mathbf a}=b_{\mathbf i,\mathbf a}\in U^\AA(w)$ such that $\overline{b_{\mathbf a}}=b_{\mathbf a}$ and 
$
b_{\mathbf a}-X_\ii^{\mathbf a}\in\sum_{\mathbf a'\not=\mathbf a} q^{-1} \ZZ[q^{-1}] X_\ii^{\mathbf a'}$.
\end{theorem}
We prove Theorem~\ref{thm:bas-from-Lusztig-lemma} in~\S\ref{subs:pf-bas-from-Lus-lemma}.
In particular, elements $b_{\ii,\mathbf a}$, $\mathbf a\in\ZZ_{\ge 0}^m$ form a basis $\mathbf B(\ii)$ of $U^\AA(w)$ which a priori depends on~$\ii$. However,
the following result implies that this is not the case.
\begin{theorem}\label{thm:inclusion of base}
Let $w\in W$ and $\mathbf i\in R(w)$. Then for all $\mathbf a\in \ZZ_{\ge 0}^{\ell(w)}$ we have $b_{\mathbf i,\mathbf a}\in \mathbf B^{up}$.
\end{theorem}
Theorem~\ref{thm:inclusion of base} is proved in~\S\ref{subs:pf inclusion of base}. It implies that for any $\ii,\ii'\in R(w)$ we have $\mathbf B(\ii)=\mathbf B(\ii')$ and thus can introduce $\mathbf B(w)$.
As a consequence, we recover the main result (Theorem~4.22) of~\cite{Kim}.
\begin{corollary}\label{cor:Kim}
$\mathbf B(w)=\mathbf B^{up}\cap U_q(w)$ for all $w\in W$.
\end{corollary}
\begin{remark}
To obtain his result, Kimura used a rather elaborate theory of global crystal bases. By contrast, our proofs of Theorems~\ref{thm:bas-from-Lusztig-lemma} 
and~\ref{thm:inclusion of base} 
are quite elementary and short.
\end{remark}

Now we turn our attention to the second goal, that is, to the action of Lusztig's symmetries on~$U_q(w)$.
\begin{theorem}[\cite{BG-dcb}*{Conjecture~1.17}]\label{thm:T_w(B(w'))}
Let $w,w'\in W$ be such that 
$\ell(ww')=\ell(w)+\ell(w')$. Then 
$$
\mathbf B(w)\subset \mathbf B(ww'),\qquad T_w(\mathbf B(w'))\subset \mathbf B(w w').
$$
\end{theorem}
\begin{remark}
\begin{enumerate}[{\rm(a)}]
 \item 
In~\cite{BG-dcb} we constructed a basis $\mathbf B_{\lie g}$ of~$U_q(\lie g)$ containing $\mathbf B^{up}$ and conjectured (\cite{BG-dcb}*{Conjecture~1.16}) 
that $T_w(\mathbf B^{up})\subset \mathbf B_{\lie g}$. Thus, Theorem~\ref{thm:T_w(B(w'))} provides supporting evidence for that conjecture.
\item It would be interesting to compare the symmetries discussed above with the quantum twist computed in~\cite{KimOya}.
\end{enumerate}
\end{remark}

We deduce Theorem~\ref{thm:T_w(B(w'))} from Theorem~\ref{thm:bas-from-Lusztig-lemma} in~\S\ref{subs:pf T_w(B(w'))}.
All these results are obtained using the following striking property of~$\mathbf B^{up}$ which is parallel to a highly non-trivial result  
of Lusztig (\cite{Lus-adv}).
\begin{theorem}\label{thm:T_i bas}
$T_{s_i}(b)\in\mathbf B^{up}$ whenever $b\in \mathbf B^{up}\cap T_{s_i}^{-1}(U_q(\lie n^+))$.
\end{theorem}
We prove Theorem~\ref{thm:T_i bas} in~\S\ref{subs:proof of thm T_i bas}. Our proof, which is quite elementary and short, relies on the notion of {\em decorated algebras} 
(Definition~\ref{defn:decorated alg}) to which 
we generalize $T_{s_i}$ and obtain an explicit formula (Theorem~\ref{thm:tau-homomorphism}) for it.

We conclude this section with the following curious application of the above constructions. It is well-known 
(see e.g. Remark~\ref{rem: star op partial}) that the natural linear anti-involution ${}^*$ on~$U_q(\lie n^+)$ fixing the Chevalley generators
(see~\S\ref{subs:prelim})
preserves $\mathbf B^{up}$. Since $T_w\circ {}^*={}^*\circ T_{w^{-1}}^{-1}$ (cf.~\cite{BG-dcb}) it follows that 
$$
U_q(w)^*=T_{w^{-1}}^{-1}(U_q(\lie b^-))\cap U_q(\lie n^+)
$$
and Corollary~\ref{cor:Kim} implies that  $U_q(w)^*$
has a basis $\mathbf B(w)^*=U_q(w)^*\cap \mathbf B^{up}$. In particular, one can consider the algebras 
$$
U_q(w,w'):=U_q(w)\cap U_q(w')^*,\qquad w,w'\in W
$$
which is natural to call bi-Schubert algebras. The following is immediate.
\begin{corollary}
For any $w,w'\in W$, the bi-Schubert algebra $U_q(w,w')$ has a basis $\mathbf B(w,w'):=\mathbf B(w)\cap \mathbf B(w')^*=U_q(w,w')\cap \mathbf B^{up}$.
\end{corollary}
\begin{remark}
\begin{enumerate}[\rm(a)]
 \item \label{rem:intersect sc.a}
Based on numerous examples (see~\S\ref{subs:bi-schub}) one can conjecture that bi-Schubert algebras are Poin\-ca\-r\'e-Birkhoff-Witt (PBW), i.e. have 
bases consisting of ordered monomials (similarly to Lemma~\ref{lem:lus PBW elt}).
\item\label{rem:intersect sc.b} 
One can also consider intersections $U_q(w)\cap U_q(w')$; however, in this case it appears (and is probably well-known) that the corresponding algebra is 
always $U_q(w'')$ where $w''$ is less than both $w$ and~$w'$ in the weak right Bruhat order and is maximal with that property.
\end{enumerate}\label{rem:intersect sc}
\end{remark}
\addtocontents{toc}{\SkipTocEntry}
\subsection*{Acknowledgements}
The main part of this paper was written while both authors were visiting Universit\'e de Gen\`eve (Geneva, Switzerland). We are happy to use 
this opportunity to thank A.~Alekseev for his hospitality. We also benefited from
the hospitality of Max-Planck-Institut f\"ur Mathematik (Bonn, Germany), which we
gratefully acknowledge.

\section{Definition and characterization of~\texorpdfstring{$\mathbf B^{up}$}{B\^up}}\label{sec:char bas}

\subsection{Preliminaries}\label{subs:prelim}
Let $\lie g$ be a symmetrizable Kac-Moody algebra with the Cartan matrix $A=(a_{ij})_{i,j\in I}$. Let $\{\alpha_i\}_{i\in I}$ be 
the standard basis of~$Q=\ZZ^I$. Fix $d_i\in \ZZ_{>0}$, $i\in I$ such that the matrix $(d_i a_{ij})_{i,j\in I}$ is symmetric and 
define a symmetric bilinear form $(\cdot,\cdot):Q\times Q\to\ZZ$ by $(\alpha_i,\alpha_j)=d_i a_{ij}$; clearly, $(\gamma,\gamma)\in2\ZZ$ for any $\gamma\in Q$.
We will write $(\alpha_i^\vee,\gamma)$, $\gamma\in Q$ as an abbreviation for $(\alpha_i,\gamma)d_i^{-1}$.

The quantized 
enveloping algebra $U_q(\lie g)$ is an associative algebra over $\kk=\mathbb Q(q^{\frac12})$ generated by the $E_i$, $F_i$, $K_i^{\pm 1}$, $i\in I$
subject to the relations
\begin{gather}
\label{eq:commutation g}
[E_i,F_j]=\delta_{ij}(q_i^{-1}-q_i)(K_{i}-K_{i}^{-1})
,\quad K_i E_j =
q_i^{a_{ij}}E_j K_i,\,\, K_i F_j =q_i^{-a_{ij}}F_jK_{i},\,\, K_iK_j=K_jK_i
\\
\label{eq:qserre}
\sum\limits_{r,s\ge 0,\,r+s=1-a_{ij}}\mskip-8mu (-1)^s  E_i^{\la r\ra}E_j E_i^{\la s\ra}= 
\sum\limits_{r,s\ge 0,\,r+s=1-a_{ij}}\mskip-8mu (-1)^s  F_i^{\la r\ra}F_j F_i^{\la s\ra}=0,\quad i\not=j
\end{gather}
for all $i,j\in I$,  where $q_i=q^{d_i}$, 
$X_i^{\la k\ra}:=\big(\prod\limits_{s=1}^k \la s\ra_{q_i}\big)^{-1} X_i^k$ and $\la s\ra_v=v^s-v^{-s}$.
We also set 
$$
(n)_v=\frac{\la n\ra_v}{\la 1\ra_v},\quad  \la n\ra_v!=\prod_{t=1}^n \la t\ra_v,\quad  (n)_v!=\frac{\la n\ra_v!}{\la 1\ra_v^n},\quad 
\binom{n}{k}_v=\frac{\prod_{t=0}^{k-1} (n-t)_v}{(k)_v!}=\frac{\prod_{t=0}^{k-1} \la n-t\ra_v}{\la k\ra_v!}
$$
and $X_i^{(n)}:=X_i^n/(n)_{q_i}!$. 

We denote by $U_q(\lie n^+)$ (respectively, $U_q(\lie n^-)$) the subalgebra of $U_q(\lie g)$ generated by the $E_i$ (respectively, the $F_i$),
$i\in I$. Let $\mathcal K$ be the subalgebra of $U_q(\lie g)$ generated by the $K_i^{\pm 1}$, $i\in I$ and set $U_q(\lie b^\pm)=\mathcal KU_q(\lie n^\pm)$.

It is easy to see from the presentation that $U_q(\lie g)$ admits anti-involutions ${}^t$ and ${}^*$, where ${}^t$ interchanges $E_i$ and $F_i$ for 
each $i\in I$ and preserves the $K_i^{\pm 1}$ while ${}^*$ preserves the $E_i$ and $F_i$ while $K_i^*=K_i^{-1}$. Furthermore, $U_q(\lie g)$
admits an anti-linear anti-involution $\bar\cdot$ which preserves all generators and maps $q^{\frac12}$ to $q^{-\frac12}$.

The algebra $U_q(\lie n^+)$ is naturally graded by $Q^+:=\bigoplus_{i\in I}\ZZ_{\ge 0}\alpha_i$ via $\deg E_i=\alpha_i$. We denote 
the homogeneous component of $U_q(\lie n^+)$ of degree~$\gamma\in Q^+$ by $U_q(\lie n^+)_\gamma$.
This can be extended to a $Q$-grading on~$U_q(\lie g)$ via $\deg F_i=-\alpha_i$, $\deg K_i=0$.

\subsection{Modified Lusztig symmetries}\label{subs:Weyl-Lusztig symm}
Let~$W$ be the Weyl group of~$\lie g$. It is generated by the simple reflections $s_i$, $i\in I$ which act on~$Q$
via $s_i(\alpha_j)=\alpha_j-a_{ij}\alpha_i$. Given $w\in W$, denote $R(w)$ the set of reduced words for $w$, that is,
the set of $\mathbf i=(i_1,\dots,i_m)\in I^m$ of minimal length $m:=\ell(w)$ such that $w=s_{i_1}\cdots s_{i_m}$. It is well-known that the
form $(\cdot,\cdot)$ is $W$-invariant.

The following essentially coincides with Theorem~1.13 from~\cite{BG-dcb}.
\begin{lemma}
\begin{enumerate}[{\rm(a)}]
\item 
For each $i\in I$ there exists a unique automorphism~$T_i$ of $U_q(\lie g)$ which satisfies
$T_i(K_j)=K_jK_i^{-a_{ij}}$
and 
\begin{gather*}
T_i(E_j)=
\begin{cases} q_i^{-1} K_{i}^{-1} F_i,&i=j\\ 
\sum\limits_{r+s=-a_{ij}}\mskip-8mu (-1)^r q_i^{s+\frac12 a_{ij}} E_i^{\la r\ra}E_j E_i^{\la s\ra},  &i\ne j\\
\end{cases}
\\
T_i(F_j)=
\begin{cases} q_i^{-1} K_{i} E_i,& i=j\\ 
\sum\limits_{r+s=-a_{ij}}\mskip-8mu  (-1)^r q_i^{s+\frac12 a_{ij}} F_i^{\la r\ra}F_j F_i^{\la s\ra}, &i\ne j\\
\end{cases}
\end{gather*}
\item For all $x\in U_q(\lie g)$, $\overline{T_i(x)}=T_i(\overline x)$, 
$(T_i(x))^*=T_i^{-1}(x^*)$ and $(T_i(x))^t=T_i^{-1}(x^t)$.
\item The $T_i$, $i\in I$ satisfy the braid relations on $U_q(\lie g)$, that is, they 
define a representation of the Artin braid group $\operatorname{Br}_{\lie g}$ of~$\lie g$ on~$U_q(\lie g)$. 
\end{enumerate}
\label{lem:T_i-defn-prop}
\end{lemma}

\subsection{Bilinear forms}
Following~\cites{Lus-book,BG-dcb}, we define a symmetric bilinear form $\la\cdot,\cdot\ra$ on $U_q(\lie n^+)$.
Let $V=\bigoplus_{i\in I} \kk E_i$ and let $\la\cdot,\cdot\ra$ be the bilinear form on~$V$ defined by 
$\la E_i,E_j\ra=\delta_{ij}(q_i-q_i^{-1})$. Extend it naturally to $T(V)$ via 
$$
\la v_1\tensor\cdots\tensor v_k,v'_1\tensor\cdots\tensor v'_l\ra' =\delta_{kl}\prod_{r=1}^k \la v_r,v'_r\ra
,\qquad v_r,\,v'_r\in V,\,1\le r\le k.
$$
Define a linear map $\Psi:V\tensor V\to V\tensor V$ by $\Psi(E_i\tensor E_j)=q^{(\alpha_i,\alpha_j)}E_j\tensor E_i$.
Finally, define $\la\cdot,\cdot\ra_\Psi$ via 
$$
\la u,v\ra_\Psi=\delta_{k,l} \la [k]_\Psi!(u),v\ra'=\delta_{k,l}\la u,[k]_\Psi!(v)\ra',\qquad u\in V^{\tensor k},\, v\in V^{\tensor l}
$$
and $[k]_\Psi!\in\End_\kk V^{\tensor k}$ is the standard notation for the braided $k$-factorial (see e.g.~\cite{BG-dcb}*{\S A.1}).
It is well-known (see e.g.~\cite{Lus-book}) that the kernel~$J$ of the canonical map $T(V)\to U_q(\lie n^+)$, $E_i\mapsto E_i$
is the radical of $\la\cdot,\cdot\ra_\Psi$. Thus, we have a well-defined non-degenerate symmetric bilinear form $\la\cdot,\cdot\ra$ on~$U_q(\lie n^+)$
given by $\la u+J,v+J\ra=\la u,v\ra_\Psi$. This form in fact coincides with the form $\la\cdot,\cdot\ra$ we introduced in~\cite{BG-dcb}*{\S A.3} if 
we identify $U_q(\lie n^-)$ with $U_q(\lie n^+)$ via ${}^{*t}$.
We will often use the following obvious
\begin{lemma}\label{lem:form-non-zero-deg}
Let $x,x'\in U_q(\lie n^+)$ be homogeneous. Then $\la x,x'\ra\not=0$ implies that $\deg x=\deg x'$.
\end{lemma}
Define $\fgfrm{\cdot}{\cdot}:U_q(\lie n^+)\tensor U_q(\lie n^+)\to \kk$ %which is 
by 
$$
\fgfrm{x}{y}=
\mu(\gamma)q^{-\frac12(\gamma,\gamma)}
\la x,y\ra,\qquad x,y\in U_q(\lie n^+)_\gamma
$$
where
\begin{equation}\label{eq:defn-mu}
\mu(\gamma)=q^{\frac14(\gamma,\gamma)+\frac12\eta(\gamma)},\qquad \gamma\in Q
\end{equation}
and $\eta\in\Hom_\ZZ(Q,\ZZ)$ is defined by $\eta(\alpha_i)=d_i$. Note the following properties of~$\mu$ which will be often used in the sequel
\begin{equation}\label{eq:mu-prop}
\mu(r\alpha_i)=q_i^{\binom{r+1}2},\quad \mu(s_i\gamma)=\mu(\gamma)q^{-\frac12(\alpha_i,\gamma)},\quad \mu(\gamma+\gamma')=
\mu(\gamma)\mu(\gamma')q^{\frac12(\gamma,\gamma')}
\end{equation}

Define an anti-linear automorphism $\tilde\cdot$ of $U_q(\lie g)$ by
$$
\tilde x=(\sgn\gamma)\bar x{}^*,\qquad x\in U_q(\lie g)_\gamma
$$
where $\sgn:Q\to\{\pm 1\}$ is the homomorphism of abelian groups defined by $\sgn(\alpha_i)=-1$. Then
(cf.~\cite{BG-dcb}*{})
\begin{equation}\label{eq:bar tilde}
\overline{\fgfrm{x}{y}}=\fgfrm{\bar x}{\tilde y}=\fgfrm{\tilde x}{\bar y}.
\end{equation}

\subsection{Lattices  and signed basis in \texorpdfstring{$U_q(\lie n^+)$}{Uq(n+)}}\label{subs:defn-basis}
Let $\AA=\ZZ[q^{\frac12},q^{-\frac12}]$ which is a subring of~$\QQ(q^{\frac12})$. Denote $\AA_0=\ZZ[q,q^{-1}]$ and $\AA_1=q^{\frac12}\AA_0$; clearly,
$\AA=\AA_0\oplus \AA_1$ as an $\AA_0$-module.
Following~\cite{BG-dcb}*{\S3.1}, for any $J\subset I$, let $U_\ZZ(\lie n^+)_J$ (respectively, $U_\ZZ(\lie n^-)_J$) be the $\AA_0$-subalgebra of~$U_q(\lie n^+)$ 
(respectively, $U_q(\lie n^-)$)
generated by the 
$E_i^{\la n\ra}$ (respectively, $F_i^{\la n\ra}$), $i\in J$, $n\in\ZZ_{\ge0}$. We abbreviate $U_\ZZ(\lie n^\pm):=U_\ZZ(\lie n^+)_I$.
Set 
$$
U^\ZZ(\lie n^+)=\{ x\in U_q(\lie n^+)\,:\, \fgfrm{x}{U_\ZZ(\lie n^+)} \subset \AA_0\}.
$$
Clearly, $U^\ZZ(\lie n^+)$ is an $\AA_0$-submodule of~$U_q(\lie n^+)$.
\begin{lemma}\label{lem:prod in U^Z}
We have $q^{\frac12(\gamma,\gamma')}xx'\in U^\ZZ(\lie n^+)$
for all $x\in U^\ZZ(\lie n^+)_\gamma$, $x'\in U^\ZZ(\lie n^+)_{\gamma'}$. In particular,
all powers of a homogeneous element of~$U^\ZZ(\lie n^+)$ are in $U^\ZZ(\lie n^+)$ and
$U^\AA(\lie n^+):=U^\ZZ(\lie n^+)\tensor_{\AA_0}\AA$ is an $\AA$-algebra.
\end{lemma}
\begin{proof}
Following~\cite{Lus-book}*{\S1.2}, 
let $\ul\Delta:U_q(\lie n^+)\to U_q(\lie n^+)\ul\tensor U_q(\lie n^+)$ be the braided co-multiplication defined by $\ul\Delta(E_i)=E_i\tensor 1+1\tensor E_i$,
where $U_q(\lie n^+)\ul\tensor U_q(\lie n^+)=U_q(\lie n^+)\tensor U_q(\lie n^+)$ endowed with an algebra structure via $(x\tensor y)(x'\tensor y')=q^{(\gamma,\gamma')}xx'\tensor yy'$
for all $x,y'\in U_q(\lie n^+)$, $y\in U_q(\lie n^+)_\gamma$, $x'\in U_q(\lie n^+)_{\gamma'}$. 
Then $\la xx',y\ra=\la x,\ul y_{(1)}\ra\la x',\ul y_{(2)}\ra$ and so
$$
\fgfrm{xx'}{y}=q^{-\frac12(\gamma,\gamma')}\fgfrm{x}{\ul y_{(1)}}\fgfrm{x'}{\ul y_{(2)}},\qquad x\in U_q(\lie n^+)_\gamma,\, x'\in U_q(\lie n^+)_{\gamma'},
$$
where $\ul\Delta(y)=\ul y_{(1)}\tensor \ul y_{(2)}$ in Sweedler-like notation. It follows from~\cite{Lus-book}*{Lemma~1.4.2} that 
$\ul\Delta(U_\ZZ(\lie n^+))\subset U_\ZZ(\lie n^+)\tensor_{\AA_0}
U_\ZZ(\lie n^+)$, hence we can assume that $\ul y_{(1)},\ul y_{(2)}\in U_\ZZ(\lie n^+)$ provided that $y\in U_\ZZ(\lie n^+)$. All assertions are now immediate.
\end{proof}

Define, for any $\gamma\in Q^+$ 
\begin{equation}\label{eq:def-bas}
\mathbf B^{\pm up}{}_\gamma=\{ b\in U^\ZZ(\lie n^+)_\gamma\,:\,\bar b=b,\, \mu(\gamma)^{-1}\fgfrm{b}{b}\in 1+q^{-1}\ZZ[[q^{-1}]])\}
\end{equation}
and set $\mathbf B^{\pm up}=\bigsqcup_{\gamma\in Q^+} \mathbf B^{\pm up}{}_\gamma$.

\subsection{\texorpdfstring{Signed basis is $(K_-,\mu)$}{(K-,mu)}-orthonormal}\label{subs:char sign bas}
Let $R$ be a commutative unital subring of a field~$\kk$.
Following~\cite{Lus-book}*{\S14.2.1}, a subset $\mathbf B^\pm$ of a free $R$-module~$L$ is a {\em signed basis} of~$L$ 
if $\mathbf B^\pm=\mathbf B\sqcup (-\mathbf B)$
for some basis $\mathbf B$ of~$L$. 

Let $K_-$
be a subring of~$\kk$ not containing~$1$.
We say that $\mathbf B$ is  $(K_-,\mu)$-orthonormal for some $\mu\in R^\times$ with respect to a fixed symmetric bilinear pairing 
$\fgfrm{\cdot}{\cdot}:L\tensor_R L\to \kk$
if
$$
\mu \cdot\fgfrm{b}{b'}\in \delta_{b,b'}+K_-,\qquad b,b'\in\mathbf B.
$$
Accordingly, we say that a signed basis $\mathbf B^\pm$ is $(K_-,\mu)$-orthonormal if it contains a $(K_-,\mu)$-orthonormal basis of~$L$.

The following result is parallel to~\cite{Lus-book}*{Theorem 14.2.3}.
\begin{theorem}\label{thm:sign-bas}
$\mathbf B^{\pm up}$ is a signed basis of~$U^\ZZ(\lie n^+)$ with $R=\AA_0$. 
Moreover, for each $\gamma\in Q^+$, $\mathbf B^{\pm up}{}_\gamma$ is 
a $(K_-,\mu(\gamma)^{-1})$-orthonormal basis where~$\mu$ is defined by~\eqref{eq:defn-mu}
and $K_-=q^{-1}\ZZ[[q^{-1}]]\cap \QQ(q)$.
\end{theorem}
\begin{proof}
We need the following general setup. 

We say that a domain $R_0$ is {\em strongly integral} if a sum of squares of its non-zero elements is never zero
and if $c_1^2+\cdots+c_n^2=1$, $c_r\in R_0$, implies that for all $1\le i\le n$, $c_i=\pm\delta_{ij}$ for some~$1\le j\le n$.

Let $R$ be a domain with a subdomain $R_0$.
Given a totally ordered additive monoid~$\Gamma$, a map $\nu:R\to \Gamma\sqcup\{-\infty\}$ is called an {\em $R_0$-linear
valuation} if the following hold for all $f,g\in R$
\begin{enumerate}[$(V_1)$]
\item\label{val-prop.1} $\nu(f)=-\infty$ if and only 
if $f=0$
\item\label{val-prop.2} $\nu(R_0\setminus\{0\})=0$, 
\item\label{val-prop.3} $\nu(fg)=\nu(f)+\nu(g)$; 
\item\label{val-prop.4} $\nu(f+g)\le \max(\nu(f),\nu(g))$. 
\end{enumerate}
It follows that 
\begin{equation}\label{eq:diff-valuations}
\nu(f)\not=\nu(g)\implies\nu(f+g)=\max(\nu(f),\nu(g)).
\end{equation}
Furthermore, for each $a\in\Gamma$, set $R_{\le a}=\{ r\in R\,:\, \nu(r)\le a\}$ and $R_{<a}=\{ r\in R\,:\, \nu(r)< a\}$.
Clearly, $R_{\le a}$ and $R_{<a}$ are $R_0$-submodules of $R$ and $R_{<a}\subset R_{\le a}$. In the spirit of~\cite{KaKh}*{\S2.1},
we call the $R_0$-module $R_{\le a}/R_{<a}$ the {\em leaf of $\nu$ at~$a$}; we say that $\nu$ {\em has one-dimensional 
leaves} if for each $a\in\nu(R)$, the leaf of~$\nu$ at~$a$ is a non-zero cyclic $R_0$-module.

Let $M$ be a free $R$-module with a basis~$\mathbf B$. Then we can define $\nu_{\mathbf B}:M\to \Gamma\cup\{-\infty\}$ by 
\begin{equation}\label{eq:pre-val}
\nu_{\mathbf B}(\sum_{b\in\mathbf B}c_b b)=\max_b
\nu(c_b).
\end{equation}
Clearly~\prpref{val-prop.4}{V} holds and we also have $\nu_{\mathbf B}(f x)=\nu(f)+\nu_\mathbf B(x)$, $f\in R$, $x\in M$. We will need the following Lemma.
\begin{lemma}\label{lem:euclidean}
Suppose that $\nu:R\to\Gamma\cup\{-\infty\}$ has one-dimensional leaves and let $M$ be a free $R$ module with a basis~$\mathbf B$. Then every $x\in M$ with 
$\nu_{\mathbf B}(x)>0$ can 
be written as $x=f x_0+x_1$ where $f\in R$ with $\nu(f)=\nu(x)$, $0\not=x_0\in \sum_{b\in \mathbf B} R_0 b$ and $x_1\in M$ satisfies
$\nu_{\mathbf B}(x_1)<\nu_{\mathbf B}(x)$.
\end{lemma}
\begin{proof}
Let $x\in M$ with $a=\nu_{\mathbf B}(x)>0$ and write 
$$
x=\sum_{b\in\mathbf B} x_b b=\sum_{b\in\mathbf B\,:\, \nu(x_b)=a} x_b b+\sum_{b\in\mathbf B\,:\, \nu(x_b)<a}x_b b.
$$
Since $\nu$ has one-dimensional leaves and $R_{\le a}\not=R_{<a}$, 
$R_{\le a}/R_{<a}$ is a non-zero cyclic $R_0$-module. Let $f\in R_{\le a}$ be
any element whose image generates~$R_{\le a}/R_{<a}$ as an $R_0$-module. 
Then 
$\nu(f)=a$ and for every $b\in\mathbf B$ with $\nu(x_b)=a$ there exists $r_b\in R_0$ such that $\nu(x_b-r_b f)<a$. Set
$$
x_0=\sum_{b\in \mathbf B\,:\,\nu(x_b)=a} r_b b,\qquad x_1=x-f x_0.
$$
Clearly, $\nu_{\mathbf B}(x_1)<a$, whence $x_0\not=0$.
\end{proof}

Henceforth 
\begin{itemize}
 \item $R_0$ is a strongly integral domain
 \item $\kk$ is a field containing $R_0$;
 \item $R_0\subset R\subset \kk$ as subrings
 \item $\nu:\kk\to \Gamma\cup\{-\infty\}$ is an $R_0$-linear valuation;
 \item $K_-$ is an $R_0$-subalgebra of $\kk$ such that $\nu(f)<0$ for all $f\in K_-$ (note that this implies that $K_-\cap R_0=\emptyset$)
 and $(1+K_-)^{-1}\subset 1+K_-$.
 \item There is a field involution $\bar\cdot$ of~$\kk$ which restricts to~$R$ and is identity on~$R_0$, while $\overline{K_-}\cap K_-=\emptyset$
 \item $\nu(R^{\bar\cdot}\setminus R_0)>0$ where $R^{\bar\cdot}=\{f\in R\,:\, \bar f=f\}$;
 \item The restriction of $\nu$ to $R^{\bar\cdot}$ is a valuation $\nu:R^{\bar\cdot}\to\Gamma\sqcup\{-\infty\}$ with one-dimensional leaves;
\end{itemize} 

For an $R$-module~$L$, an endomorphism of $\ZZ$-modules $\varphi:L\to L$ is called anti-linear if for all $r\in R$, $x\in L$ 
we have $\overline{r\cdot x}=\overline r\cdot \overline x$. Anti-linear endomorphisms of a $\kk$-vector space~$V$ are defined similarly.

Let $V$ be a $\kk$-vector space with a non-degenerate symmetric bilinear form $\fgfrm\cdot\cdot$. Suppose that 
$\varphi$, $\varphi'$ are anti-linear involutions on~$V$ satisfying $\overline{\fgfrm xy}=\fgfrm{\varphi(x)}{\varphi'(y)}$, $x,y\in V$.
Let $L$ be a free $R$-module such that $V=\kk\tensor_R L$. 
Denote $L^\vee=\{ x\in V\,:\, \fgfrm{x}{L}\subset R\}$. Clearly, $L^\vee$ is a free $R$-module and $V=\kk\tensor_R L^\vee$.

Given $\mu\in R^\times$ define 
\begin{gather*}
\mathbf B^\pm(\mu)=\{ b\in L\,:\,\varphi'(b)=b,\, \mu\cdot \fgfrm{b}{b}\in 1+K_-\}\\
\mathbf B^\vee_\pm(\mu)=\{ b\in L^\vee\,:\,\varphi(b)=b,\, \mu\cdot \fgfrm{b}{b}\in 1+K_-\}.
\end{gather*}
\begin{proposition}\label{prop:gen-dual-basis}
Suppose that $\dim_\kk V<\infty$.
The following are equivalent
\begin{enumerate}[{\rm(a)}]
 \item\label{prop:gen-dual-basis.a} $\mathbf B^\pm(\mu)$ is a $(K_-,\mu)$-orthonormal signed $R$-basis of~$L$. 
 \item\label{prop:gen-dual-basis.b} $\mathbf B^\vee_\pm(\mu^{-1})$ is a $(K_-,\mu^{-1})$-orthonormal signed $R$-basis of~$L^\vee$.
 \end{enumerate}
In that case, $\mathbf B^\pm(\mu)$ and $\mathbf B^\vee_\pm(\mu)$ are dual to each other with respect to $\fgfrm{\cdot}{\cdot}$.
\end{proposition}
\begin{proof}
\eqref{prop:gen-dual-basis.a}$\implies$\eqref{prop:gen-dual-basis.b}
Let $\ul{\mathbf B}^\pm(\mu)$ be any basis of~$L$ contained in~$\mathbf B^\pm(\mu)$.

Since $\fgfrm{\cdot}{\cdot}$ is non-degenerate, for each $b\in \ul{\mathbf B}^\pm(\mu)$ there exists a unique $\delta_b\in L^\vee$ such that 
$\fgfrm{\delta_b}{b'}=\delta_{b,b'}$. Clearly, the set $\ul{\mathbf B}^\pm(\mu)^\vee:=\{ \delta_b\,:\, b\in \ul{\mathbf B}^\pm(\mu)\}$ is a basis of~$L^\vee$.
Note that $\varphi(\delta_b)=\delta_b$. 
\begin{lemma}\label{lem:scal-quad-delta func}
The set $\ul{\mathbf B}^\pm(\mu)^\vee$ 
is $(K_-,\mu^{-1})$-orthonormal basis of $L^\vee$.
In particular, $\nu(\mu^{-1}\fgfrm{\delta_b}{\delta_{b'}})\le 0$ with the equality if and only if~$b=b'$.
\end{lemma}
\begin{proof}
We need the following
\begin{lemma}
Let $G=(G_{r,s})_{1\le r,s\le n}$ be a matrix over $\kk$ such that $\mu G_{rs}\in\delta_{rs}+K_-$. Then $G$ is invertible 
and $M=(M_{rs})_{1\le r,s\le n}=G^{-1}$ satisfies $\mu^{-1} M_{rs}\in\delta_{rs}+K_-$.
\end{lemma}
\begin{proof}
Let $\Delta_{r,s}(G)$ be the minor of~$G$ obtained by removing the $r$th row and the $s$th column. Then it is easy to see 
that $\mu^{n-1}\Delta_{r,s}(G)\in \delta_{rs}+K_-$. Similarly, $\mu^n\det G\in 1+K_-$ hence $G$ is invertible. 
Moreover, $\mu^{-n}(\det G)^{-1}\in 1+K_-$. Since $M_{rs}=(-1)^{r+s}(\det G)^{-1}\Delta_{s,r}(G)$, the assertion follows.
\end{proof}

Since $\ul{\mathbf B}^\pm(\mu)$ is $(K_-,\mu)$-orthogonal, the above Lemma applies to 
the (finite) Gram matrix $G=(\fgfrm{b}{b'})_{b,b'\in \ul{\mathbf B}^\pm(\mu)}$ of $\fgfrm{\cdot}{\cdot}$ with respect to the basis $\ul{\mathbf B}^\pm(\mu)$ and hence 
$\mu^{-1} M_{b,b'}\in\delta_{b,b'}+K_-$ where $M=G^{-1}$. 
Since $\mu^{-1}\delta_b=\sum_{b'\in\ul{\mathbf B}^\pm(\mu)} \mu^{-1} M_{b,b'}b'$, 
we have $\mu^{-1} \delta_b\in b+K_-\cdot \ul{\mathbf B}^\pm(\mu)$. This implies that %for all $b,b'\in \ul{\mathbf B}^\pm(\mu)$ one has
\begin{equation*}
\mu^{-1} \fgfrm{\delta_b}{\delta_{b'}}=\fgfrm{\mu^{-1}\delta_b}{\delta_{b'}}\in \fgfrm{b}{\delta_{b'}}+K_-\fgfrm{\ul{\mathbf B}^\pm(\mu)}{\delta_{b'}}
=\delta_{b,b'}+K_-,\qquad b,b'\in \ul{\mathbf B}^\pm(\mu)
\end{equation*}
This proves Lemma~\ref{lem:scal-quad-delta func}.
\end{proof}

Note that for any $x=\sum_{b} x_b \delta_b$, $y=\sum_{b'} y_{b'}\delta_{b'}$ in~$L^\vee$ we have 
$$
\mu^{-1}\fgfrm{x}{y}=\sum_{b,b'} x_b y_{b'}\mu^{-1}\fgfrm{\delta_b}{\delta_{b'}}.
$$
Define $\nu=\nu_{\ul{\mathbf B}^\pm(\mu)^\vee}:L^\vee\to \Gamma\sqcup\{-\infty\}$ as in~\eqref{eq:pre-val}. 
Since $\nu(\mu^{-1}\fgfrm{\delta_b}{\delta_{b'}})\le 0$
for all $b,b'\in \ul{\mathbf B}^\pm(\mu)$ by Lemma~\ref{lem:scal-quad-delta func}, it follows from \prpref{val-prop.3}{V} and~\prpref{val-prop.4}{V}
that 
\begin{equation}\label{eq:prod-ineq}
\nu(\mu^{-1}\fgfrm{x}{y})\le \nu(x)+\nu(y),\qquad x,y\in L^\vee.
\end{equation}

Clearly, for $x\in L^\vee$ we have 
\begin{equation}\label{eq:invar-form}
\varphi(x)=x\iff x\in \sum_{b\in\ul{\mathbf B}^\pm(\mu)} R^{\bar\cdot}\delta_b.
\end{equation}
Thus, the set $(L^\vee)^\varphi$ of $\varphi$-invariant elements in~$L^\vee$ is a free $R^{\bar\cdot}$-module with a basis $\ul{\mathbf B}^\pm(\mu)^\vee$.

The following Lemma is the crucial point of our argument.
\begin{lemma}\label{lem:key Lusztig's lemma}
Let $x\in L^\vee$ and suppose that $\varphi(x)=x$. 
Then
\begin{enumerate}[{\rm(a)}]
\item\label{lem:key Lusztig's lemma.c} If $\nu(x)=0$, that is, $x=\sum_{b} x_b \delta_b$ with $x_b\in R_0$, then $\mu^{-1}\fgfrm{x}{x}-\sum_{b} x_b^2\in K_-$
and $\nu(\mu^{-1}\fgfrm{x}{x})=0$.

\item\label{lem:key Lusztig's lemma.b} If $\nu(x)>0$ then $\nu(\mu^{-1}\fgfrm{x}{x})>0$
\end{enumerate}
\end{lemma}
\begin{proof}
Write $x=\sum_{b} x_b \delta_b$, $x_b\in R^{\bar\cdot}$.

To prove~\eqref{lem:key Lusztig's lemma.c}, note that $\nu(x)=0$ and $\varphi(x)=x$ imply that $x_b\in R_0$ for all $b\in\ul{\mathbf B}^\pm(\mu)$. 
We have 
$$
\mu^{-1}\fgfrm{x}{x}=\sum_{b} x_b^2 \mu^{-1}\fgfrm{\delta_b}{\delta_b}+\sum_{b\not=b'} x_bx_{b'}\mu^{-1}\fgfrm{\delta_b}{\delta_{b'}}.
$$
By Lemma~\ref{lem:scal-quad-delta func}, the first sum belongs to $\sum_{b}x_b^2+K_-$ while the second sum 
belongs to $K_-$. Since $R_0$ is strongly integral, $\sum_{b} x_b^2\not=0$. Thus,
$\nu(\mu^{-1}\fgfrm{x}{x})=0$.

To prove~\eqref{lem:key Lusztig's lemma.b},
let $a=\nu(x)>0$. Applying Lemma~\ref{lem:euclidean} to $M=(L^\vee)^{\varphi}$ and the ring~$R^{\bar\cdot}$, we can write 
$x=f x_0+x_1$ where $f\in R^{\bar\cdot}$, $\nu(f)=a$, $\varphi(x_0)=x_0$ (and so $\varphi(x_1)=x_1$), $\nu(x_0)=0$ and $\nu(x_1)<a$.
Then 
$$
\nu(\mu^{-1}\fgfrm{x}{x})=\nu(f^2 \mu^{-1}\fgfrm{x_0}{x_0}+2 f\mu^{-1}\fgfrm{x_0}{x_1}+\mu^{-1}\fgfrm{x_1}{x_1})
=\nu(f^2\mu^{-1}\fgfrm{x_0}{x_0})=2a>0
$$
since $\nu(\mu^{-1}\fgfrm{x_1}{x_1}),\nu(f\mu^{-1}\fgfrm{x_0}{x_1})<2 a$ by~\eqref{eq:prod-ineq} and \prpref{val-prop.3}{V}, \prpref{val-prop.4}{V}
while $\nu(\mu^{-1}\fgfrm{x_0}{x_0})=0$ by part~\eqref{lem:key Lusztig's lemma.c}.
This proves~\eqref{lem:key Lusztig's lemma.b}.
\end{proof}
It follows from Lemma~\ref{lem:key Lusztig's lemma}(\ref{lem:key Lusztig's lemma.c},\ref{lem:key Lusztig's lemma.b}) that if $x\in L^\vee$ 
is fixed by~$\varphi$
and $\mu^{-1}\fgfrm{x}{x}\in 1+K_-$ then $x=\pm\delta_b$ for some $b\in \ul{\mathbf B}^\pm(\mu)$ by the strong integrality of~$R_0$. Thus, 
$\mathbf B^\vee_\pm(\mu)=\ul{\mathbf B}^\pm(\mu)^\vee\bigsqcup(-\ul{\mathbf B}^\pm(\mu)^\vee)$. This completes the proof of the implication 
\eqref{prop:gen-dual-basis.a}$\implies$\eqref{prop:gen-dual-basis.b} and the last assertion.  
The opposite implication follows by the symmetry between 
$L$ and $L^\vee$ and $\varphi$ and $\varphi'$.
\end{proof}

We now apply Proposition~\ref{prop:gen-dual-basis} with $L=U_\ZZ(\lie n^+)_\gamma$, $v=q^{\frac12}$, $\varphi=\bar\cdot$, $\varphi'=\tilde\cdot$,
$R=\AA_0=\ZZ[v^2,v^{-2}]$, $\kk=\QQ(v)$
and $K_-=v^{-2}\ZZ[[v^{-2}]]\cap \QQ(v)$. We define $\nu:\QQ(v)\to\ZZ\cup\{-\infty\}$ via 
$$
\nu\Big(c v^n\frac{1+f}{1+g}\Big)=n
$$
where $c\in\QQ^\times$, $n\in\ZZ$ and $f,v\in v^{-1}\ZZ[v^{-1}]$. Note  that $R^{\bar\cdot}=\ZZ[q+q^{-1}]$ and $\nu$ has one-dimensional leaves on~$R^{\bar\cdot}$
since $\nu((v+v^{-1})^n)=n$.
By \cite{Lus-book}*{Theorem~14.2.3}, $\mathbf B^{\pm can}\cap 
U_\ZZ(\lie n^+)$ is a $(K_-,\mu(\gamma))$-orthonormal signed basis of $L$. Since $\mathbf B^{\pm up}{}_\gamma=(\mathbf B^{\can}\cap 
U_\ZZ(\lie n^+))^\vee_\pm$ in the notation of Proposition~\ref{prop:gen-dual-basis}, it is a signed $(K_-,\mu(\gamma)^{-1})$-orthonormal basis 
of $L^\vee=U^\ZZ(\lie n^+)_\gamma$. This completes the proof of Theorem~\ref{thm:sign-bas}.
\end{proof}

\subsection{Choosing~\texorpdfstring{$\mathbf B^{up}$}{Bup} inside the signed basis}\label{subs:chos basis}
It remains to describe a canonical way to choose $\mathbf B^{up}$ inside~$\mathbf B^{\pm up}$. Needless to say, it can be taken as the dual basis 
of $\mathbf B^{can}$ with respect to $\fgfrm{\cdot}{\cdot}$. However, it more instructive to provide an intrinsic definition. 

To that effect, following~\cite{BG-dcb}*{\S3.5} and also~\cite{Lus-book}*{Proposition~3.1.6}, define $\kk$-linear endomorphisms $\partial_i,\partial_i^{op}$, $i\in I$ of~$U_q(\lie n^+)$ by 
\begin{equation}\label{eq:partial-def}
[F_i,x ]=(q_i-q_i^{-1})(
q^{-\frac12(\alpha_i,\gamma-\alpha_i)} K_i\partial_i(x )-q^{\frac12(\alpha_i,\gamma-\alpha_i)} K_i^{-1}\partial_i^{op}(x )),\qquad x \in U_q(\lie n^+)_\gamma.
\end{equation}
We need the following properties of these operators (cf.~\cite{BG-dcb}*{Lemmata~3.18 and~3.20}).
\begin{lemma}\label{lem:partial-prop}
For all $x \in U_q(\lie n^+)_\gamma$ and $i\in I$ we have 
\begin{enumerate}[{\rm(a)}]
 \item \label{lem:partial-prop.i}
$\overline{\partial_i(x )}=\partial_i(\overline{x })$,
$\overline{\partial_i^{op}(x )}=\partial_i^{op}(\overline{ x })$, $\partial_i(x{}^*)^*=\partial_i^{op}(x)$
and $\partial_i\partial_i^{op}(x)=\partial_i^{op}\partial_i(x)$.
\item \label{lem:partial-prop.ii}
for all $y\in U_q(\lie n^+)$, $n\in\ZZ_{\ge 0}$
$$
\fgfrm{x}{yE_i^{\la n\ra}}=\fgfrm{\partial_i^{(n)}(x)}{y},\qquad 
\fgfrm{x}{E_i^{\la n\ra}y}=\fgfrm{(\partial_i^{op})^{(n)}(x )}{y},
$$
where $f_i^{(n)}=(q_i-q_i^{-1})^n f_i^{\la n\ra}$.
\item  \label{lem:partial-prop.iii}
$\partial_i$, $\partial_i^{op}$ are quasi-derivations. Namely, for $x\in U_q(\lie n^+)_\gamma$, $y \in U_q(\lie n^+)_{\gamma'}$ we have
\begin{equation}\label{eq:partial-inv}
\begin{gathered}
\partial_i(x y )=q^{\frac12(\alpha_i,\gamma' )}\partial_i(x )y  +
q^{-\frac12(\alpha_i,\gamma )}x \partial_i(y ),
\\
\partial_i^{op}(x y )=q^{-\frac12(\alpha_i,\gamma' )}\partial_i^{op}(x )y +q^{\frac12(\alpha_i,\gamma )}x  \partial_i^{op}(y ).
\end{gathered}
\end{equation}

\end{enumerate}
\end{lemma}
It is easy to see that
\begin{equation}\label{eq:partial E_i^r}
\partial_i^{(n)}(E_i^r)=\binom{r}{n}_{q_i} E_i^{r-n}=(\partial_i^{op})^{(n)}(E_i^r)
\end{equation}
whence 
\begin{equation}\label{eq:q-der preserve Lusztig lattice}
((q_i-q_i^{-1})\partial_i)^n(E_i^{\la r\ra})=E_i^{\la r-n\ra}=((q_i-q_i^{-1})\partial_i^{op})^n(E_i^{\la r\ra})
\end{equation}
The following is an immediate consequence of this identity and Lemma~\ref{lem:partial-prop}.
\begin{corollary}\label{cor:preserv-lattice}
For all $i\in I$ we have 
\begin{enumerate}[\rm(a)]
 \item\label{cor:preserv-lattice.a}
 if $x\in U_\ZZ(\lie n^+)_\gamma$ then $\la 1\ra_{q_i}\partial_i(x),\la 1\ra_{q_i}\partial_i^{op}(x)\in q^{\frac12(\alpha_i,\gamma)}U_\ZZ(\lie n^+)$;
 \item\label{cor:preserve-lattice.b}
 the
$\partial_i^{(n)}$, $(\partial_i^{op})^{(n)}$, $n\in \ZZ_{\ge0}$
restrict to operators on $U^\ZZ(\lie n^+)$. 
\end{enumerate}
\end{corollary}
By degree considerations it is clear that $\partial_i$, $\partial_i^{op}$ are locally nilpotent, that is,
for any $x\in U_q(\lie n^+)$ we have $\partial_i^k(x)=(\partial_i^{op})^k(x)=0$ for $k\gg0$. Thus, for each $x\in U_q(\lie n^+)\setminus\{0\}$
we can define $\ell_i(x)$ as the maximal $k> 0$ such that $\partial_i^k(x)\not=0$. 
Define $\partial_i^{(top)},(\partial_i^{op})^{(top)}:U_q(\lie n^+)\setminus\{0\}\to U_q(\lie n^+)\setminus\{0\}$ by $\partial_i^{(top)}(x)=\partial_i^{(\ell_i(x))}(x)$
and $(\partial_i^{op})^{(top)}(x)=(\partial_i^{(top)}(x^*))^*=(\partial_i^{op})^{(\ell_i(x^*))}(x)$.
Similar notation will be used for other locally nilpotent operators in the sequel.

For any sequence $\mathbf i=(i_1,\dots,i_m)\in I^m$ set $\partial_{\mathbf i}^{(top)}=\partial_{i_m}^{(top)}\cdots \partial_{i_1}^{(top)}$.
\begin{proposition}\label{prop:string char}
For every $b\in \mathbf B^{\pm up}$ there exists $\ii=(i_1,\dots,i_m)$ such that $\partial_\ii^{(top)}(b)\in\{\pm 1\}$.
Moreover, if $\ii'=(i'_1,\dots,i'_{m'})$ also satisfies $\partial_{\ii'}^{(top)}(b)\in\{\pm 1\}$ then 
$\partial_\ii^{(top)}(b)=\partial_{\ii'}^{(top)}(b)\in\{\pm 1\}$.
\end{proposition}

Thus, we can define $\mathbf B^{up}$ to be the set of all $b\in \mathbf B^{\pm up}$ such that $\partial_\ii^{(top)}(b)=1$ for some $\ii=(i_1,\dots,i_m)$.

\begin{proof}
By Proposition~\ref{prop:gen-dual-basis}, $\mathbf B^{\pm up}$ contains the dual basis $\mathbf B'$ of~$\mathbf B^{\can}$. Our goal
is to prove that $\mathbf B^{up}=\mathbf B'$.
We need the following result.
\begin{lemma}\label{lem:tops}
$\partial_i^{(top)}(b)\in\mathbf B'$ for all $b\in \mathbf B'$, $i\in I$. Moreover, if $\partial_i^{(top)}(b)=\partial_i^{(top)}(b')$ and 
$\ell_i(b)=\ell_i(b')$ for some $b'\in\mathbf B'$ then $b=b'$.
\end{lemma}
\begin{proof}
Following~\cite{Lus-book}*{\S14.3}, denote $\mathbf B^{\can}_{i;\ge r}=\mathbf B^{\can}\cap E_i^r U_q(\lie n^+)$ and $\mathbf B^{\can}_{i;r}=\mathbf B^{\can}_{i;\ge r}\setminus
\mathbf B^{\can}_{i;\ge r+1}$. It follows from~\cite{Lus-book}*{\S14.3} that for all~$i\in I$, 
\begin{equation}\label{eq:Lus-decomp}
\mathbf B^{\can}=\bigsqcup_{r\ge 0} \mathbf B^{\can}_{i;r}.
\end{equation}
Let $b\in\mathbf B^{\can}$ and let $n=\ell_i(\delta_b)$, $u=\partial_i^{(top)}(\delta_b)=
\partial_i^{(n)}(\delta_b)$, where $\delta_b$ is the element of~$\mathbf B'$ satisfying 
$\fgfrm{\delta_b}{b'}=\delta_{b,b'}$. Then $u\in\ker \partial_i$ which, by
Lemma~\ref{lem:partial-prop}\eqref{lem:partial-prop.iii}, is orthogonal to $\mathbf B^{\can}_{i;s}$, $s>0$. Thus, we can write 
$$
u=\sum_{b'\in \mathbf B^{\can}_{i;0}} \fgfrm{u}{b'} \delta_{b'}=\sum_{b'\in \mathbf B^{\can}_{i;0}} \fgfrm{\delta_b}{E_i^{\la n\ra}b'}\delta_{b'}.
$$
By~\cite{Lus-book}*{Theorem~14.3.2}, for each $b'\in \mathbf B^{\can}_{i;0}$ there exists a unique $\pi_{i,n}(b')\in\mathbf B^{\can}_{i;n}$
such that 
$E_i^{\la n\ra} b'-\pi_{i;n}(b')\in\sum_{r>n} \ZZ[q,q^{-1}]\mathbf B^{\can}_{i;r}$. Using Lemma~\ref{lem:partial-prop}\eqref{lem:partial-prop.iii} again, 
we conclude that for 
any $b''\in\mathbf B^{\can}_{i;r}$ with~$r>n$, 
$\fgfrm{ \delta_b}{b''}\in \fgfrm{\delta_b}{E_i^{\la r\ra}U_q(\lie n^+)}=\fgfrm{\partial_i^{(r)}(\delta_b)}{U_q(\lie n^+)}=0$. Thus,
$$
u=\sum_{b'\in \mathbf B^{\can}_{i;0}} \fgfrm{\delta_b}{\pi_{i;n}(b')}\delta_{b'}.
$$
Note that, since~$u\not=0$, we cannot have $\fgfrm{\delta_b}{\pi_{i;n}(b')}=0$ for all~$b'\in\mathbf B^{\can}_{i;0}$.
Since~$\fgfrm{\delta_b}{b''}=\delta_{b,b''}$, we conclude that there exists a unique $b'\in\mathbf B^{\can}_{i;0}$ such that 
$\pi_{i;n}(b')=b$ and then $u=\partial_i^{(top)}(\delta_b)=\delta_{b'}$. Since $\pi_{i;n}:\mathbf B^{\can}_{i;0}\to\mathbf B^{\can}_{i;n}$
is a bijection by~\cite{Lus-book}*{Theorem~14.3.2}, the first assertion follows. The second assertion is immediate from~\eqref{eq:Lus-decomp}.
\end{proof}
This implies that for every element $b\in \mathbf B'$, there exists $\ii=(i_1,\dots,i_m)$ such that $\partial_\ii^{(top)}(b)=1$.
Since $1$ is the unique element of $\mathbf B'$ of degree~$0$, for any sequence $\ii'$ such that $\partial_{\ii'}^{(top)}(b)=c\in\kk^\times$,
one has $c=1$. This completes the proof of Proposition~\ref{prop:string char}.
\end{proof}
\begin{remark}\label{rem: star op partial}
Since $\mathbf B^{can}$ is preserved by~${}^*$ by~\cite{Lus-book}*{Theorem~14.4.3} and ${}^*$ is self-adjoint with respect 
to $\fgfrm{\cdot}{\cdot}$, it follows that $\mathbf B^{up}$ is preserved by~${}^*$. In particular,
we can replace $\partial_i$ by $\partial_i^{op}$ in Lemma~\ref{lem:tops} and Proposition~\ref{prop:string char}. 
\end{remark}

Note that Lemma~\ref{lem:tops} and Remark~\ref{rem: star op partial} immediately yield the following well-known fact.
\begin{corollary}\label{cor:top-decomp-B up}
Let $x\in U_q(\lie n^+)$ and write $x=\sum_{b\in\mathbf B^{up}} c_b(x) b$. Then $c_b(x)\not=0$ implies that $\ell_i(b)\le \ell_i(x)$ 
and $\ell_i(b^*)\le \ell_i(x^*)$
and 
$$
\partial_i^{(top)}(x)=\sum_{b\in\mathbf B^{up}\,:\,\ell_i(b)=\ell_i(x)} c_b(x) \partial_i^{(top)}(b),
\quad 
(\partial_i^{op})^{(top)}(x)=\sum\limits_{b\in\mathbf B^{up}\,:\,\ell_i(b^*)=\ell_i(x^*)} c_b(x) (\partial_i^{op})^{(top)}(b)
$$
are the decompositions of~$(\partial_i)^{(top)}(x)$ and $(\partial_i^{op})^{(top)}(x)$, respectively,
in the basis~$\mathbf B^{up}$.
\end{corollary}

\section{Decorated algebras and proof of Theorem~\ref{thm:T_i bas}}

\subsection{Decorated algebras}\label{subs:decorated alg}
Let $\mathcal A$ be an associative $\ZZ$-graded algebra over~$\kk=\mathbb Q(v^{\frac12})$. 
Denote the degree of a homogeneous $u\in\mathcal A$ by~$|u|$.

\begin{definition}\label{defn:decorated alg}
We say that $\mathcal A=\mathcal A(E,\ul F_+,\ul F_-)$ is {\em decorated} if it contains an element 
$E$ with $|E|=2$ and 
admits mutually commuting locally nilpotent $\kk$-linear endomorphisms $\ul F_+,\ul F_-:\mathcal A\to \mathcal A$ of degree~$-2$
satisfying $\ul F_\pm(E)=1$ and
\begin{equation}\label{eq:skew Leibnitz rule}
\ul F_\pm(xy)=v^{\pm\frac12 |y|}\ul F_\pm(x)y+v^{\mp\frac12|x|}x\ul F_\pm(y)
\end{equation}
for $x,y\in\mathcal A$ homogeneous. 
\end{definition}

Denote $\mathcal A_\pm=\ker D_\mp$ and set $\mathcal A_0=\mathcal A_+\cap\mathcal A_-$.
Clearly, $\ul F_\pm$ restricts to an endomorphism of $\mathcal A_\pm$ which will also be denoted by $\ul F_\pm$.
Since $F_\pm$ are skew derivations, $\mathcal A_\pm$ are subalgebras of~$\mathcal A$.

The following is a basic example of a decorated algebra. Let $\mathcal F_{m,n}=\kk\la E,x,y\ra$ with the $\ZZ$-grading defined by $|x|=-m$, $|y|=-n$, $|E|=2$.
The following is immediate 
\begin{lemma}\label{lem:free subalgebra}
There exists unique operators $\ul F_\pm\in\End_\kk \mathcal F_{m,n}$ such that $\ul F_\pm(E)=1$, $\ul F_\pm(x)=\ul F_\pm(y)=0$
and~\eqref{eq:skew Leibnitz rule} holds. In particular, $\mathcal F_{m,n}$ is a decorated algebra
and for any decorated algebra $\mathcal A$ and any $x',y'\in\mathcal A_0$ homogeneous
the natural homomorphism of graded associative algebras $\phi_{|x'|,|y'|}:\mathcal F_{|x'|,|y'|}\to \mathcal A$,
$x\mapsto x'$, $y\mapsto y'$
is a homomorphism 
of decorated algebras, that is, it commutes with $\ul F_\pm$.
\end{lemma}

Define $\ul K^{\frac12},\ul E_\pm\in\End_\kk \mathcal A$ by 
$$
\ul K^{\frac12}(x)=v^{\frac12|x|} u,\quad 
\ul E_\pm(x)=\pm\la 1\ra_{v}^{-1} (v^{\pm\frac12 |x|}Ex-v^{\mp\frac12|x|}xE),
$$
for $x\in\mathcal A$ homogeneous. Clearly $\ul E_\pm$ are of degree~$2$.
The following is easily checked.
\begin{lemma}
\begin{enumerate}[{\rm(a)}]
\item\label{lem:algebra A sl2 action.a} For $x,y\in\mathcal A$ homogeneous we have 
\begin{gather}\label{eq:divided power E_i}
\begin{aligned}
&\ul E_+^{(r)}(x)=\sum_{r'+r''=r}(-1)^{r'}v^{\frac12(r+|x|-1)(r'-r'')}E^{\la r'\ra}x E^{\la r''\ra}\\
&\ul E_-^{(r)}(x)=\sum_{r'+r''=r}(-1)^{r''}v^{\frac12(r+|x|-1)(r''-r')}E^{\la r'\ra}x E^{\la r''\ra}
\end{aligned} 
\\
\label{eq:powers E_i action}
\ul E_\pm^{(r)}(xy)=\sum_{r'+r''=r} v^{\pm\frac12(r'|y|-r''|x|)} \ul E_\pm^{(r')}(x)\ul E_\pm^{(r'')}(y).
\end{gather}
and $[\ul E_\pm,\ul F_\pm]=\la 1\ra_v{}^{-1}(\ul K-\ul K^{-1})$.
%\end{gather}
In particular, $\ul E_\pm$, $\ul F_\pm$ and $\ul K$ provide actions of Chevalley generators of $U_v(\lie{sl}_2)$ in its {\em standard} presentation 
on~$\mathcal A$;
\item\label{lem:algebra A sl2 action.b} $\ul E_\pm$ restrict to endomorphisms of $\mathcal A_\pm$. 
In particular, $\mathcal A_\pm$ is a $U_v(\lie{sl}_2)_\pm$-submodule of~$\mathcal A$ and $\mathcal A_0$
is the space of lowest weight vectors for both actions. 
\item\label{lem:algebra A sl2 action.c}
A homomorphism of decorated algebras $\mathcal A\to \mathcal A'$ is a homomorphism 
of $U_v(\lie{sl}_2)_+$- and $U_v(\lie{sl}_2)_-$-modules.
\end{enumerate}
\label{lem:algebra A sl2 action}
\end{lemma}
\begin{remark}
Suppose that~$y\in\mathcal A_0$. The following is rather standard an is an obvious consequence of say~\cite{Lus-book}*{Corollary~3.1.9}.
\begin{equation}\label{eq:lowest weight action}
\ul F_\pm^{(a)}\ul E_\pm^{(b)}(y)=\begin{cases}\displaystyle\binom{a-b-|y|}{a}_v \ul E_\pm^{(b-a)}(y),& 0\le a\le b\\
                                   0,& a>b
                                  \end{cases}
\end{equation}
\end{remark}

Suppose that $\ul E_\pm$ are locally nilpotent on~$\mathcal A_\pm$. Then $\mathcal A_\pm$ are direct sums of finite dimensional
$U_v(\lie{sl}_2)_\pm$-modules and if $x\in\mathcal A_0$ is homogeneous then $|x|\le 0$.
We need following
\begin{lemma}
\begin{enumerate}[{\rm(a)}]
 \item\label{lem:sigma eta defn.a}
 There exists unique isomorphisms of~$U_v(\lie{sl}_2)$-modules $\sigma_\pm:\mathcal A_\pm\to \mathcal A_\mp$ such that 
 $\sigma_\pm|_{\mathcal A_0}=\id_{\mathcal A_0}$, where $\mathcal A_\pm$ is regarded as a $U_v(\lie{sl}_2)_\pm$-module. 
 In particular, $\sigma_\pm\circ \sigma_\mp=\id_{\mathcal A_\mp}$.
 \item\label{lem:sigma eta defn.b}
 There exists unique $\kk$-linear involution $\eta_\pm:\mathcal A_\pm\to\mathcal A_\pm$ such that 
 \begin{equation}\label{eq:intertwiner}
 \eta_\pm\circ\ul E_\pm=\ul F_\pm\circ\eta_\pm,\quad \eta_\pm\circ \ul F_\pm=\ul E_\pm\circ\eta_\pm,\quad 
 \eta_\pm\circ \ul K=\ul K^{-1}\circ\eta_\pm
\end{equation}
 and $\eta_\pm(x)=\ul E_\pm^{(top)}(x)=\ul E_\pm^{(-|x|)}(x)$ for $x\in\mathcal A_0$ homogeneous.
\end{enumerate}
\label{lem:sigma eta defn}
\end{lemma}
\begin{proof}
Part~\eqref{lem:sigma eta defn.a} is immediate from the semi-simplicity of~$\mathcal A_\pm$ as $U_v(\lie{sl}_2)_\pm$-modules and 
the fact that any endomorphism of any lowest weight $U_v(\lie{sl}_2)$-module fixing all lowest weight vectors is identity on that module. 
To prove~\eqref{lem:sigma eta defn.b} recall that every simple finite dimensional $U_v(\lie{sl}_2)$-module $V_\lambda$ of type~$1$ 
has a basis $\{z_k\}_{0\le k\le}$ such that $\ul E(x_k)=(k)_v z_{k-1}$, $\ul F(z_k)=(\lambda-k)_v z_{k+1}$, $\ul K(z_k)=v^{\lambda-2k}z_k$.
Then it is easy to see that $\eta_\lambda\in\End_\kk V_\lambda$ defined by $\eta(z_k)=z_{\lambda-k}$ is the unique linear map 
satisfying~\eqref{eq:intertwiner} and such that $\eta_\lambda(z)=\ul E^{(\lambda)}(z)$ for any lowest weight vector~$z$ of~$V_\lambda$.
It remains to observe that $\eta_\lambda$ can be extended uniquely to any semi-simple $U_v(\lie{sl}_2)$-module.
\end{proof}
\subsection{An isomorphism between \texorpdfstring{$\mathcal A_-$}{A-} and~\texorpdfstring{$\mathcal A_+$}{A+}}
The following is quite surprising.
\begin{theorem}\label{thm:tau-homomorphism}
Let $\mathcal A$ be a decorated algebra such that the operators $\ul E_\pm$ are locally nilpotent on~$\mathcal A_\pm$. 
Then the map $\tau:=\eta_+\circ\sigma_-:\mathcal A_-\to \mathcal A_+$ is an isomorphism of algebras.
\end{theorem}

\begin{proof}
Let $x,y\in\mathcal A_0$ be homogeneous and let $m=-|x|$, $n=-|y|$.  
For $r\ge 0$, define $x*_r y\in\mathcal A$
by 
$$
x*_r y=
\sum_{\substack{r',r''\ge 0\\ r'+r''\le r}} (-1)^{r'+r''} \prod_{t=1}^{r'} (n-r+t)_{v}
\prod_{t=1}^{r''}(m-r+t)_v\prod_{t=r'+r''+1}^{r}\mskip-7mu (m+n-2r+t+1)_v E^{\la r'\ra}x E^{\la r-r'-r''\ra}y E^{\la r''\ra}.
$$

Clearly $x*_r y$ is homogeneous of degree $2r-m-n$.
\begin{proposition}\label{prop:common lowest vector}
Let $\mathcal A$ be a decorated algebra and let $x,y\in\mathcal A_0$ be homogeneous
with $|x|=-m$, $|y|=-n$.
For all $r\ge 0$ we have $x*_r y\in \mathcal A_0$ and 
\begin{equation}\label{eq:two expressions for x*_r y}
\begin{aligned}
x*_r y&=\sum_{t'+t''=r}(-1)^{t''}v^{\frac12(m t'-n t''+(r-1)(t''-t'))}
\frac{ (m-t')_v!(n-t'')_v!}{(n-r)_v!(m-r)_v!}\ul E_+^{(t')}(x)\ul E_+^{(t'')}(y)
\\&=\sum_{t'+t''=r}(-1)^{t'}v^{\frac12(nt''-m t'+(r-1)(t'-t''))}
\frac{ (m-t')_v!(n-t'')_v!}{(n-r)_v!(m-r)_v!}\ul E_-^{(t')}(x)\ul E_-^{(t'')}(y)
\end{aligned}
\end{equation}
\end{proposition}
\begin{proof}
By Lemma~\ref{lem:free subalgebra} it suffices to prove the proposition for the decorated algebra $\mathcal F_{m,n}$.
Let $\mathcal V_{m,n}$ be the subspace of~$\mathcal F_{m,n}$ with the basis $\{ E^{\la a\ra} x E^{\la b\ra} y E^{\la c\ra}\,:\,a,b,c\in\ZZ_{\ge 0}\}$.
Clearly, $\ul E_\pm(\mathcal V_{m,n})$, $\ul F_\pm(\mathcal V_{m,n})\subset\mathcal V_{m,n}$ and $x*_r y\in \mathcal V_{m,n}$.
We need the following
\begin{lemma}\label{lem:free kernels}
$(\mathcal F_{m,n})_0\cap \mathcal V_{m,n}$ is spanned by the $x*_r y$, $r\ge 0$ as a $\kk$-vector space.
\end{lemma}
\begin{proof}
It is easy to check that $x*_r y\in (\mathcal F_{m,n})_0$. Conversely, let $u\in(\mathcal F_{m,n})_0\cap \mathcal V_{m,n}$ be homogeneous
of degree $2r-m-n$ and write
$$
u=\sum_{r',r''\ge 0,r'+r''\le r} c_{r',r''} E^{\la r'\ra}x E^{\la r-r'-r''\ra}y E^{\la r''\ra}.
$$
Then 
\begin{align*}
\la 1\ra_v\ul F_\pm(u)&=\sum_{r'\ge 1,\,r'\ge 0,\,r'+r''\le r}
c_{r',r''} v^{\pm\frac12(|x|+|y|+2(r-r'))} E^{\la r'-1\ra}x E^{\la r-r'-r''\ra}y E^{\la r''\ra}
\\&\quad+\sum_{r',r''\ge0,r'+r''\le r-1}
c_{r',r''}v^{\pm\frac12(|y|-|x|+2(r''-r'))}E^{\la r'\ra}x E^{\la r-r'-r''-1\ra}y E^{\la r''\ra}
\\&\quad+
\sum_{r'\ge0,\,r''\ge 1,\,r'+r''\le r}
c_{r',r''}v^{\mp\frac12(|x|+|y|+2(r-r''))}E^{\la r'\ra}x E^{\la r-r'-r''\ra}y E^{\la r''-1\ra}
\\
&=\sum_{r',r''\ge 0,\,r'+r''\le r-1}
(c_{r'+1,r''}v^{\pm\frac12(|x|+|y|+2(r-r'-1))}+c_{r',r''}v^{\pm\frac12(|y|-|x|+2(r''-r'))}\\&\quad+c_{r',r''+1}v^{\mp\frac12(|x|+|y|+2(r-r''-1))})
E^{\la r'\ra}x E^{\la r-r'-r''\ra}y E^{\la r''\ra}.
\end{align*}
It follows that 
$$
c_{r'+1,r''}\la |x|+|y|+2r-r'-r''-2\ra_v
+c_{r',r''}\la |y|-r'+r-1\ra_v=0
$$
and 
$$
c_{r',r''+1}\la |x|+|y|+2r-r'-r''-2\ra_v
+c_{r',r''}\la |x|-r''+r-1\ra_v=0.
$$
Thus,
\begin{multline*}
c_{r',r''}
=(-1)^{r'} \frac{\prod_{t=1}^{r'} (n-r+t)_v}{\prod_{t=1}^{r'} (m+n-2r+r''+t)_v}
c_{0,r''}\\=(-1)^{r'+r''} \frac{ \prod_{t=1}^{r'}(n-r+t)_v\prod_{t=1}^{r''}(m-r+t)_v}{\prod_{t=1}^{r'+r''} (m+n-2r+t+1)_v} c_{0,0}.
\end{multline*}
Therefore, $u=\prod_{t=1}^r (m+n-2r-t+1)_v{}^{-1} c_{0,0} x*_r y$. 
\end{proof}
Thus, $x*_ry\in \mathcal A_0$. Furthermore, it is easy to check, using~\eqref{eq:skew Leibnitz rule}, that 
right hand sides of~\eqref{eq:two expressions for x*_r y} are in $(\mathcal F_{m,n})_0\cap \mathcal V_{m,n}$ and 
hence proportional to $x*_r y$ by Lemma~\ref{lem:free kernels}. It remains then to compare the coefficient
of $E^{\la r\ra} x y$ in both expressions, which is easily calculated using~\eqref{eq:divided power E_i}.
\end{proof}
\begin{remark}\label{rem:Verma}
Clearly, there exist unique injective homomorphisms $j_{\pm,m}$ and $j_{\pm,n}$ from the lowest weight Verma $U_v(\lie{sl}_2)_\pm$-modules $M^\pm_{-m}$, $M^\pm_{-n}$
of lowest weight $-m$
(respectively, $-n$)
to $\mathcal V_{m,n}$ sending a fixed lowest weight vector to $x$ (respectively, to $y$). This yields natural injective homomorphisms 
of $U_v(\lie{sl}_2)_\pm$-modules  $j_{\pm,m,n}:M^\pm_{-m}\tensor M^\pm_{-n}\to \mathcal V_{m,n}$ where the comultiplication on~$U_v(\lie{sl}_2)_\pm$
is defined by
$\Delta_\pm(\ul E_\pm)=\ul E_\pm\tensor \ul K^{\pm\frac12}+\ul K^{\pm\frac12}\tensor \ul E_\pm$. In particular, Proposition~\ref{prop:common lowest vector}
implies that $j_{\pm,m,n}(M^\pm_{-m}\tensor M^{\pm}_{-n})$ share lowest weight vectors of weight $-m-n+2r$.
\end{remark}
The following Lemma is essentially concerned with quantum Clebsch-Gordan coefficients (also known as $3j$-symbols, see e.g.~\cite{Kassel}*{Chapter VII}).
\begin{lemma}
\begin{enumerate}[{\rm(a)}]
Let $\mathcal A$, $x,y\in\mathcal A_0$ be as in Proposition~\ref{prop:common lowest vector}.
\item For $r\le \min(m,n)$ we have
\begin{equation}\label{eq:CG-defn}
\begin{aligned}
&\ul E_+^{(a)}(x*_r y)=\sum_{t'+t''=a+r} C_{r;t',t''}(v)\ul E_+^{(t')}(x)\ul E_+^{(t'')}(y),\\
&\ul E_-^{(a)}(x*_r y)=\sum_{t'+t''=a+r} (-1)^r C_{r;t',t''}(v^{-1})\ul E_-^{(t')}(x)\ul E_-^{(t'')}(y)
\end{aligned}
\end{equation}
where
$$
C_{r;t',t''}(v)=v^{\frac12(m t''-n t')}
\sum_{k+l=r}
(-1)^l v^{l t'-k t''+\frac12(k-l)(1+m+n-r)}
\frac{ (n-l)_v! (m-k)_v!}{(n-r)_v!(m-r)_v!}\binom{t'}{k}_v\binom{t''}{l}_v\in\ZZ[v,v^{-1}].
$$
\item\label{lem:Clebsch-Gordan.b} The $C_{r;t',t''}(v)$ satisfy the following recurrence relations
\begin{gather}
\label{eq:recurr-relation for CG}
\begin{split}
(m+n-r-t'-t'')_v C_{r;t',t''}(v)&=v^{t''-\frac12 n}(m-t')_v C_{r;t'+1,t''}(v)\\&\qquad\qquad+v^{\frac12 m-t'}(n-t'')_v C_{r;t',t''+1}(v),
\end{split}\\
\label{eq:recurr-relation for CG.2}
(t'+t''-r)_v C_{r;t',t''}(v)=v^{t''-\frac12n}(t')_v C_{r;t'-1,t''}(v)+v^{\frac12 m-t'}(t'')_v C_{r;t',t''-1}(v).
\end{gather}
\item\label{lem:Clebsch-Gordan.c} For all $0\le t'\le m$, $0\le t''\le n$, $0\le r\le \min(m,n)$ we have 
\begin{equation}\label{eq:key-combinatorial-identity}
C_{r;m-t',n-t''}(v)=(-1)^r C_{r;t',t''}(v^{-1}).
\end{equation}
\end{enumerate}
\label{lem:Clebsch-Gordan}
\end{lemma}
\begin{proof}
Let $a=0$. Then 
$$
C_{r;t',t''}(v)=(-1)^{t''}v^{\frac12(m t'-n t''+(r-1)(t''-t'))}
\frac{ (m-t')_v!(n-t'')_v!}{(n-r)_v!(m-r)_v!}
$$
and the assertion follows from~\eqref{eq:two expressions for x*_r y}. The case of arbitrary $a$ is then easily deduced by applying $\ul E_\pm^{(a)}$ 
to~\eqref{eq:two expressions for x*_r y} and using Lemma~\ref{lem:algebra A sl2 action}\eqref{lem:algebra A sl2 action.a}.
To prove~\eqref{eq:recurr-relation for CG} (respectively, \eqref{eq:recurr-relation for CG.2})
it suffices to apply $\ul F$ (respectively, $\ul E$) to both sides of the first identity in~\eqref{eq:CG-defn}. 
We leave the details of these computations as an exercise for the reader.

We now prove part~\eqref{lem:Clebsch-Gordan.c}.
Note first that for all $0\le t'\le m$, $0\le t''\le n$
\begin{equation}\label{eq:boundary-cnd}
\begin{aligned}
&C_{r;t',0}(v)=
v^{\frac12 (r(1+m+n-r)-nt')}
\frac{ (n)_v!}{(n-r)_v!}\binom{t'}{r}_v,
\\
&C_{r;0,t''}(v)=
(-1)^r v^{\frac12(mt''-r(1+m+n-r))}
\frac{ (m)_v!}{(m-r)_v!}\binom{t''}{r}_v.
\end{aligned}
\end{equation}

Using~\eqref{eq:recurr-relation for CG} with $t''=n$ we obtain
$$
C_{r;t'+1,n}(v)=\frac{(m-r-t')_v}{(m-t')_v}\, v^{-\frac12 n} C_{r;t',n}
$$
whence by~\eqref{eq:boundary-cnd}
\begin{multline*}
C_{r;t',n}(v)=v^{-\frac12 t'n}\frac{ (m-r)_v!(m-t')_v!}{(m)_v!(m-r-t')_v!}\, C_{r;0,n}(v)\\=(-1)^r
v^{\frac12(n(m-t')-r(1+m+n-r))}\frac{(n)_v!}{(n-r)_v!}\binom{m-t'}{r}_v=(-1)^r C_{r;m-t',0}(v^{-1}).
\end{multline*}
Thus, \eqref{eq:key-combinatorial-identity} holds for all $0\le t'\le m$ and for $t''=n$.
Suppose that~\eqref{eq:key-combinatorial-identity} was established for all $0\le t'\le m$ and for all $s+1\le t''\le n$.
We have by~\eqref{eq:boundary-cnd}
\begin{multline*}
(-1)^r(n-r-s)_v C_{r;m,s}(v^{-1})=(-1)^r C_{r;m,s+1}(v^{-1})(n-s)_v v^{\frac12 m}=v^{\frac12 m}(n-s)_v C_{r;0,n-s-1}(v)
\\
=(-1)^r v^{\frac12(m(n-s)-r(1+m+n-r))}
\frac{ (m)_v!(n-s)_v}{(m-r)_v!}\binom{n-s-1}{r}_v
=(n-r-s)_v C_{r;0,n-s}(v).
\end{multline*}
Finally, assume that~\eqref{eq:key-combinatorial-identity} is established for $k+1\le t'\le m$ and for $t''=s$. Then using~\eqref{eq:recurr-relation for CG} 
and~\eqref{eq:recurr-relation for CG.2} we obtain
\begin{align*}
(-1)^r(m+&n-r-k-s)_v C_{r;k,s}(v^{-1})\\&=(-1)^r C_{r;k+1,s}(v^{-1})(m-k)_v v^{-s+\frac12 n}+(-1)^r C_{r;k,s+1}(v^{-1})(n-s)_v v^{-\frac12 m+k}
\\&=C_{r;m-k-1,n-s}(v)(m-k)_v v^{-s+\frac12 n}+C_{r;m-k,n-s-1}(n-s)_v v^{-\frac12 m+k}
\\&=(m+n-r-k-s)C_{r;m-k,n-s}(v).
\end{align*}
This proves the inductive step and completes the proof of the Lemma.
\end{proof}

We can now complete the proof of Proposition~\ref{thm:tau-homomorphism}. By construction, $\tau$ is an isomorphism of $U_v(\lie{sl}_2)$-modules. 
Explicitly, if $z\in \mathcal A_0$ then $\tau(\ul E_-^{(r)}(z))=\ul E_+^{(-|z|-r)}(z)$.
It suffices to prove that for any $x,y\in\mathcal A_0$ homogeneous
with $|x|=-m$, $|y|=-n$ we have 
$$
\tau(E_-^{(k)}(x))\tau(E_-^{(l)}(y))=\ul E_+^{(m-k)}(x)\ul E_+^{(n-k)}(y)=\tau(E_-^{(k)}(x)E_-^{(l)}(y)).
$$
It is immediate from the Remark~\ref{rem:Verma} that $C_{r;t',t''}(v)$ (respectively, $(-1)^r C_{r;t',t''}(v^{-1})$)
provide the transition matrix between the two bases of $U_v(\lie{sl}_2)_+$-) (respectively, $U_v(\lie{sl}_2)_-$) modules 
$V_m\tensor V_n=\bigoplus_{0\le k\le\min(m,n)} V_{m+n-2k}$. 
In particular, there exists $\tilde C_{r;k,l}(v)\in\kk$, $0\le k\le m$, $0\le l\le n$, $0\le r\le \min(m,n,k+l)$ such that 
\begin{equation}\label{eq:inverse}
\sum_{r=0}^{\min(m,n,k+l)}(-1)^r \tilde C_{r;k,l}(v) C_{r;t',t''}(v^{-1})=\delta_{k,t'}\delta_{l,t''}.
\end{equation}
Then 
$$
\ul E_-^{(k)}(x)\ul E_-^{(l)}(y)=\sum_{r=0}^{\min(m,n,k+l)} \tilde C_{r;k,l}(v) \ul E_-^{(k+l-r)}(x*_r y)
$$
and so
\begin{align*}
\tau(\ul E_-^{(k)}(x)&\ul E_-^{(l)}(y))=\sum_{r=0}^{\min(m,n,k+l)} \tilde C_{r;k,l}(v)\ul E_+^{(m+n-r-k-l)}(x*_r y)
\\
&=\sum_{r=0}^{\min(m,n,k+l)} \tilde C_{r;k,l}(v)\sum_{s'+s''=m+n-k-l}C_{r;s',s''}(v)\ul E_+^{(s')}(x)\ul E_+^{(s'')}(y)
\\
&=\sum_{t'+t''=k+l}\Big(\sum_{r=0}^{\min(m,n,k+l)} \tilde C_{r;k,l}(v)C_{r;m-t',n-t''}(v)\Big)\ul E_+^{(m-t')}(x)\ul E_+^{(n-t'')}(y)
\\
&=\sum_{t'+t''=k+l}\Big(\sum_{r=0}^{\min(m,n,k+l)} (-1)^r \tilde C_{r;k,l}(v)C_{r;t',t''}(v^{-1})\Big)\ul E_+^{(m-t')}(x)\ul E_+^{(n-t'')}(y)
\\
&=\ul E_+^{(m-k)}(x)\ul E_+^{(n-l)}(y)=\tau(E_-^{(k)}(x))\tau(E_-^{(l)}(y)),
\end{align*}
where we used~\eqref{eq:key-combinatorial-identity} and~\eqref{eq:inverse}.
\end{proof}
Note that for $x\in \mathcal A_-$ homogeneous, $\tau(x)$ can be calculated explicitly in the following way. First, if 
$y\in \mathcal A_0$ is homogeneous and $x=\ul E_-^{(r)}(y)$ then 
\begin{equation}\label{eq:tau expl}\tau(x)=\ul E_+^{(-|y|-r)}(y)
=\ul E_+^{(r-|x|)}(y)=\binom{2r-|x|}{r}_v^{-1}\ul E_+^{(r-|x|)}\ul F_-^{(r)}(x). 
\end{equation}
By linearity, it remains to observe that any homogeneous element of~$\mathcal A_-$
can be written, uniquely, as $x=\sum_{r\ge \max(0,|x|)} \ul E_-^{(r)}(x_r)$ where $x_r\in \mathcal A_0$ and $|x_r|=|x|-2r$.

We will also need the following property of~$\tau$.
\begin{lemma}
$\ul F_+^{(top)}\circ\tau=\ul F_-^{(top)}$.
\label{lem:tau properties}
\end{lemma}
\begin{proof}
Given $x\in\mathcal A_-$, write $x=\sum_{r\ge \max(0,|x|)}\ul E_-^{(r)}(x_r)$ where $x_r\in\mathcal A_0$ and $|x_r|=|x|-2r$. 
Then $\tau(x)=\sum_{r\ge \max(0,|x|)}\ul E_+^{(r-|x|)}(x_r)$.
Let 
$r_0=\max\{r\ge \max(0,|x|)\,:\, x_r\not=0\}$. Then by~\eqref{eq:lowest weight action}
$$
\ul F_-^{(top)}(x)=\ul F_-^{(r_0)}(x)=\ul F_-^{(r_0)}\ul E_-^{(r_0)}(x_{r_0})=\binom{2r_0-|x|}{r_0}_v x_{r_0}.
$$
On the other hand,
\begin{equation*}
\ul F_+^{(top)}\tau(x)=\ul F_+^{(r_0-|x|)}\tau(x)=\ul F_+^{(r_0-|x|)}\ul E_+^{(r_0-|x|)}(x_{r_0})=\binom{2r_0-|x|}{r_0}_vx_{r_0}.\qedhere
\end{equation*}
\end{proof}

\subsection{\texorpdfstring{$U_q(\lie n^+)$}{Uq(n+)} as a decorated algebra}
Define a comultiplication $\Delta$ on~$U_q(\lie g)$ by 
$$
\Delta(E_i)=E_i\tensor 1+K_i^{-1}\tensor E_i,\qquad \Delta(F_i)=1\tensor F_i+F_i\tensor K_i,\qquad i\in I.
$$
Then (cf.~\cite{Lus-book}*{\S3.1.5})
\begin{equation}\label{eq:delta-angl powers}
\Delta(E_i^{\la r\ra})=\sum_{r'+r''=r} q_i^{-r'r''}E_i^{\la r'\ra}K_i^{-r''}\tensor E_i^{\la r''\ra},
\quad 
\Delta(F_i^{\la r\ra})=\sum_{r'+r''=r} q_i^{r'r''} F_i^{\la r'\ra}\tensor K_i^{r'}F_i^{\la r''\ra}.
\end{equation}
For any $J,J'\subset I$ denote $U_\ZZ(\lie g)_{J,J'}$ the $\AA_0$-subalgebra of $U_q(\lie g)$ generated by $U_\ZZ(\lie n^-)_J$,
$U_\ZZ(\lie n^+)_{J'}$, the $K_i^{\pm 1}$ and the $\binom{K_i;c}{a}_{q_i}$, $a\in \ZZ_{\ge 0}$, $c\in \ZZ$ and $i\in J\cap J'$,
where 
$$
\binom{K;c}{a}_{v}=\prod_{k=0}^{a-1} \frac{K^{-1}v^{k-c}-K v^{c-k}}{v^{k+1}-v^{-k-1}}.
$$
We also abbreviate $U_\ZZ(\lie g)_J=U_\ZZ(\lie g)_{J,I}$ and $U_\ZZ(\lie g):=U_\ZZ(\lie g)_{I,I}$. 
The corresponding $\kk$-subalgebras of~$U_q(\lie g)$ will be denoted $U_q(\lie g)_{J,J'}$.
It follows from~\eqref{eq:delta-angl powers} that $U_\ZZ(\lie g)$ is a Hopf $\AA_0$-algebra.

Let $\ad$ be the corresponding adjoint action of~$U_q(\lie g)$ on itself. Consider the extension~$\widetilde U_q(\lie g)$
of $U_q(\lie g)$ obtained by adjoining $K_i^{\pm\frac12}$, $i\in I$.
Define operators $\ul E_i$, $\ul F_i$ on~$\widetilde U_q(\lie g)$ via 
\begin{equation}\label{eq:ul definition}
\begin{aligned}
&\ul E_i(x)=-(\ad E_i^{\la 1\ra}K_i^{\frac12})=\frac{E_i K_i^{\frac12} xK_i^{-\frac12}-K_i^{-\frac12}xK_i^{\frac12}E_i}{q_i^{-1}-q_i},\\
&\ul F_i(x)=(\ad K_i^{-\frac12}F_i^{\la 1\ra})(x)=\partial_i(x)-K_i^{-1}\partial_i^{op}(x)K_i^{-1}.
\end{aligned}
\end{equation}
Clearly, $\ul E_i$ and $\ul F_i$ restrict to operators on $U_q(\lie g)$ and we have 
\begin{equation}\label{eq:sl_2-rel}
[\ul E_i,\ul F_i]=-\la 1\ra_{q_i}{}^{-2}\ad([E_i,F_i])=\la 1\ra_{q_i}{}^{-1}(\ul K_i-\ul K_i^{-1}),
\end{equation}
where $\ul K_i(x)=K_i x K_i^{-1}$.
We will also need operators $\ul E_i^{op}$, $\ul F_i^{op}$ defined by
\begin{equation}\label{eq:ul E_i^op}
\ul E_i^{op}(x)=(\ul E_i(x^*))^*,\qquad \ul F_i^{op}(x)=(\ul F_i(x^*))^*.
\end{equation}

We collect some properties of these operators in the following Lemma.
\begin{lemma}
\begin{enumerate}[{\rm(a)}]
\item\label{lem:prop-ul E_i.a'''} $U_q(\lie n^+)$ is a decorated algebra with $E=E_i$, $\ul F_+=\partial_i$,
$\ul F_-=\partial_i^{op}$, $v=q_i$ and $|x|=(\alpha_i^\vee,\gamma)$ for $x\in U_q(\lie n^+)_\gamma$; in particular,
$\ul E_+=\ul E_i$ and $\ul E_-=\ul E_i^{op}$.
\item\label{lem:prop-ul E_i.a'} $\ul E_i$, $\ul F_i$ commute with $\bar\cdot$.
\item\label{lem:prop-ul E_i.a''}
If $x\in U_\ZZ(\lie g)_\gamma$ then 
$\ul E_i^{(r)}(x),\ul F_i^{(r)}(x)\in q^{\frac12 (r\alpha_i,\gamma)}U_\ZZ(\lie g)$ for all $r\in\ZZ_{\ge0}$;
\item\label{lem:prop-ul E_i.b} For all $x\in U_q(\lie n^+)_\gamma$, $y\in U_q(\lie n^+)$  we have 
$$
\fgfrm{\ul E_i^{(r)}(x)}{y}=
\sum_{r'+r''=r}(-1)^{r'}q_i^{-\frac12(r+(\alpha_i^\vee,\gamma)-1)(r''-r')}\fgfrm{x}{\partial_i^{(r'')}
(\partial_i^{op})^{(r')}(y)}\\
$$
\item\label{lem:prop-ul E_i.c}
$T_i\circ\ul E_i^{op}=\ul F_i\circ T_i$, $T_i\circ \ul F_i^{op}=\ul E_i\circ T_i$.

\end{enumerate}\label{lem:prop-ul E_i}
\end{lemma}
\begin{proof}
Parts~\eqref{lem:prop-ul E_i.a'''} and~\eqref{lem:prop-ul E_i.a'} are obvious from the definitions. 
Since $U_\ZZ(\lie g)$ is a Hopf $\AA_0$-algebra, the first assertion in~\eqref{lem:prop-ul E_i.a''} follows from
$$
\ul E_i^{(r)}
=(-1)^r q_i^{\binom{r}2}\ad(E_i^{\la r\ra}K_i^{\frac r2}),\quad \ul F_i^{(r)}=q_i^{\binom r2}(\ad K_i^{-\frac r2}F_i^{\la r\ra}),
$$
while the second is immediate from the above formulae and~\eqref{eq:delta-angl powers}.
Part~\eqref{lem:prop-ul E_i.b} is immediate from part~\eqref{lem:prop-ul E_i.a'''}, \eqref{eq:divided power E_i} and 
Lemma~\ref{lem:partial-prop}\eqref{lem:partial-prop.ii}. Part~\eqref{lem:prop-ul E_i.c} is easy to check using Lemma~\ref{lem:T_i-defn-prop}.
\end{proof}
\subsection{A new formula for~\texorpdfstring{$T_i$}{Ti}}\label{subs:formula T_i}
Using the notation from~\cite{BG-dcb}, denote by $U_i:=T_i^{-1}(U_q(\lie n^+))\cap U_q(\lie n^+)$
and ${}_iU:=T_i(U_q(\lie n^+))\cap U_q(\lie n^+)$. It follows from~\cite{Lus-book}*{Proposition~38.1.6} that 
$U_i=\ker\partial_i$ and ${}_i U=\ker\partial_i^{op}$.
Let $U^\ZZ_i=U_i\cap U^\ZZ(\lie n^+)$ and ${}_iU^\ZZ={}_i U\cap U^\ZZ(\lie n^+)$.

\begin{lemma}
Let $i\in I$.
\label{lem:ul E_i adjoint}
\begin{enumerate}[{\rm(a)}]
\item\label{lem:ul E_i adjoint.a} 
For all $x\in U_q(\lie n^+)_\gamma$, $y\in {}_i U$, $z\in U_i$ and $r\ge 0$ we have 
$$
\fgfrm{\ul E_i^{(r)}(x)}{y}=q_i^{-\frac12(r+(\alpha_i^\vee,\gamma)-1)r}\fgfrm{x}{\ul F_i^{(r)}(y)}, \quad
\fgfrm{(\ul E_i^{op})^{(r)}(x)}{z}=q_i^{-\frac12(r+(\alpha_i^\vee,\gamma)-1)r}\fgfrm{x}{(\ul F_i^{op})^{(r)}(z)}.
$$

\item\label{lem:ul E_i adjoint.b}
$\ul E_i$, $\ul F_i$ (respectively, $\ul E_i^{op}$, $\ul F_i^{op}$) restrict to locally nilpotent operators on~${}_i U$ (respectively, on~$U_i$).
\item\label{lem:ul E_i adjoint.c}
$\ul E_i^{(n)}({}_i U^\ZZ),\ul F_i^{(n)}({}_i U^\ZZ)\subset {}_i U^\ZZ$ (respectively, 
$(\ul E_i^{op})^{(n)}(U^\ZZ_i), (\ul F_i^{op})^{(n)}(U^\ZZ_i)\subset U^\ZZ_i$) for all $n\ge 0$.
\end{enumerate}
\end{lemma}
\begin{proof}
We only prove the assertion for $\ul E_i$ and~$\ul F_i$. The assertion for $\ul E_i^{op}$ and~$\ul F_i^{op}$ is proved similarly 
using~\eqref{eq:ul E_i^op} and the fact that ${}^*$ is self-adjoint with respect to $\fgfrm{\cdot}{\cdot}$.
Since by~\eqref{eq:ul definition} $\ul F_i|_{{}_i U}=\partial_i|_{{}_i U}$, in particular, 
$\ul F_i$ is a locally nilpotent operator on~${}_i U$. Part~\eqref{lem:ul E_i adjoint.a} is now immediate from Lemma~\ref{lem:prop-ul E_i}\eqref{lem:prop-ul E_i.b}.

Suppose that $x\in U_q(\lie n^+)_\gamma$ and 
$\ul E_i^{(n)}(x)\not=0$ for all $n\ge 0$.
Then $T_i(\ul E_i^{(n)}(x))$ is homogeneous of degree~$s_i(\gamma+n\alpha_i)=\gamma-((\alpha_i^\vee,\gamma)+n)\alpha_i\notin Q^+$ for 
$n\gg 0$. Since $T_i({}_i U)\subset U_q(\lie n^+)$, this is a contradiction. This proves~\eqref{lem:ul E_i adjoint.b}.

To prove~\eqref{lem:ul E_i adjoint.c}, note that the assertion for $\ul F_i$ follows from Corollary~\ref{cor:preserv-lattice}. Since 
$U_\ZZ(\lie n^+)={}_i U_\ZZ\oplus (E_i U_q(\lie n^+)\cap U_\ZZ(\lie n^+))$, it suffices to prove that for $x\in {}_i U^\ZZ\cap U_q(\lie n^+)_\gamma$,
$y\in U_i^\ZZ\cap U_q(\lie n^+)_{\gamma+n\alpha_i}$ we have 
$\fgfrm{\ul E_i^{(n)}(x)}{y}\in \AA_0$. 
But for such $y$ we have by part~\eqref{lem:ul E_i adjoint.a} and~\eqref{eq:mu-prop}
$$
\fgfrm{\ul E_i^{(n)}(x)}{y}=q_i^{-\frac12 n(n+(\alpha_i^\vee,\gamma)-1)}\fgfrm{x}{\ul F_i^{(n)}(y)}=q_i^{\binom{n}2}\fgfrm{x}
{q_i^{-\frac12 n(\alpha_i^\vee,\gamma+n\alpha_i)}\ul F_i^{(n)}(y)}\in \AA_0
$$
and $q_i^{-\frac12 n(\alpha_i^\vee,\gamma+n\alpha_i)}\ul F_i^{(n)}(y)\in U_\ZZ(\lie n^+)$ by Lemma~\ref{lem:prop-ul E_i}\eqref{lem:prop-ul E_i.a''}.
\end{proof}
Thus, given $x\in U_i\cap U_q(\lie n^+)_\gamma$, $y\in {}_i U\cap U_q(\lie n^+)_\gamma$ we can write uniquely 
\begin{equation}\label{eq:sl_2 decomp}
x=\sum_{r\ge \max(0,(\alpha_i^\vee,\gamma))} (\ul E_i^{op})^{(r)}(x_r),\quad 
y=\sum_{r\ge \max(0,(\alpha_i^\vee,\gamma))} \ul E_i^{(r)}(y_r),\quad x_r,y_r\in {}_i U\cap U_i\cap U_q(\lie n^+)_{\gamma-r\alpha_i},
\end{equation}
and $x_r,y_r=0$ for $r\gg 0$.

\begin{corollary}\label{cor:form sl2 decomp}
Let $x,x'\in U_i\cap U_q(\lie n^+)_\gamma$, $y,y'\in {}_i U\cap U_q(\lie n^+)_\gamma$
and write $x$, $x'$ and $y$, $y'$ as in~\eqref{eq:sl_2 decomp}. Then 
\begin{equation}\label{eq:form sl2 decomp}
\begin{aligned}
\fgfrm{x}{x'}=\sum_{r\ge \max(0,(\alpha_i^\vee,\gamma))} q_i^{\frac12 r(r+1-(\alpha_i^\vee,\gamma))}\binom{2r-(\alpha_i^\vee,\gamma)}{r}_{q_i}\fgfrm{x_r}{x'_r},
\\
\fgfrm{y}{y'}=\sum_{r\ge \max(0,(\alpha_i^\vee,\gamma))} q_i^{\frac12 r(r+1-(\alpha_i^\vee,\gamma))}\binom{2r-(\alpha_i^\vee,\gamma)}{r}_{q_i}\fgfrm{y_r}{y'_r}.
\end{aligned}
\end{equation}
\end{corollary}
\begin{proof}
Let $z,z'\in {}_i U\cap U_i$, $z'\in U_q(\lie n^+)_{\gamma'}$. Let $a\ge b\ge 0$. Then we have by Lemma~\ref{lem:ul E_i adjoint}\eqref{lem:ul E_i adjoint.a}
and~\eqref{eq:lowest weight action}
\begin{multline*}
\fgfrm{\ul E_i^{(a)}(z)}{\ul E_i^{(b)}(z')}=q_i^{-\frac12 b(b+(\alpha_i^\vee,\gamma')-1)}\fgfrm{\ul F_i^{(b)}\ul E_i^{(a)}(z)}{z'}\\
=q_i^{-\frac12 b(b+(\alpha_i^\vee,\gamma')-1)}\binom{b-a-(\alpha_i^\vee,\gamma')}{b}_{q_i}\fgfrm{\ul E_i^{(a-b)}(z)}{z'}
=\delta_{a,b}q_i^{-\frac12a(a+(\alpha_i^\vee,\gamma')-1)}\binom{-(\alpha_i^\vee,\gamma')}{a}_{q_i}\fgfrm{z}{z'}.
\end{multline*}
This yields the second identity in~\eqref{eq:form sl2 decomp}. The first follows from the second one by applying~${}^*$.
\end{proof}

We now establish a formula for the action of $T_i$ on $U_i$ in terms of $\ul E_i^{op}$ and $\ul E_i$.
\begin{theorem}\label{thm:T_i formula}
Write $x\in U_i\cap U_q(\lie n^+)_\gamma$ as in~\eqref{eq:sl_2 decomp}. Then 
$$
T_i(x)=\sum_{r\ge \max(0,(\alpha_i^\vee,\gamma))} \ul E_i^{(r-(\alpha_i^\vee,\gamma))}(x_r).
$$
In particular, $\partial_i^{(top)}T_i(x)=(\partial_i^{op})^{(top)}(x)$.
\end{theorem}
\begin{proof}
We apply Proposition~\ref{thm:tau-homomorphism} to~$\mathcal A=U_q(\lie n^+)$ which is a decorated algebra by Lemma~\ref{lem:prop-ul E_i}\eqref{lem:prop-ul E_i.a'''}
with locally nilpotent $\ul E_\pm$ on~$\mathcal A_\pm$ by Lemma~\ref{lem:ul E_i adjoint}\eqref{lem:ul E_i adjoint.b}.
We claim that $T_i=\tau$.
Since both $\tau$ and~$T_i$ are isomorphisms of algebras $U_i\to {}_i U$, it is enough to check that they coincide on generators of~$U_i$.

Let $j\not=i$ and define for all $0\le l\le -a_{ij}$
\begin{equation}\label{eq:E_ji^l}
E_{ji^{l}}=\binom{-a_{ij}}{l}_{q_i}^{-1}(\ul E_i^{op})^{(l)}(E_j),\qquad E_{i^lj}=(E_{ji^l})^*=\binom{-a_{ij}}{l}_{q_i}^{-1}\ul E_i^{(l)}(E_j).
\end{equation}
Clearly, $\overline{E_{ji^l}}=E_{ji^l}$ and 
it is immediate from~\eqref{eq:divided power E_i} that 
\begin{equation}\label{eq:E_ji^m explicit}
E_{ji^l}=\binom{-a_{ij}}{l}_{q_i}^{-1}\sum_{r+s=l} (-1)^r q_i^{\frac12(r-s)(l+a_{ij}-1)} E_i^{\la s\ra}E_j E_i^{\la r\ra}.
\end{equation}

\begin{lemma}\label{lem:gen-U_i}
Let $i\not=j\in I$, $0\le m\le -a_{ij}$. Then 
\begin{enumerate}[{\rm(a)}]
\item\label{lem:gen-U_i.b}
$T_i(E_{ji^m})=E_{i^{-a_{ij}-m}j}=\tau(E_{ji^m})$
\item\label{lem:gen-U_i.c}
The elements $E_{ji^l}$ (respectively, $E_{i^lj}$), $j\not=i$, $0\le l\le -a_{ij}$ generate the algebra 
$U_i$ (respectively ${}_i U$).
\end{enumerate}
\end{lemma}
\begin{proof}
To prove~\eqref{lem:gen-U_i.b}, note that by Lemma~\ref{lem:T_i-defn-prop} and~\eqref{eq:E_ji^m explicit} we have $T_i(E_j)=E_{i^{-a_{ij}}j}$. 
On the other hand, $\tau(E_j)=\ul E_i^{(-a_{ij})}(E_j)=E_{i^{-a_{ij}}j}$.
Then by Lemma~\ref{lem:prop-ul E_i}\eqref{lem:prop-ul E_i.c} and~\eqref{eq:lowest weight action}
$$
T_i(E_{ji^l})=
\binom{-a_{ij}}{l}_{q_i}^{-1} \ul F_i^{(l)}(E_{i^{-a_{ij}}j})
=\binom{-a_{ij}}{l}_{q_i}^{-1} \ul F_i^{(l)}\ul E_i^{(-a_{ij})}(E_j)=E_{i^{-a_{ij}-l}j}.
$$
Since by construction $\tau$ also satisfies Lemma~\ref{lem:prop-ul E_i}\eqref{lem:prop-ul E_i.c}, it follows that $\tau(E_{ji^l})=T_i(E_{ji^l})$.

Part~\eqref{lem:gen-U_i.c} can be easily deduced from~\cite{Lus-book}*{\S38.1.1}.
\end{proof}
\noindent
This implies that $\tau=T_i$ on~$U_i$. The second assertion of Theorem~\ref{thm:T_i formula} follows from Lemma~\ref{lem:tau properties}.
\end{proof}
We now prove the following
\begin{proposition}\label{prop:T_i lattice}
For all $i\in I$, $T_i(U^\ZZ_i)={}_i U^\ZZ$.
\end{proposition}
\begin{proof}
We need the following
\begin{lemma}\label{lem:Jantzen}
Any element $x\in U_\ZZ(\lie n^+)$ can be written as $x=\sum_{r,s\ge 0} E_i^{\la r\ra} x_{rs} E_i^{\la s\ra}$, where 
$x_{rs}\in {}_i U\cap U_i\cap U_\ZZ(\lie n^+)$ and only finitely many of them are non-zero.
\end{lemma}
\begin{proof}
We need the following elementary fact.
\begin{lemma}\label{lem:elem-lin.alg}
Let $V$ be a finite dimensional $\kk$-vector space with a non-degenerate bilinear form $\fgfrm{\cdot}{\cdot}:V\tensor V\to \kk$. 
Assume that we have two orthogonal direct sum decompositions 
$V=V_1\oplus W_1=V_2\oplus W_2$ with respect to that form. Then $V=(V_1\cap V_2)\oplus (W_1+W_2)=(W_1\cap W_2)\oplus (V_1+V_2)$
(orthogonal direct sum decompositions).
\end{lemma}
\begin{proof}
Clearly $(V_1\cap V_2)$ is orthogonal to $W_1+W_2$ and $(W_1\cap W_2)$ is orthogonal to $V_1+V_2$.
Note that for any $v\in V_i$, $\fgfrm{v}{v}=0$ if and only if $v=0$. This implies that the sums $U_1=(V_1\cap V_2)+(W_1+W_2)$, $U_2=(W_1\cap W_2)+(V_1+V_2)$ are direct.
It remains to prove that $\dim U_1=\dim U_2=\dim V$. Since 
\begin{multline*}
\dim U_1+\dim U_2=(\dim W_1+\dim W_2-\dim W_1\cap W_2)+\dim V_1\cap V_2\\+(\dim V_1+\dim V_2-\dim V_1\cap V_2)+\dim W_1\cap W_2=2\dim V,
\end{multline*}
and $\dim U_1,\dim U_2\le \dim V$ the assertion follows.
\end{proof}
Given $\gamma\in Q^+$, let $n_i(\gamma)$ be the coefficient of~$\alpha_i$ in~$\gamma$.
For any $\gamma\in Q^+$ we have two orthogonal direct sum decompositions $U_q(\lie n^+)_\gamma=(\ker\partial_i|_{U_q(\lie n^+)_\gamma}\oplus 
U_q(\lie n^+)_{\gamma-\alpha_i}E_i)=(\ker\partial_i^{op}|_{U_q(\lie n^+)_\gamma}\oplus E_i U_q(\lie n^+)_{\gamma-\alpha_i})$.
Since $U_q(\lie n^+)_\gamma$ is finite dimensional, it follows from Lemma~\ref{lem:elem-lin.alg} that $U_q(\lie n^+)_\gamma=({}_i U\cap U_i\cap U_q(\lie n^+)_\gamma)
\oplus (U_q(\lie n^+)_{\gamma-\alpha_i}E_i+E_i U_q(\lie n^+)_{\gamma-\alpha_i})$. Then an obvious induction on~$n_i(\gamma)$ implies that 
every $x\in U_q(\lie n^+)$ can be written in $x=\sum_{r,s\ge 0} E_i^{\la r\ra} x_{rs} E_i^{\la s\ra}$, where 
$x_{rs}\in {}_i U\cap U_i$ and only finitely many of the $x_{rs}$ are non-zero.

We now prove by induction on~$n_i(\gamma)$ that if $x\in U_\ZZ(\lie n^+)_\gamma$ then $x_{rs}\in U_\ZZ(\lie n^+)\cap {}_i U\cap U_i$. 
If $n_i(\gamma)=0$ then $x=x_{00}$ and there is nothing to prove. For the inductive step, we have 
$$
\sum_{r,s\ge 0} E_i^{\la r\ra}(q_i^{-r-1} x_{r+1,s}+q_i^{-(\alpha_i^\vee,\gamma)+s+1} x_{r,s+1})E_i^{\la s\ra}
=q_i^{-\frac12(\alpha_i^\vee,\gamma)}\la 1\ra_{q_i}\partial_i(x)\in U_\ZZ(\lie n^+),
$$
where we used Lemma~\ref{lem:partial-prop}\eqref{lem:partial-prop.iii} and Corollary~\ref{cor:preserv-lattice}\eqref{cor:preserv-lattice.a}.
Then $x_{r+1,s}+q_i^{-(\alpha_i^\vee,\gamma)+r+s+2} x_{r,s+1}\in
U_\ZZ(\lie n^+)$ by the induction hypothesis. Let $s_0$ be such that $x_{rs}=0$ for all $r$ and for all $s>s_0$. It follows then that $x_{r,s_0}\in U_\ZZ(\lie n^+)$ for all $r\ge 0$.
Suppose now that $x_{rt}\in U_\ZZ(\lie n^+)$ for all $r$ and for all $s+1\le t\le s_0$. Since $x_{r,s}=-q_i^{-(\alpha_i^\vee,\gamma)+s+r+1} x_{r-1,s+1}$ it follows 
that $x_{r,s}\in U_\ZZ(\lie n^+)$ for all $r,s\ge 0$ with $r+s>0$. It remains to observe that $x_{00}=x-\sum_{r,s\ge 0,r+s>0} E_i^{\la r\ra} x_{rs} E_i^{\la s\ra}$.
\end{proof}

\begin{lemma}\label{lem:action-Jantzen}
Let $x\in U_i^\ZZ\cap U_q(\lie n^+)_\gamma$ and write $x=\sum_{r\ge \max(0,(\alpha_i^\vee,\gamma))} 
(\ul E_i^{op})^{(r)}(x_r)$ where $x_r\in {}_i U\cap U_i\cap U_q(\lie n^+)_{\gamma-r\alpha_i}$. 
Then $\binom{-(\alpha_i^\vee,\gamma-r\alpha_i)}{r}_{q_i}x_r\in U_i^\ZZ$.
\end{lemma}
\begin{proof}
The argument is by induction on~$\ell_i(x^*)$. If $\ell_i(x^*)=0$, that is $\partial_i^{op}(x)=0$, then $x=x_0$ and there is nothing to do.
If~$\ell_i(x^*)=n$ then $x_r=0$ for all $r>n$.
We have $(\partial_i^{op})^{(top)}(x)
=(\partial_i^{op})^{(n)}(x)\in U_i^\ZZ$ by Corollary~\ref{cor:preserv-lattice}\eqref{cor:preserve-lattice.b}. On the other hand,
$
(\partial_i^{op})^{(n)}(x)=(\partial_i^{op})^{(n)}(\ul E_i^{op})^{(n)}(x_{n})=\binom{2n-(\alpha_i^\vee,\gamma)}{n}_{q_i}x_{n}
$
by~\eqref{eq:lowest weight action}.
Thus, $\binom{2n-(\alpha_i^\vee,\gamma)}{n}_{q_i}x_n\in U_i^\ZZ\cap {}_i U^\ZZ$. It remains to observe that the induction hypothesis applies to $x-(\ul E_i^{op})^{(n)}(x_n)$.
\end{proof}
Thus, is suffices to consider $x=(\ul E_i^{op})^{(r)}(z)\in U_i^\ZZ$ where $z\in {}_i U\cap U_i\cap U_q(\lie n^+)_\gamma$. 
We claim that $T_i(x)=\ul E_i^{(-(\alpha_i^\vee,\gamma)-r)}(z)\in {}_i U^\ZZ$. 
Given $y\in U_\ZZ(\lie n^+)_{\gamma+(-(\alpha_i^\vee,\gamma)-r)\alpha_i}$, use Lemma~\ref{lem:Jantzen}
to write $y=\sum_{s,s'\ge 0} E_i^{\la s'\ra}y_{s's} E_i^{\la s\ra}$ with $y_{s's}\in {}_i U\cap U_i\cap U_\ZZ(\lie n^+)$. Then
by Lemma~\ref{lem:partial-prop}\eqref{lem:partial-prop.ii} and~\eqref{eq:lowest weight action}
\begin{multline*}
\fgfrm{\ul E_i^{(-(\alpha_i^\vee,\gamma)-r)}(z)}{y}=\sum_{s',s\ge 0} \fgfrm{\ul E_i^{(-(\alpha_i^\vee,\gamma)-r)}(z)}{E_i^{\la s'\ra}y_{s's} E_i^{\la s\ra}}
=\sum_{s\ge 0} \fgfrm{\partial_i^{(s)}\ul E_i^{(-(\alpha_i^\vee,\gamma)-r)}(z)}{y_{0s}}\\
=\sum_{s\ge 0} \fgfrm{\ul F_i^{(s)}\ul E_i^{(-(\alpha_i^\vee,\gamma)-r)}(z)}{y_{0s}}
=\sum_{s=0}^{-(\alpha_i^\vee,\gamma)-r} \binom{s+r}{r}_{q_i} \fgfrm{\ul E_i^{(-(\alpha_i^\vee,\gamma)-r-s)}(z)}{y_{0s}}
\\=\binom{-(\alpha_i^\vee,\gamma)}{r}_{q_i}\fgfrm{z}{y_{0,-(\alpha_i^\vee,\gamma)-r}},
\end{multline*}
since $\fgfrm{{}_i U}{E_i U_q(\lie n^+)}=0$ and
$\fgfrm{\ul E_i^{(a)}(z)}{y_{0s}}=0$ if $a>0$ by Lemma~\ref{lem:prop-ul E_i}\eqref{lem:prop-ul E_i.b}.
Since $\binom{-(\alpha_i^\vee,\gamma)}{r}_{q_i}z\in U^\ZZ(\lie n^+)$ by Lemma~\ref{lem:action-Jantzen}, it follows that $\fgfrm{\ul E_i^{(-(\alpha_i^\vee,\gamma)-r)}(z)}{y}\in\AA_0$.
\end{proof}

\subsection{Proof of Theorem~\ref{thm:T_i bas}}\label{subs:proof of thm T_i bas}
We need the following result which can also be deduced from~\cite{Lus-book}*{Proposition 38.2.1}.
However, our argument is much shorter.
\begin{lemma}\label{lem:orthogonality of T_i}
For all $x,x'\in U_q(\lie n^+)_\gamma\cap U_i$, $a,a'\in \ZZ_{\ge 0}$ we have 
$$
\fgfrm{E_i^{a}T_i(x)}{E_i^{a'}T_i(x')}=q^{\frac12(a-1)(\alpha_i,\gamma)}\delta_{a,a'}\la a\ra_{q_i}!
\fgfrm{x}{x'}
=q^{\frac12(a-1)(\alpha_i,\gamma)}\delta_{a,a'}\mu(a\alpha_i)\prod_{t=1}^a(1-q_i^{-2t})
\fgfrm{x}{x'},
$$
where~$\mu$ is defined as in Theorem~\ref{thm:sign-bas}.
\end{lemma}
\begin{proof}
It follows immediately from Lemma~\ref{lem:partial-prop}\eqref{lem:partial-prop.ii} and~\eqref{eq:partial E_i^r} that if $y,y'\in\ker\partial_i^{op}$, 
$y\in U_q(\lie n^+)_{\gamma'}$ and 
$a,a'\in\ZZ_{\ge0}$
then 
$$
\fgfrm{E_i^a y}{E_i^{a'}y'}=\delta_{a,a'}\la a\ra_{q_i}!q^{-\frac12 a(\alpha_i,\gamma')}
\fgfrm{y}{y'}=\delta_{a,a'}q_i^{\binom{a+1}2-\frac12 a(\alpha_i^\vee,\gamma')}\prod_{t=1}^a (1-q_i^{-2t})\fgfrm{y}{y'}.
$$
Let $\gamma_i=(\alpha_i^\vee,\gamma)$. Since $T_i(x),T_i(x')\in \ker\partial_i^{op}\cap U_q(\lie n^+)_{s_i\gamma}$, it remains 
to prove 
the assertion for $a=a'=0$, $x=(\ul E_i^{op})^{(a)}(z)$ and $x'=(\ul E_i^{op})^{(b)}(z')$ where $z,z'\in {}_i U\cap U_i$ and $a\ge b\ge \max(0,\gamma_i)$.
Since $z\in U_q(\lie n^+)_{\gamma-a\alpha_i}$, $z'\in U_q(\lie n^+)_{\gamma-b\alpha_i}$ 
we have, by Corollary~\ref{cor:form sl2 decomp}
$$
\fgfrm{x}{x'}=\delta_{a,b}q_i^{\frac12 a(1+a-\gamma_i)}\binom{2a-\gamma_i}{a}_{q_i}\fgfrm{z}{z'}.
$$
On the other hand, using Theorem~\ref{thm:T_i formula} 
and Corollary~\ref{cor:form sl2 decomp} we obtain
\begin{equation*}
\begin{split}
\fgfrm{T_i(x)}{T_i(y)}=\fgfrm{\ul E_i^{(a-\gamma_i)}(z)}{\ul E_i^{(b-\gamma_i)}(z')}
=q_i^{\frac12(a-\gamma_i)(a+1)}\delta_{a,b}\binom{2a-\gamma_i}{a-\gamma_i}_{q_i}\fgfrm{z}{z'}=q_i^{-\frac12\gamma_i}\fgfrm{x}{y}.\qedhere
\end{split}
\end{equation*}
\end{proof}

Let $b\in \mathbf B^{up}{}_\gamma\cap T_i^{-1}(U_q(\lie n^+))$. Since $T_i$ commutes with $\bar\cdot$
we have $\overline{T_i(b)}=T_i(\overline b)=T_i(b)$. By Proposition~\ref{prop:T_i lattice}
we have $T_i(b)\in U^\ZZ(\lie n^+)$.
Furthermore, by Lemma~\ref{lem:orthogonality of T_i} and~\eqref{eq:mu-prop}
$$
\mu(s_i\gamma)^{-1}\fgfrm{T_i(b)}{T_i(b)}=\mu(\gamma)^{-1}q^{\frac12(\alpha_i,\gamma)}\fgfrm{T_i(b)}{T_i(b)}=\mu(\gamma)^{-1}\fgfrm{b}{b}\in 1+K_-.
$$
Thus, $T_i(b)\in\mathbf B^{\pm up}$ by~\eqref{eq:def-bas}.

It remains to prove that $T_i(b)\in \mathbf B^{up}$. 
Since $(\partial_i^{op})^{(top)}(b)\in\mathbf B^{up}$ by Remark~\ref{rem: star op partial}, there exists
a sequence $\ii'=(i_1,\dots,i_{m})\in I^{m}$ such that $\partial_{\ii'}^{(top)}((\partial_i^{op})^{(top)}(b))=1$. Let 
$\ii=(i,i_1,\dots,i_m)$. Then 
$\partial_\ii^{(top)}(T_i(b))=\partial_{\ii'}^{(top)}\partial_i^{(top)}T_i(b)
=\partial_{\ii'}^{(top)}((\partial_i^{op})^{(top)}(b))
=1$ by Theorem~\ref{thm:T_i formula}. Thus, $T_i(b)\in\mathbf B^{up}$.\qed
\section{Proofs of mains results}

\subsection{Properties of quantum Schubert cells}\label{subs:some prop q S c}
Let $w\in W$ and $\mathbf i=(i_1,\dots,i_m)\in R(w)$. Set $X_{\mathbf i,k}=T_{i_1}\cdots T_{i_{k-1}}(E_{i_k})$, $1\le k\le m$,
and let $U_q(\ii)$ be the subalgebra of $U_q(\lie n^+)$
generated by the $X_{\mathbf i,k}$, $1\le k\le m$ and set $U^\ZZ(\ii)=U_q(\ii)\cap U^\ZZ(\lie n^+)$, $U_\ZZ(\ii)=U_q(\ii)\cap U_\ZZ(\lie n^+)$.
The following is well-known.
\begin{lemma}[\cite{Lus-book}*{Propositions~40.2.1 and~41.1.4}]\label{lem:lus PBW elt}
The elements $X_\ii^{\la \mathbf a\ra}:=X_{\ii,1}^{\la a_1\ra}\cdots X_{\ii,m}^{\la a_m\ra}$, $\mathbf a\in \ZZ_{\ge 0}^m$
form an $\AA_0$-basis of $U_\ZZ(\ii)$
and a $\kk$-basis of~$U_q(\ii)$. 
\end{lemma}
Set $\alpha^{(k)}=\alpha^{(k)}_\ii:=s_{i_1}\cdots s_{i_{k-1}}(\alpha_{i_k})=\deg X_{\mathbf i,k}$ and given $\mathbf a=(a_1,\dots,a_m)\in\ZZ^m$
denote $|\mathbf a|=|\mathbf a|_\ii=\sum_{k=1}^m a_k\deg X_{\ii,k}=\sum_{k=1}^m a_k\alpha^{(k)}_\ii$.
Define 
$$
X_\ii^{\mathbf a}=q_{\mathbf i,\mathbf a} X_{\ii,1}^{a_1}\cdots X_{\ii,m}^{a_m},\qquad \mathbf a=(a_1,\dots,a_m)\in\ZZ_{\ge 0}^m,
$$
where 
\begin{equation}\label{eq:q_i a defn}
q_{\mathbf i,\mathbf a}=q^{\frac12 \sum_{1\le k<l\le m}(\alpha_\ii^{(k)},\alpha_\ii^{(l)})a_ka_l}.
\end{equation}
This choice is justified by the following 
\begin{proposition}\label{prop:X i^a scalar square}
For all $\ii\in R(w)$, $\mathbf a,\mathbf a'\in \ZZ_{\ge 0}^m$ we have $X_{\ii}^{\mathbf a}\in U^\ZZ(\lie n^+)$ and 
\begin{equation}\label{eq:X i^a scalar square}
\mu(|\mathbf a|)^{-1}\fgfrm{X_\ii^{\mathbf a}}{X_\ii^{\mathbf a'}}=\delta_{\mathbf a,\mathbf a'}\prod_{r=1}^m\prod_{t=1}^{a_m}(1-q_{i_r}^{-2t}).
\end{equation}
Thus, the set $\{ X_\ii^{\mathbf a}\,:\,|\mathbf a|_\ii=\gamma\}$ is a $(K_-,\mu(\gamma)^{-1})$-orthonormal basis of~$U^\ZZ(\ii)_\gamma$
and
\begin{equation}\label{eq:dual PBW}
\fgfrm{X_\ii^{\mathbf a}}{X_\ii^{\la \mathbf a'\ra}}=\delta_{\mathbf a,\mathbf a'},\qquad \mathbf a,\mathbf a'\in\ZZ_{\ge 0}^m.
\end{equation}
\end{proposition}
\begin{proof}
We need the following 
\begin{lemma}\label{lem:T_w(E_j)}
For all $w\in W$, $j\in I$ such that $\ell(ws_j)=\ell(w)+1$ we have $T_w(E_j)\in U^\ZZ(\lie n^+)$. 
\end{lemma}
\begin{proof}
The argument is by induction on~$\ell(w)$.
If $\ell(w)=0$ there is nothing to prove. Suppose that $w=s_iw'$ with $\ell(w)=\ell(w')+1$. Clearly, $\ell(w's_j)=\ell(w')+1$.  Then 
$T_{w'}(E_j)\in \ker\partial_i$ by~\cite{Lus-book}*{Lemma~40.1.2} and also $T_{w'}(E_j)\in U^\ZZ(\lie n^+)$ by the induction hypothesis. Then by Proposition~\ref{prop:T_i lattice},
$T_w(E_j)=T_i(T_{w'}(E_j))\in U^\ZZ(\lie n^+)$.
\end{proof}
This implies that $X_{\mathbf i,k}\in U^\ZZ(\lie n^+)$ and hence $X_{\ii}^{\mathbf a}\in U^\ZZ(\lie n^+)$ by Lemma~\ref{lem:prod in U^Z}.

To prove~\eqref{eq:X i^a scalar square} we use induction on~$\ell(w)$. The case~$\ell(w)=0$ is trivial.
For the inductive step, assume that $\ell(s_iw)=\ell(w)+1$ and note that we have 
$$
X_{(i,\ii)}^{(a,\mathbf a)}=q^{\frac12 a(\alpha_i,s_i|\mathbf a|_\ii)}E_{i}^{a} T_{i}(X_{\ii}^{\mathbf a})=q^{-\frac12 a(\alpha_i,|\mathbf a|_\ii)}
E_{i}^{a} T_{i}(X_{\ii}^{\mathbf a}).
$$
Since $T_i(X_{\ii}^{\mathbf a})\in {}_i U$, we have by
Lemmata~\ref{lem:partial-prop}, \ref{lem:orthogonality of T_i} and~\eqref{eq:mu-prop}
\begin{align*}
\fgfrm{X_{(i,\ii)}^{(a,\mathbf a)}}{&X_{(i,\ii)}^{(a',\mathbf a')}}
=q^{-\frac12(a(\alpha_i,|\mathbf a|_\ii)+a'(\alpha_i,|\mathbf a'|_\ii)}\fgfrm{E_{i}^{a}T_i(X_{\ii}^{\mathbf a})}{E_{i}^{a'}T_i(X_{\ii}^{\mathbf a'})}\\
&=\delta_{a,a'}\mu(a\alpha_i) q^{-\frac12(a+1)(\alpha_i,|\mathbf a|_\ii)}\prod_{t=1}^a (1-q_i^{-2t})\fgfrm{X_{\ii}^{{\mathbf a}}}{X_{\ii}^{{\mathbf a}'}}\\
&=\delta_{a,a'}\delta_{\mathbf a,\mathbf a'}\mu(|\mathbf a|)\mu(a\alpha_i)q^{-\frac12(a+1)(\alpha_i^\vee,|\mathbf a|_\ii)}\prod_{t=1}^{a}(1-q_i^{-2t})
\prod_{r=1}^m\prod_{t=1}^{a_r}(1-q_{i_r}^{-2r})
\\&=\delta_{a,a'}\delta_{\mathbf a,\mathbf a'}\mu(|(a,\mathbf a)|_{(i,\ii)})\prod_{t=1}^{a}(1-q_i^{-2t})
\prod_{r=1}^m\prod_{t=1}^{a_r}(1-q_{i_r}^{-2r}),
\end{align*}
since $|(a,\mathbf a)|_{(i,\ii)}=a\alpha_i+s_i(|\mathbf a|_\ii)$. Finally, \eqref{eq:dual PBW} is immediate from~\eqref{eq:X i^a scalar square} and 
\eqref{eq:mu-prop}.
\end{proof}

Set 
$U^\ZZ(w)=U^\ZZ(\lie n^+)\cap U_q(w)$ where $U_q(w)$ is defined by~\eqref{eq:quantum-Schubert-cell}. 
\begin{proposition}\label{prop:U(w)-U(ii)}
For each $\ii\in R(w)$, $\{X_\ii^{\mathbf a}\}_{\mathbf a\in\ZZ_{\ge 0}^m}$ is an $\AA_0$-basis of~$U^\ZZ(w)$.
In particular, 
$U^\ZZ(w)=U^\ZZ(\ii)$ for all $w\in W$, $\ii\in R(w)$. 
\end{proposition}
\begin{proof}
Since $U_q(w)=U_q(\ii)$ by \cite{T}*{Proposition~2.10}, for any $x\in U^\ZZ(w)$ we can write
$x=\sum_{\mathbf a'} c_{\mathbf a'} X_\ii^{\mathbf a'}$ where $c_{\mathbf a'}\in\kk$. 
Since $\fgfrm{x}{X_\ii^{\mathbf a}}\in \AA_0$, it follows from~\eqref{eq:dual PBW} that $c_{\mathbf a}\in \AA_0$
for all $\mathbf a\in\ZZ_{\ge 0}^m$. Thus, the $X_\ii^{\mathbf a}$, $\mathbf a\in\ZZ_{\ge 0}^m$ generate 
$U^\ZZ(w)$ as an $\AA_0$-module. Since they are already linearly independent over~$\kk$, they form its $\AA_0$-basis.
\end{proof}

\begin{theorem} \label{thm:two-forms}
Let $w\in W$, $\ii=(i_1,\dots,i_m)\in R(w)$. Then the algebra $U^\AA(w)=U^\ZZ(w)\tensor_{\AA_0}\AA$ 
has the following presentation 
\begin{equation}\label{eq:straight-rels'}
q^{-\frac12(\alpha_\ii^{(k)},\alpha_\ii^{(l)})}X_{\mathbf i,l}X_{\mathbf i,k}-
q^{\frac12(\alpha_\ii^{(k)},\alpha_\ii^{(l)})} X_{\mathbf i,k}X_{\mathbf i,l}\in (q_{i_k}-q_{i_k}^{-1})\sum_{\mathbf a
=(0,\dots,0,a_{k+1},\dots,a_{l-1},0,\dots,0)\in\ZZ_{\ge 0}^m}\mskip-10mu \AA_0 X_{\ii}^{\mathbf a},
\end{equation}
for all $1\le k<l\le m$
\end{theorem}
\begin{proof}
We need the following
\begin{lemma}\label{lem:commutator}
Let $w'\in W$ and $i,j\in I$ be such that $\ell(s_iw's_j)=\ell(w')+2$. 
Then
\begin{equation}\label{eq:q-comm-E_i}
T_{s_iw'}(E_j)E_i-q^{-(\alpha_i,w'\alpha_j)}E_i T_{s_iw'}(E_j)\in (q_i-q_i^{-1}) T_i(U^\AA(w')).
\end{equation}
\end{lemma}
\begin{proof}
First we prove that 
\begin{equation}\label{eq:comm-F_i}
K_i[F_i,T_{w'}(E_j)]\in \la 1\ra_{q_i} U^\AA(w').
\end{equation}
Our assumption implies that $\ell(s_iw')=\ell(w')+1$, $\ell(w's_j)=\ell(w')+1$, whence $T_{w'}(E_j),T_{s_iw'}(E_j)\in U_q(\lie n^+)$
by~\cite{Lus-book}*{Lemma~40.1.2} and 
so $T_{w'}(E_j)\in\ker\partial_i$ by~\cite{Lus-book}*{Proposition~38.1.6}. Moreover, by Proposition~\ref{prop:X i^a scalar square} 
we have $T_{w'}(E_j),T_{s_iw'}(E_j)\in U^\AA(\lie n^+)$.
Then
$$
K_i[F_i,T_{w'}(E_j)]=-(1-q_i^{-2})q^{-\frac12(\alpha_i,w'\alpha_j)}\partial_i^{op}(T_{w'}(E_j))\in \la 1\ra_{q_i} U^\AA(\lie n^+),
$$
where we used \eqref{eq:partial-def}, Lemma~\ref{lem:T_w(E_j)}, Proposition~\ref{prop:T_i lattice} and Corollary~\ref{cor:preserv-lattice}\eqref{cor:preserve-lattice.b}.
On the other hand, $T_{w'}^{-1}(F_i)\in U_q(\lie n^-)$, whence $[T_{w'}^{-1}(F_i),E_j]\in U_q(\lie b^-)$. Therefore, 
$$
T_{w'}^{-1}(K_i[F_i,T_{w'}(E_j)])=T_{w'}^{-1}(K_i)[T_{w'}^{-1}(F_i),E_j]\in U_q(\lie b^-).
$$
Thus,
\begin{equation*}
K_i[F_i,T_{w'}(E_j)]\in \la 1\ra_{q_i} T_{w'}(U_q(\lie b^-))\cap U^\AA(\lie n^+)=\la 1\ra_{q_i} U^\AA(w').
\end{equation*}
This proves~\eqref{eq:comm-F_i}. Since $T_i(K_i F_i)=q_i^{-1}E_i$
sand $K_iT_{s_iw'}(E_j)K_i^{-1}=q^{-(\alpha_i,w'\alpha_j)}T_{s_iw'}(E_j)$, \eqref{eq:q-comm-E_i} follows
by applying $T_i$ to both sides of~\eqref{eq:comm-F_i}.
\end{proof}
Now we use induction on~$\ell(w)$, the induction base being trivial.
Applying $T_{i_1}\cdots T_{i_{k-1}}$ to~\eqref{eq:q-comm-E_i} with $w'=s_{i_{k+1}}\cdots s_{i_{l-1}}$, $i=i_k$, $j=i_l$ we obtain 
$$
X_lX_k-q^{-(\alpha_{i_k},s_{i_{k+1}}\cdots s_{i_{l-1}}\alpha_{i_l})}X_kX_l\in \la 1\ra_{q_{i_k}} T_{i_1}\cdots T_{i_{k}}(U^\AA(w')).
$$
By Proposition~\ref{prop:U(w)-U(ii)}, 
$U^\AA(w')$ has an $\AA$-basis 
$\{X_{\mathbf i',1}^{a_{k+1}}\dots X_{\mathbf i',l-1}^{a_{l-1}}\,:\,a_{k+1},\dots,a_{l-1}\in \ZZ_{\ge 0}\}$
where $\ii'=(i_{k+1},\dots,i_{l-1})$.
Applying $T_{i_1}\cdots T_{i_k}$
we conclude that 
 $\{X_{\mathbf i,k+1}^{a_{k+1}}\dots X_{\mathbf i,l-1}^{a_{l-1}}\,:\,a_{k+1},\dots,a_{l-1}\in \ZZ_{\ge 0}\}$ is an $\AA$-basis of $T_{i_1}\cdots T_{i_k}(U^\AA(w'))$.
Note that $(\alpha_{i_k},s_{i_{k+1}}\cdots s_{i_{l-1}}\alpha_{i_l})=-(\alpha_\ii^{(k)},\alpha_\ii^{(l)})$ and so we can write
$$
q^{-\frac12(\alpha_\ii^{(k)},\alpha_\ii^{(l)})}X_{\mathbf i,l}X_{\mathbf i,k}-
q^{\frac12(\alpha_\ii^{(k)},\alpha_\ii^{(l)})} X_{\mathbf i,k}X_{\mathbf i,l}\in \sum_{\mathbf a
=(0,\dots,0,a_{k+1},\dots,a_{l-1},0,\dots,0)\in\ZZ_{\ge 0}^m} \la 1\ra_{q_{i_k}}c_{\mathbf a} X_{\ii}^{\mathbf a},
$$
where $c_{\mathbf a}\in \AA$. Repeating the argument from the proof of Proposition~\ref{prop:U(w)-U(ii)}
we conclude that 
$\la 1\ra_{q_{i_k}}c_{\mathbf a}\in \AA_0$. Thus, $c_{\mathbf a}\in \AA\cap (\la 1\ra_{q_{i_k}})^{-1}\AA_0=\AA_0$.
Since relations~\eqref{eq:straight-rels'} imply that $U^\AA(w)$ is generated, as an $\AA$-module, by the $X_\ii^{\mathbf a}$, $\mathbf a\in\ZZ_{\ge 0}^m$,
it follows that~\eqref{eq:straight-rels'} is a presentation.
\end{proof}
\begin{remark}\label{rem:2}
 Let $A(w)$ be the $\ZZ$-algebra defined by $A(w)=U^\ZZ(w)/(q-1)U^\ZZ(w)$. Clearly, $A(w)$ is commutative and identifies with the coordinate 
 algebra $\ZZ[U(w)]$, where $U(w)=U\cap w(U^-)w^{-1}$ is the Schubert cell in the maximal unipotent subgroup~$U$ of the Kac-Moody group~$G$ corresponding 
 to~$\lie g$. This justifies~\eqref{eq:quantum-Schubert-cell} and the name quantum Schubert cell used for~$U_q(w)$.
\end{remark}

\subsection{Lusztig's Lemma and proof of Theorem~\ref{thm:bas-from-Lusztig-lemma}}\label{subs:pf-bas-from-Lus-lemma}
Let $w\in W$, $\ii=(i_1,\dots,i_m)\in R(w)$ and let $\mathbf e_1,\dots,\mathbf e_m$ be the standard basis of~$\ZZ^m$.
For each pair $1\le k<l\le m$ with $k+1\le l-1$ define $\mathcal A_{k,l}=\mathcal A_{k,l}(\ii)$ to be the finite set of all tuples
$(a_{k+1},\dots,a_{l-1})$ such that $X_{\mathbf i,k+1}^{a_{k+1}}\cdots X_{\mathbf i,l-1}^{a_{l-1}}$ occurs in the right hand side 
of~\eqref{eq:straight-rels'} with a non-zero coefficient.
Let 
$C_{\mathbf i}$ be the submonoid of $\ZZ^m$ generated by elements 
$$
\mathbf e_k+\mathbf e_{l}-\sum_{r=k+1}^{l-1} a_r \mathbf e_r
$$
for all $1\le k,l\le m$ with $k+1\le l-1$ such that $\mathcal A_{kl}\not=\emptyset$ and for all $(a_{k+1},\dots,a_{l-1})\in \mathcal A_{kl}$.
\begin{proposition}\label{prop:order}
$C_{\mathbf i}$ is pointed, that is, if $\mathbf x,-\mathbf x\in C_{\mathbf i}$ then $\mathbf x=0$. In particular,
the relation $\prec$ on $\mathbb Z^m_{\ge 0}$ defined by 
$$
\mathbf a\preceq \mathbf a'\iff \mathbf a'-\mathbf a\in C_{\mathbf i}
$$
is a partial order.
\end{proposition}
\begin{proof}
The first assertion is a special case of the following
\begin{lemma}
For each $k<l$ fix $\mathcal A_{k,l}\subset \Big(\bigoplus\limits_{i=k+1}^{l-1} \ZZ_{\ge 0}\mathbf e_i\Big)\setminus\{0\}$. Let $\Gamma$
be the submonoid of $\ZZ^m$ generated by all elements of the form $\mathbf e_k+\mathbf e_l-\mathbf a$,
$\mathbf a\in\mathcal A_{k,l}$ for all $k<l$ such that $\mathcal A_{k,l}\not=\emptyset$.
Then $\Gamma$ is pointed.
\end{lemma}
\begin{proof}
Let $\mathbf y=\sum_{k<l}\sum_{\mathbf a\in\mathcal A_{k,l}} n_{k,l,\mathbf a_{k,l}}(\mathbf e_k+\mathbf e_l-\mathbf a_{k,l})$
where $n_{k,l,\mathbf a_{k,l}}\in\ZZ_{\ge 0}$ and are not all zero. Let $k$ be minimal such that $n_{k,l,\mathbf a}\not=0$
for some $l>k$, $\mathbf a\in\mathcal A_{k,l}$. Then the coefficient of~$\mathbf e_k$ in~$\mathbf y$ is positive. This immediately
implies that $0$ admits a unique presentation in~$\Gamma$.
\end{proof}
To prove the second assertion, note that the relation~$\prec$ is clearly transitive. Furthermore, if $\mathbf a'\prec \mathbf a$
and $\mathbf a\prec\mathbf a'$ then $\mathbf a'-\mathbf a,\mathbf a-\mathbf a'\in C_{\mathbf i}$ which implies that $\mathbf a=\mathbf a'$.
\end{proof}

Since $T_w$ commutes with $\bar\cdot$-anti-involution, $\overline{U_q(w)}=U_q(w)$ and $\overline{X_{\ii,k}}=X_{\ii,k}$. Since also $\overline{U^\ZZ(\lie n^+)}=U^\ZZ(\lie n^+)$,
it follows that $\overline{U^\ZZ(w)}=U^\ZZ(w)$. Thus, the restriction of $\bar\cdot$ to $U^\ZZ(\ii)$ is the unique anti-linear anti-involution of that algebra 
fixing its generators $X_{\ii,k}$.

Note that for each $\gamma\in Q^+$, the set $\{ \mathbf a\in\ZZ_{\ge 0}^m\,:\, |\mathbf a|_\ii=\gamma\}$ is finite.
The following result is crucial for the proof of Theorem~\ref{thm:bas-from-Lusztig-lemma}.
\begin{proposition}\label{prop:initial basis}
For all $\mathbf a\in\ZZ_{\ge 0}^m$ we have 
$
\overline{X_\ii^{\mathbf a}}-X_\ii^{\mathbf a}\in\sum_{\mathbf a'\prec\mathbf a} \AA_0 X_\ii^{\mathbf a'}$.
\end{proposition}

\begin{proof}   
We need some notation. Let 
$\mathcal U=U^\ZZ(\ii)$ and let $\mathcal I=[1,m]$. Let $\mathcal B$ be the set of all finite non-decreasing sequences in~$\mathcal I$. 
Given a sequence ${\mathbf k}=(k_1,\dots,k_N)\in \mathcal I^N$, let $\mathbf e_{\mathbf k}=\sum_{r=1}^N \mathbf e_{k_r}$
and define 
$$ 
X({\mathbf k})=q^{\frac12\sum_{1\le r<s\le N} \limits\sign (k_s-k_r)(\alpha^{(k_r)},\alpha^{(k_s)})} X_{\ii,k_1}\cdots X_{\ii,k_N}.
$$
In particular, if ${\mathbf k}=(k_1,\dots,k_N)\in\mathcal B$ and $a_k=\#\{ 1\le r\le N\,:\, k_r=k\}$ then 
$X({\mathbf k})=X_\ii^{\mathbf a}$. 
Given $\mathbf a\in\ZZ_{\ge 0}^m$, set 
$$
\mathcal U_{\prec \mathbf a}=\sum_{\mathbf a'\prec\mathbf a}\AA_0 \cdot X_\ii^{\mathbf a'}=\sum_{{\mathbf k}\in\mathcal B\,:\, \mathbf e_{\mathbf k}\prec \mathbf a} \AA_0
\cdot X({\mathbf k}),
\qquad 
\mathcal U'_{\prec \mathbf a}:=\sum_{N\ge 0,\,{\mathbf k}\in \mathcal I^N:\mathbf e_{\mathbf k}\prec\mathbf a} \AA_0\cdot X({\mathbf k})
$$
with the convention that $\mathcal U_{\prec\mathbf a}={\mathcal U}'_{\prec \mathbf a}=\{0\}$ if $\mathbf a$ is minimal with respect to $\prec$.
Clearly, both are increasing filtration on $\mathcal U$.
Note following immediate 
\begin{lemma}\label{lem:finite poset}
If $\mathbf a'\prec \mathbf a$, $\mathbf a,\mathbf a'\in\ZZ_{\ge 0}^m$ then $|\mathbf a'|_\ii=|\mathbf a|_\ii$. In particular,
$\mathcal U_{\prec \mathbf a}$, $\mathcal U'_{\prec \mathbf a}$ are finite dimensional.
\end{lemma}

\begin{lemma}
\label{lem:recursive base transition}  For any  sequence ${\mathbf k}=(k_1,\ldots,k_N) \in \mathcal I^N$, $N\ge 0$
and for any $\sigma\in S_N$ we have 
\begin{equation}
\label{eq:monomial permutation}
X(\sigma({\mathbf k}))-X({\mathbf k})\in  {\mathcal U}'_{\prec \mathbf e_{\mathbf k}}, 
\end{equation}
where $\sigma({\mathbf k})=(k_{\sigma(1)},\dots,k_{\sigma(N)})$.
\end{lemma}

\begin{proof} Clearly, it suffices to prove the assertion for a transposition~$\sigma=(r,r+1)$. 
Without loss of generality we may assume that $k_r<k_{r+1}$. Let ${\mathbf k}_r^-=(k_1,\ldots,k_{r-1})$, ${\mathbf k}_r^+=(k_{r+2},\ldots,k_N)$. 
Then the relation \eqref{eq:straight-rels'} taken with $k=k_r$, $l=k_{r+1}$ implies
\begin{equation}
\label{eq:transposition monomial}
X_{\sigma({\mathbf k})}=X_{{({\mathbf k}_r^-,k_{r+1},k_r,{\mathbf k}_r^+)}}=
X_{\mathbf k}+\sum_{{\mathbf k}'\in {\mathcal B}\,:\,\mathbf e_{{\mathbf k}'}\prec \mathbf e_{i_r}+\mathbf e_{i_{r+1}}} 
c_{{\mathbf k}'} X_{({\mathbf k}_r^-,{\mathbf k}',{\mathbf k}_r^+)},\qquad c_{{\mathbf k}'}\in \AA_0.
\end{equation}
Clearly, $\mathbf e_{({\mathbf k}_r^-,{\mathbf k}',{\mathbf k}_r^+)}=\mathbf e_{{\mathbf k}_r^-}+\mathbf e_{{\mathbf k}'}+\mathbf e_{{\mathbf k}_r^+}\prec 
\mathbf e_{\mathbf k}$ for all ${\mathbf k}'\in\mathcal B$ such that $\mathbf e_{{\mathbf k}'}\prec \mathbf e_{k_r}+
\mathbf e_{k_{r+1}}$. 
This implies that each $X_{({\mathbf k}_r^-,{\mathbf k}',{\mathbf k}_r^+)}$ in the right hand side of 
\eqref{eq:transposition monomial} belongs to ${\mathcal U}'_{\prec e_{\mathbf k}}$ and we obtain \eqref{eq:monomial permutation} for $\sigma=(r,r+1)$.
\end{proof}

\begin{lemma} \label{lem:decresing filtration A strict}
${\mathcal U}_{\prec \mathbf a}={\mathcal U}'_{\prec\mathbf a}$ for all $\mathbf a\in\ZZ_{\ge 0}^m$.
\end{lemma}

\begin{proof}
The inclusion ${\mathcal U}_{\prec \mathbf a}\subseteq {\mathcal U}'_{\prec \mathbf a}$ is obvious. To prove the opposite inclusion,
we use induction on the partial order $\prec$ which is applicable since 
$\{\mathbf a'\in\ZZ_{\ge 0}^m\,:\,\mathbf a'\prec\mathbf a\}$ is finite for all~$\mathbf a\in\ZZ_{\ge 0}^m$. 

If $\mathbf a\in \ZZ_{\ge 0}^m$ is minimal with respect to $\prec$, then ${\mathcal U}_{\prec \mathbf a}=\{0\}$ 
and we have nothing to prove. 
Assume now that $\mathbf a$ is not minimal. Then for each ${\mathbf k}\in \mathcal I^N$, $N\ge 0$ such that $\mathbf  e_{\mathbf k}\prec \mathbf a$ we have 
${\mathcal U}'_{\prec \mathbf e_{\mathbf k}}={\mathcal U}_{\prec \mathbf e_{\mathbf k}}$
by the induction hypothesis.

Using this and Lemma \ref{lem:recursive base transition}, we conclude that for any $\sigma\in S_N$
$$X({\mathbf k})-X(\sigma({\mathbf k}))\in  {\mathcal U}_{\prec \mathbf e_{{\mathbf k}}}.$$

Taking $\sigma$ such that $\sigma({\mathbf k})\in {\mathcal B}$, that is, is non-decreasing, implies that 
$X({\mathbf k})\in {\mathcal U}_{\prec \mathbf a}$. 
\end{proof}

Combining Lemmata~\ref{lem:recursive base transition} and \ref{lem:decresing filtration A strict} we obtain the following obvious corollary:
\begin{corollary}\label{cor:perm-bas}
For any ${\mathbf k}\in {\mathcal B}$ and any $\sigma\in S_N$, we have $
X(\sigma({\mathbf k}))-X({\mathbf k})\in {\mathcal U}_{\prec \mathbf e_{\mathbf k}}$.
\end{corollary}
Note that $\overline{X(\mathbf k)}=X(\mathbf k^{op})$ for any $\mathbf k\in\mathcal I^N$ where $\mathbf k^{op}$ is~$\mathbf k$ written in the reverse order,
and $X(\mathbf k)=X_\ii^{\mathbf e_\mathbf k}$ for $\mathbf k\in\mathcal B$. Since for any $\mathbf a\in\ZZ_{\ge 0}$ there exists 
a unique $\mathbf k\in\mathcal B$ such that $\mathbf e_\mathbf k=\mathbf a$, these observations together with
the above Corollary complete the proof Proposition~\ref{prop:initial basis}.
\end{proof}

Proposition~\ref{prop:initial basis} implies that for each $\gamma\in Q^+$
the assumptions of~\cite{BZ}*{Theorem~1.1} with ${(L,\prec)}=(\{\mathbf a\in\ZZ_{\ge 0}^{m}\,:\,|\mathbf a|_\ii=\gamma\},\prec)$ and $v=q^{-1}$ are satisfied. The assertion of Theorem~\ref{thm:bas-from-Lusztig-lemma}
now follows.
\qed

Note the following useful fact, which is immediate from the proof of Proposition~\ref{prop:initial basis}.
\begin{corollary}\label{cor:skew-symm form}
Define $\Lambda=\Lambda_\ii:\ZZ^m\tensor_\ZZ \ZZ^m\to\ZZ$ by $\Lambda(\mathbf e_k,\mathbf e_l)=\sign(l-k)(\alpha_\ii^{(k)},\alpha_\ii^{(l)})$. Then 
for all $\mathbf a,\mathbf a'\in\ZZ_{\ge 0}^m$
\begin{gather*}
X_\ii^{\mathbf a}X_\ii^{\mathbf b}-q^{-\frac12\Lambda(\mathbf a,\mathbf b)}X_\ii^{\mathbf a+\mathbf b}\in
q^{-\frac12\Lambda(\mathbf a,\mathbf b)}\sum_{\mathbf a'\prec\mathbf a+\mathbf b} \AA_0 X_\ii^{\mathbf a'}
\\
X_\ii^{\mathbf b}X_\ii^{\mathbf a}-q^{\Lambda(\mathbf a,\mathbf b)}X_\ii^{\mathbf a}X_\ii^{\mathbf b}\in 
q^{\frac12\Lambda(\mathbf a,\mathbf b)}\sum_{\mathbf a'\prec\mathbf a+\mathbf b} \AA_0 X_\ii^{\mathbf a'}.
\end{gather*}
\end{corollary}
We note an obvious property of~$\Lambda$ which will be used in the sequel.
\begin{lemma}\label{lem:Lambda pairing}
For any $1\le k\le m$, $\mathbf a=(a_1,\dots,a_m)\in\ZZ^m$ 
we have 
$
\Lambda_\ii(\mathbf e_k,\mathbf a)=(\alpha_\ii^{(k)},|\mathbf a_{>k}|_\ii-|\mathbf a_{<k}|_\ii)
$,
where $\mathbf a_{<k}=\sum_{t=1}^{k-1} a_t \mathbf e_t$, $\mathbf a_{>k}=\sum_{t=k+1}^{m} a_t \mathbf e_t$.
\end{lemma}

\subsection{Containment of \texorpdfstring{$\mathbf B(\ii)$}{B(i)} in~\texorpdfstring{$\mathbf B^{up}$}{Bup} and proof of Theorem~\ref{thm:inclusion of base}}
\label{subs:pf inclusion of base}
Let $w\in W$, $\mathbf i=(i_1,\dots,i_m)\in R(w)$.
Let $\gamma=|\mathbf a|_\ii$.
Since $b_{\mathbf i,\mathbf a}\in X_{\mathbf i}^{\mathbf a}+\sum\limits_{\mathbf a'\not=\mathbf a,\,|\mathbf a'|_\ii=\gamma}K_- X_{\ii}^{\mathbf a'}$,
it follows from~\eqref{eq:X i^a scalar square} that 
$$
\mu(\gamma)^{-1}\fgfrm{b_{\ii,\mathbf a}}{b_{\ii,\mathbf a}}\in \mu(\gamma)^{-1}\fgfrm{X_{\ii}^{\mathbf a}}{X_{\ii}^{\mathbf a}}
+\sum_{\mathbf a'\in\ZZ_{\ge 0}^m\setminus\{\mathbf a\}\,:\,|\mathbf a'|=\gamma} K_-\mu(\gamma)^{-1}\fgfrm{X_{\ii}^{\mathbf a'}}{X_{\ii}^{\mathbf a'}}
\in 1+K_-.
$$
Since $b_{\mathbf i,\mathbf a}\in U^\ZZ(\lie n^+)$ and 
$\overline{b_{\mathbf i,\mathbf a}}=b_{\mathbf i,\mathbf a}$, it follows from~\eqref{eq:def-bas} that $b_{\ii,\mathbf a}\in \mathbf B^{\pm up}$.

To prove that~$b_{\ii,\mathbf a}\in\mathbf B^{up}$, we use induction on~$m$. The induction base is trivial. 
For the inductive step, write $X_\ii^{\mathbf a}=\sum_{b\in \mathbf B^{up}} c_{\mathbf a,b} b$.
Since $\pm b_{\ii,\mathbf a}\in \mathbf B^{up}$, it follows that $c_{\mathbf a,b}\in K_-$ for all $b\not=b_0=\pm b_{\ii,\mathbf a}$
and $c_{\mathbf a,b_0}=\pm 1$. Thus, we only need to prove that $c_{\mathbf a,b_0}=1$ for some $b_0\in \mathbf B^{up}$.

Let $i=i_1$ and $a=a_1$. 
Since 
$
X_\ii^{\mathbf a}=
q^{-\frac12a(\alpha_{i},|\mathbf a'|_{\ii'})}E_{i}^{a}T_{i}(X_{\ii'}^{\mathbf a'})
$
where $\ii'=(i_2,\dots,i_m)$, $\mathbf a'=(a_2,\dots,a_m)$, $T_{i}(X_{\ii'}^{\mathbf a'})\in\ker\partial_i^{op}$ and $(\partial_i^{op})^{(top)}(E^a)
=(\partial_i^{op})^{(a)}(E^a)=1$, we have 
$$
(\partial_i^{op})^{(top)}(X_\ii^{\mathbf a})=(\partial_i^{op})^{(a)}(X_\ii^{\mathbf a})=T_i(X_{\ii'}^{\mathbf a'})
=\sum_{b\in \mathbf B^{up}\,:\,\ell_i(b^*)=a} c_{\mathbf a,b} (\partial_i^{op})^{(top)}(b),
$$
where we used Corollary~\ref{cor:top-decomp-B up}.
Since $T_i^{-1}((\partial_i^{op})^{(top)}(b))\in\mathbf B^{up}$ for any $b\in\mathbf B^{up}$ by Theorem~\ref{thm:T_i bas}, we obtain 
from the above that
$$
X_{\ii'}^{\mathbf a'}=T_i^{-1}((\partial_i^{op})^{(top)}(X_\ii^{\mathbf a}))
=\sum_{b\in \mathbf B^{up}\,:\, \ell_i(b^*)=a } c_{\mathbf a,b} T_i^{-1}((\partial_i^{op})^{(top)}(b))
$$
is the decomposition of~$X_{\ii'}^{\mathbf a'}$ with respect to~$\mathbf B^{up}$. By 
the induction hypothesis, $b_{\ii',\mathbf a''}\in\mathbf B^{up}$ for all $\mathbf a''\in\ZZ_{\ge0}^{m-1}$ and 
therefore precisely one of the $c_{\mathbf a,b}$, $\ell_{i}(b^*)=a$ is not in~$K_-$ and is equal to~$1$. 
\qed
\begin{remark}
Note that for any $w\in W$, $\ii\in R(w)$, $1\le k\le \ell(w)$ and $a\ge 0$ we have $X_{\ii,k}^a\in\mathbf B^{up}$.
\end{remark}

\subsection{Embeddings of bases and proof of Theorem~\ref{thm:T_w(B(w'))}}\label{subs:pf T_w(B(w'))}
Note that $U_q(w)\subset U_q(ww')$. Since $\mathbf B(w)=U_q(w)\cap \mathbf B^{up}$ and $\mathbf B(ww')=U_q(ww')\cap \mathbf B^{up}$,
the first assertion follows.
To establish the second assertion, it suffices to prove that for $i\in I$ such that $\ell(s_iw)=\ell(w)+1$ we have $T_i(\mathbf B(w))\subset\mathbf B(s_iw)$. The 
assumption implies that $T_i(\mathbf B(w))\subset U_q(\lie n^+)$ and therefore is contained in~$\mathbf B^{up}$ by Theorem~\ref{thm:T_i bas}.
Since $T_i(U_q(w))\subset U_q(s_i w)$, it follows that $T_i(\mathbf B(w))\subset U_q(s_i w)\cap \mathbf B^{up}=\mathbf B(s_i w)$.\qed

\section{Examples}\label{sec:examples}
In this section we compute bases $\mathbf B(w)$ for various Schubert cells $U_q(w)$. 
We denote by $E_{i_1^{a_1}\cdots i_r^{a_r}}$ the unique element~$b$ of~$\mathbf B^{up}$ for which 
$\partial_{\ii}^{(top)}(b)=\partial_{i_r}^{(a_r)}\cdots\partial_{i_1}^{(a_1)}(b)=1$ where $\ii=(i_1,\dots,i_r)$.
Note that this element also satisfies $(\partial_{\ii^{op}}^{op})^{(top)}(b)=(\partial_{i_1}^{op})^{(a_1)}\cdots (\partial_{i_r}^{op})^{(a_r)}(b)=1$.
We use the notation from~\S\ref{subs:pf-bas-from-Lus-lemma}.

\subsection{Bases for repetition free elements}\label{subs:rep-free}
We say that $w\in W$ is repetition-free if $w=s_{i_1}\dots s_{i_m}$ where $\ii=(i_1,\dots,i_m)\in R(w)$ is repetition free. Clearly,
if $w$ is repetition free then so is each $\ii\in R(w)$.
Such an 
element is called a {\em Coxeter element} if $\ell(w)=|I|$, that is, any $\ii\in R(w)$ is an ordering of~$I$.
\begin{lemma}\label{lem:Coxeter}
Let $w\in W$ be repetition free and let $\ii \in R(w)$. Then in the notation of~\S\ref{subs:some prop q S c}:
\begin{enumerate}[{\rm(a)}]
 \item\label{lem:Coxeter.a} $U_q(w)$ is a quantum plane of rank~$\ell(w)$ 
 with presentation 
\begin{equation}\label{eq:Coxeter.pres}
q^{-\frac12(\alpha_\ii^{(k)},\alpha_\ii^{(l)})}X_{\ii,l}X_{\ii,k}=q^{\frac12(\alpha_\ii^{(k)},\alpha_\ii^{(l)})}X_{\ii,k}X_{\ii,l},\qquad 1\le k<l\le \ell(w).
\end{equation}
\item\label{lem:Coxeter.b} $\mathbf B(w)=\{ X_\ii^{\mathbf a}\,:\,\mathbf a\in\ZZ_{\ge 0}^{\ell(w)}\}$.
\item\label{lem:Coxeter.c}
$X_{\ii,k}=E_{i_1^{m_{k1}}\cdots i_{k-1}^{m_{k,k-1}}i_k}=\ul E_{i_1}^{(m_{k1})}\cdots \ul E_{i_{k-1}}^{(m_{k,k-1})}(E_{i_k})$ where $m_{kr}=
-(\alpha_{i_r}^\vee,s_{i_{r+1}}\cdots s_{i_{k-1}}(\alpha_{i_k}))=d_{i_r}^{-1}(\alpha_\ii^{(k)},\alpha_\ii^{(r)})$.
\end{enumerate}
\end{lemma}
\begin{proof}
Note that the coefficient of $\alpha_{i_k}$ in every
element of the submonoid of $Q^+$ generated by $\alpha_\ii^{(r)}$, $k<r<l$ is zero. Since the algebra $U_q(w)$ is $Q^+$-graded,
it follows that the right hand side of~\eqref{eq:straight-rels'} is zero. This proves part~\eqref{lem:Coxeter.a}.
In particular, it follows that $\overline{X_\ii^{\mathbf a}}=X_{\ii}^{\mathbf a}$ for all $\mathbf a\in\ZZ_{\ge 0}^{\ell(w)}$,
hence $b_{\ii,\mathbf a}=X_{\ii}^{\mathbf a}$. To prove~\eqref{lem:Coxeter.b} it remains to apply Theorems~\ref{thm:bas-from-Lusztig-lemma} and~\ref{thm:inclusion of base}.
To prove part~\eqref{lem:Coxeter.c}, let $u_r=T_{i_{r+1}}\cdots T_{i_{k-1}}(E_{i_k})$ and observe that 
the coefficient of~$\alpha_{i_r}$ in~$\deg u_r=s_{i_{r+1}}\cdots s_{i_{k-1}}(\alpha_{i_k})$ is zero if~$\ii$ is 
repetition free. Therefore, $u_r\in {}_{i_r} U\cap U_{i_r}$, $T_{i_r}(u_r)=\ul E_i^{(-(\alpha_{i_r}^\vee,\deg u_r))}(u_r)$ by Theorem~\ref{thm:T_i formula}
and so $\ell_i(T_{i_r}(u_r))=-(\alpha_{i_r}^\vee,\deg u_r)$. The assertion 
now follows by induction on $k-r$.
\end{proof}
\begin{remark}
The assertion of Lemma~\ref{lem:Coxeter}\eqref{lem:Coxeter.a} holds for any $w\in W$, $\ii\in R(w)$ and $1\le k<l\le \ell(w)$ such that 
the subsequence $(i_k,\dots,i_l)$ is repetition free.
\end{remark}

\subsection{Bases for elements with a single repetition}\label{subs:single rep}
We say that $w\in W$ is an element with a single repetition if there exists 
$\ii=(i_1,\dots,i_m)\in R(w)$ with $i_k\not=i_l$, $k<l$ unless $k=r$ and $l=r'$ for some $1\le r<r'\le m$. 
Clearly, all $\ii'\in R(w)$ have that property.
\begin{proposition}\label{prop:single repetition presentation}
Let $w\in W$ be an element with a single repetition and let $\ii=(i_1,\dots,i_m)\in R(w)$,
where the $i_k$, $k\not=r,r'$, $1\le k\le m$ are distinct and $i_r=i_{r'}=i$, $1\le r<r'\le m$. Then 
$U_q(w)$ is generated by the $X_{\ii,k}$, $1\le k\le m$ where 
\begin{equation}\label{eq:string-nomenclature-single rep}
X_{\ii,k}=\begin{cases}
           E_{i_1^{m_{k,1}}\cdots i_{k-1}^{m_{k,k-1}}i_k},&k\not=r'\\
           E_{i_1^{m_{r',1}}\cdots i_{r-1}^{m_{r',r-1}}i_r^{1+m_{r',r}}\cdots i_{r'-1}^{m_{r',r'-1}}},&k=r'
          \end{cases}
\end{equation}
with 
$m_{kl}=
-(\alpha_{i_l}^\vee,s_{i_{l+1}}\cdots s_{i_{l-1}}(\alpha_{i_l}))=d_{i_l}^{-1}(\alpha_\ii^{(k)},\alpha_\ii^{(l)})$,
subject to the relations 
\begin{equation}\label{eq:single repetition presentation}
\begin{aligned}
&q^{-\frac12(\alpha_\ii^{(k)},\alpha_\ii^{(l)})}X_{\ii,l}X_{\ii,k}=q^{\frac12(\alpha_\ii^{(k)},\alpha_\ii^{(l)})}X_{\ii,k}X_{\ii,l},\qquad 1\le k<l\le \ell(w),
\, k\not=r,\,l\not=r'
\\
&q^{-\frac12(\alpha_\ii^{(r')},\alpha_\ii^{(r)})}X_{\mathbf i,r'}X_{\mathbf i,r}=
q^{\frac12(\alpha_\ii^{(r)},\alpha_\ii^{(r')})} X_{\mathbf i,r}X_{\mathbf i,r'}+(q_{i}-q_{i}^{-1})X_\ii^{\mathbf n(r,r')},
\quad \mathbf n(r,r')=-\sum_{k=r+1}^{r'-1} a_{i_ki}\mathbf e_k.
\end{aligned}
\end{equation}
\end{proposition}
\begin{proof}
Clearly, the sequences $(i_1,\dots,i_{r'-1})$ and $(i_{r+1},\dots,i_m)$ are repetition free. 
In particular, for $1\le k\le r'-1$ we have 
$X_{\ii,k}=E_{i_1^{m_{k1}}\cdots i_{k-1}^{m_{k,k-1}}i_k}$ by Lemma~\ref{lem:Coxeter}\eqref{lem:Coxeter.c}.
Furthermore,
$$
X_{\ii,r'}=T_{i_1}\cdots T_{i_{r-1}}T_iT_{i_{r+1}}\cdots T_{i_{r'-1}}(E_i)=T_{i_1}\cdots T_{i_{r-1}}T_i( E_{i_{r+1}^{m_{r',r+1}}\cdots i_{r'-1}^{m_{r',r'-1}}i})
$$
where $i=i_r=i_{r'}$. Clearly, $(\partial_i^{op})^2(E_{i_{r+1}^{m_{r',r+1}}\cdots i_{r'-1}^{m_{r',r'-1}}i})=0$, hence 
$$
E_{i_{r+1}^{m_{r',r+1}}\cdots i_{r'-1}^{m_{r',r'-1}}i}=(2+m_{r',r})_{q_i}{}^{-1}\ul E_i^{op}(E_{i_{r+1}^{m_{r',r+1}}\cdots i_{r'-1}^{m_{r',r'-1}}})+
x_0$$ 
where $x_0\in {}_i U\cap U_i$. This implies that 
$$
T_i(E_{i_{r+1}^{m_{r',r+1}}\cdots i_{r'-1}^{m_{r',r'-1}}i})=(2+m_{r',r})_{q_i}^{-1}
\ul E_i^{(1+m_{r',r})}(E_{i_{r+1}^{m_{r',r+1}}\cdots i_{r'-1}^{m_{r',r'-1}}})
+\ul E_i^{(m_{r',r})}(x_0),
$$
and so $T_i(E_{i_{r+1}^{m_{r',r+1}}\cdots i_{r'-1}^{m_{r',r'-1}}i})=E_{i^{1+m_{r',r}} i_{r+1}^{m_{r',r+1}}\cdots i_{r'-1}^{m_{r',r'-1}}}$, whence 
$$
X_{\ii,r'}=E_{i_1^{m_{r',1}}\cdots i_{r-1}^{m_{r',r-1}}i^{1+m_{r',r}} i_{r+1}^{m_{r',r+1}}\cdots i_{r'-1}^{m_{r',r'-1}}}.
$$
Since
the sequence $(i_{r+1},\dots,i_k)$, $r'+1\le k\le m$ is repetition free, we have
$T_{i_{r+1}}\cdots T_{i_{k-1}}(E_{i_k})=E_{i_{r+1}^{m_{k,r+1}}\cdots i_{k-1}^{m_{k,k-1}}i_k}$
by Lemma~\ref{lem:Coxeter}\eqref{lem:Coxeter.c} and hence is in~${}_i U\cap U_i$. Then 
$X_{\ii,k}=E_{i_1^{m_{k,1}}\cdots i_{k-1}^{m_{k,k-1}}i_k}$ by Theorem~\ref{thm:T_i formula}.
This proves~\eqref{eq:string-nomenclature-single rep}. The first identity in~\eqref{eq:single repetition presentation} 
is proved similarly to~\eqref{eq:Coxeter.pres}. To prove the second, we need the following combinatorial fact similar to~\cite{BR}*{Lemma~4.8}.
\begin{lemma}\label{lem:root-comb}
Let $w\in W$ and suppose that $\ii=(i_1,\dots,i_m)\in R(w)$ has a single repetition $i_r=i_{r'}=i$. Then   
\begin{equation}\label{eq:deg-eq}
\alpha_\ii^{(r)}+\alpha_\ii^{(r')}=-\sum_{k=r+1}^{r'-1} a_{i_k i}\alpha_\ii^{(k)}
\end{equation}
and any proper subset of $\{\alpha_\ii^{(k)}\}_{r\le k\le r'}$ is linearly independent.
\end{lemma}
\begin{proof}
Fix $r<k\le r'$. Then 
\begin{multline}\label{eq:telescope.a}
-\sum_{t=r+1}^{k-1} a_{i_t,i}\alpha_\ii^{(t)}=-\sum_{t=r+1}^{k-1} (\alpha_{i_t}^\vee,\alpha_i)s_{i_1}\cdots s_{i_{t-1}}(\alpha_{i_t})
=\sum_{t=r+1}^{k-1} s_{i_1}\cdots s_{i_{t-1}}(s_{i_t}(\alpha_i)-\alpha_i)
\\
=s_{i_1}\cdots s_{i_{k-1}}(\alpha_i)-s_{i_1}\cdots s_{i_r}(\alpha_i)
=s_{i_1}\cdots s_{i_{k-1}}(\alpha_i)+\alpha_\ii^{(r)}.
\end{multline}
The first assertion of the Lemma is now immediate. To prove the second, suppose that 
$\sum_{t=r}^{r'} c_t \alpha_\ii^{(t)}=0$. Using~\eqref{eq:deg-eq} we may assume that~$c_{r'}=0$ and let $r< k<r'$ be maximal 
such that $c_k\not=0$. Then $\alpha_{i_k}$ occurs with coefficient~$1$
in~$\alpha_\ii^{(k)}$ and does not occur in~$\alpha_\ii^{(t)}$ with~$t<k$, whence $c_k=0$ which contradicts with the choice of~$k$. 
\end{proof}
It follows from~\eqref{eq:straight-rels'} and Lemma~\ref{lem:root-comb} that 
\begin{equation}\label{eq:tmp-determ-c}
q^{-\frac12(\alpha_\ii^{(r')},\alpha_\ii^{(r)})}X_{\mathbf i,r'}X_{\mathbf i,r}-
q^{\frac12(\alpha_\ii^{(r)},\alpha_\ii^{(r')})} X_{\mathbf i,r}X_{\mathbf i,r'}=(q_{i}-q_{i}^{-1})c X_\ii^{\mathbf n(r,r')},
\end{equation}
for some
$c\in \AA_0$. We may assume, without loss of 
generality, that $r=1$. Then $\ell_i(X_{\ii,k})=m_{k,1}$, $2\le k\le r'-1$,
hence by~\eqref{eq:deg-eq}
$$
\ell_i(X_{\ii}^{\mathbf n(r,r')})=-\sum_{k=2}^{r'-1} m_{k,1}a_{i_k,i}=-\sum_{k=2}^{r'-1} (\alpha_i^\vee,\alpha_\ii^{(k)})a_{i_k,i}
=(\alpha_i^\vee,\alpha_i+\alpha_\ii^{(r')})=
m_{r',1}+2.
$$
Applying $\partial_i^{(m_{r',1}+2)}$ to both sides of~\eqref{eq:tmp-determ-c} and taking into account that $\ell_i(X_{\ii,r'})=m_{r',1}+1$ we obtain 
$$
\partial_i^{(top)}X_{\mathbf i,r'}=c\partial_i^{(top)}X_\ii^{\mathbf n(r,r')}.
$$
Since $X_{\ii,r'}$ and $X_\ii^{\mathbf n(r,r')}$ are in~$\mathbf B^{up}$ this implies that~$c=1$.
\end{proof}
\begin{theorem}\label{thm:single rep basis}
Let $w\in W$ and suppose that $\ii=(i_1,\dots,i_m)\in R(w)$ has a single repetition $i_r=i_{r'}=i$. Then 
$$
\mathbf B(w)=\{ q^{\frac12 a\Lambda(\mathbf a,\mathbf e_r+\mathbf e_{r'})}X_\ii^{\mathbf a}Y_\ii^a 
\,:\, \mathbf a=(a_1,\dots,a_m)\in\ZZ_{\ge 0}^{m},\,\min(a_r,a_{r'})=0,\, a\in\ZZ_{\ge 0}\}
$$
where $\Lambda=\Lambda_\ii$ is defined as in Corollary~\ref{cor:skew-symm form} and 
\begin{equation}\label{eq:Y-defn}
Y_\ii =q^{\frac12(\alpha_\ii^{(r)},\alpha_\ii^{(r')})} X_{\mathbf i,r}X_{\mathbf i,r'}-q_{i}^{-1} X_\ii^{\mathbf n(r,r')}
=E_{i_1^{m_{r',1}}\cdots i_{r-1}^{m_{r',r-1}}i_r^{1+m_{r',r}}\cdots i_{r'-1}^{m_{r',r'-1}}i_1^{m_{r,1}}\cdots i_{r-1}^{m_{r,r-1}}i_r}.
\end{equation}
\end{theorem}
\begin{proof}
By~\eqref{eq:deg-eq} $Y_\ii \in U_q(\lie n^+)_{\alpha_\ii^{(r)}+\alpha_\ii^{(r')}}$. It is immediate from~\eqref{eq:single repetition presentation} that 
$\overline{Y_\ii }=Y_\ii $, whence $Y_\ii \in\mathbf B(w)$ by Theorems~\ref{thm:bas-from-Lusztig-lemma} and~\ref{thm:inclusion of base}.
It is easy to see that $(\partial_{i_1}^{op})^{(m_{r,1})}\cdots(\partial_{i_{r-1}}^{op})^{(m_{r,r-1})}\partial_{i_r}^{op}(Y_\ii )=X_{\ii,r'}$
whence $Y_\ii =E_{i_1^{m_{r',1}}\cdots i_{r-1}^{m_{r',r-1}}i_r^{1+m_{r',r}}\cdots i_{r'-1}^{m_{r',r'-1}}i_1^{m_{r,1}}\cdots i_{r-1}^{m_{r,r-1}}i_r}$.

Furthermore, we need the following
\begin{lemma}\label{lem:comm-rels Y}
$X_{\ii}^{\mathbf a}Y_\ii =q^{-\Lambda(\mathbf a,\mathbf e_r+\mathbf e_{r'})}Y_\ii X_{\ii}^{\mathbf a}$ for all $\mathbf a\in\ZZ_{\ge 0}^m$.
\end{lemma}
\begin{proof}
It suffices to prove the assertion for $\mathbf a=\mathbf e_k$, $1\le k\le m$.
By Corollary~\ref{cor:skew-symm form} and Lemma~\ref{lem:Lambda pairing} we have for $k\not=r,r'$ 
\begin{equation}\label{eq:comm-rel quantum plane}
X_{\ii,k}X_\ii^{\mathbf a}=q^{-\Lambda(\mathbf e_k,\mathbf a)}X_\ii^{\mathbf a}X_{\ii,k}
=q^{(\alpha_\ii^{(k)},|\mathbf a_{<k}|_\ii-|\mathbf a_{>k}|_\ii)}X_{\ii}^{\mathbf a}X_{\ii,k}.
\end{equation}
This immediately yields the assertion for $k<r$ or~$k>r'$. 
If~$r<k<r'$ then 
\begin{multline}\label{eq:n(r,r')><}
(\alpha_\ii^{(k)},|\mathbf n(r,r')_{<k}|_\ii-|\mathbf n(r,r')_{>k}|_\ii)
=(\alpha_\ii^{(k)},\alpha_\ii^{(r)}-\alpha_\ii^{(r')}+s_{i_1}\cdots s_{i_{k-1}}(s_{i_k}(\alpha_i)+\alpha_i))
\\=(\alpha_\ii^{(k)},\alpha_\ii^{(r)}-\alpha_\ii^{(r')})+(\alpha_{i_k},s_{i_k}(\alpha_i)+\alpha_i)
=(\alpha_\ii^{(k)},\alpha_\ii^{(r)}-\alpha_\ii^{(r')}),
\end{multline}
where we used~\eqref{eq:deg-eq} and~\eqref{eq:telescope.a}. Thus, 
$X_{\ii,k}X_{\ii}^{\mathbf n(r,r')}=q^{(\alpha_\ii^{(k)},\alpha_\ii^{(r)}-\alpha_\ii^{(r')})}X_\ii^{\mathbf n(r,r')}X_{\ii,k}$.
Since we also have
$$
X_{\ii,k}X_\ii^{\mathbf e_r+\mathbf e_{r'}}=q^{(\alpha_\ii^{(k)},\alpha_\ii^{(r)}-\alpha_\ii^{(r')})}X_\ii^{\mathbf e_r+\mathbf e_{r'}}X_{\ii,k},
$$
we conclude that the assertion holds in this case. 
Furthermore,
\begin{multline*}
X_{\ii,r'}Y_\ii =q^{\frac12(\alpha_\ii^{(r)},\alpha_\ii^{(r')})} X_{\ii,r'}X_{\ii,r}X_{\ii,r'}-q_i^{-1} X_{\ii,r'}X_\ii^{\mathbf n(r,r')}
\\=q^{\frac12(\alpha_\ii^{(r)},\alpha_\ii^{(r')})}(q^{(\alpha_\ii^{(r)},\alpha_\ii^{(r')})} X_{\mathbf i,r}X_{\mathbf i,r'}
+q^{\frac12(\alpha_\ii^{(r')},\alpha_\ii^{(r)})}(q_{i}-q_{i}^{-1})X_\ii^{\mathbf n(r,r')})X_{\ii,r'}-
q_i^{-1} q^{(\alpha_\ii^{(r')},\alpha_\ii^{(r)}+\alpha_\ii^{(r')})}X_\ii^{\mathbf n(r,r')}X_{\ii,r'}
\\
=q^{(\alpha_\ii^{(r)},\alpha_\ii^{(r')})}Y_\ii  X_{\ii,r'}
%=q^{-\Lambda(\mathbf e_{r'},\mathbf e_r+\mathbf e_{r'})} Y_\ii  X_{\ii,r'}
\end{multline*}
and similarly
$Y_\ii  X_{\ii,r}=q^{(\alpha_\ii^{(r)},\alpha_\ii^{(r')})}X_{\ii,r} Y_\ii$. Since $(\alpha_\ii^{(r)},\alpha_\ii^{(r')})=
\Lambda(\mathbf e_r,\mathbf e_r+\mathbf e_{r'})=-\Lambda(\mathbf e_{r'},\mathbf e_r+\mathbf e_{r'})$ this completes the proof of Lemma~\ref{lem:comm-rels Y}.
\end{proof}

\begin{proposition}\label{prop:trans-matrix}
For all $\mathbf a\in \ZZ_{\ge 0}^m$ we have 
$$
X_\ii^{\mathbf a}=\sum_{k+l=\min(a_r,a_{r'})} q_i^{-k(k+|a_r-a_{r'}|)}\qbinom[q_i^{-2}]{\min(a_r,a_{r'})}{k}
b(\mathbf a-\min(a_r,a_{r'})(\mathbf e_r+\mathbf e_{r'})+k\mathbf n(r,r'),l)
$$
where for $\mathbf n\in\ZZ_{\ge 0}^m$ with $\min(n_r,n'_r)=0$ we set 
$$
b(\mathbf n,l)=q^{\frac12 l\Lambda(\mathbf n,\mathbf e_r+\mathbf e_{r'})}X_\ii^{\mathbf n}Y_\ii^{l}
$$
and $\qbinom[v]{m}{n}\in 1+v \ZZ[v]$ is the Gaussian binomial coefficient defined by
$$
\qbinom[v]{m}{n}=\prod_{t=0}^{n-1} \frac{[m-t]_v}{[t+1]_v},\quad [k]_v=\sum_{l=0}^{k-1} v^l. 
$$
\end{proposition}
\begin{proof}
We need the following 
\begin{lemma}\label{lem:special case trans matrix}
$
\displaystyle X_{\ii}^{m(\mathbf e_r+\mathbf e_{r'})}=\sum_{k+l=m} q_i^{-k^2}\qbinom[q_i^{-2}]{m}{k}X_\ii^{k\mathbf n(r,r')} Y_\ii ^{l}$
for all $m\ge 0$.
\end{lemma}
\begin{proof}
The argument is by induction on~$m$. The case~$m=0$ is obvious. For the inductive step, note that we have, by the definition of~$Y_\ii $
\begin{multline*}
X_{\ii}^{(m+1)(\mathbf e_r+\mathbf e_{r'})}=q^{\frac12(m+1)^2(\alpha_\ii^{(r)},\alpha_\ii^{(r')})}
X_{\ii,r}^m X_{\ii,r}X_{\ii,r'}X_{\ii,r'}^{m}
=q^{(\frac12m^2+m)(\alpha_\ii^{(r)},\alpha_\ii^{(r')})}X_{\ii,r}^m(Y_\ii +q_i^{-1}X_\ii^{\mathbf n(r,r')})X_{\ii,r'}^m
\\=X_\ii^{m(\mathbf e_r+\mathbf e_{r'})}Y_\ii +q_i^{-1-2m}
X_\ii^{\mathbf n(r,r')}X_\ii^{m(\mathbf e_r+\mathbf e_{r'})}.
\end{multline*}
By Corollary~\ref{cor:skew-symm form} we have  $X_\ii^{\mathbf a}X_\ii^{k\mathbf a}=X_\ii^{(k+1)\mathbf a}$ if $\mathbf a\in\sum_{t=r+1}^{r'-1}\ZZ_{\ge 0}\mathbf e_t$, whence
by the induction hypothesis
\begin{align*}
X_\ii^{(m+1)(\mathbf e_r+\mathbf e_{r'})}&=\sum_{k+l=m} q_i^{-k^2}\qbinom[q_i^{-2}]{m}{k}X_\ii^{k\mathbf n(r,r')}Y_\ii ^{l+1}
+\sum_{k+l=m} q_i^{-k^2-1-2m}\qbinom[q_i^{-2}]{m}{k} X_\ii^{(k+1)\mathbf n(r,r')}Y_\ii ^l
\\
&=
\sum_{k+l=m+1} q_i^{-k^2}\Big(\qbinom[q_i^{-2}]{m}{k}+q_i^{-2(m+1-k)}\qbinom[q_i^{-2}]{m}{k-1}\Big)X_\ii^{k\mathbf n(r,r')}Y_\ii ^l
\\&=\sum_{k+l=m+1} q_i^{-k^2}\qbinom[q_i^{-2}]{m+1}{k}X_\ii^{k\mathbf n(r,r')}Y_\ii ^l.\qedhere
\end{align*}
\end{proof}
Using Corollary~\ref{cor:skew-symm form} we can write
$$
X_\ii^{\mathbf a}=q^{\frac12\Lambda(\mathbf a,a_r \mathbf e_r+a_{r'}\mathbf e_{r'})}X_\ii^{\mathbf a-a_r\mathbf e_r-a_{r'}\mathbf e_{r'}}
X_\ii^{a_r\mathbf e_r+a_{r'}\mathbf e_{r'}}
=q^{-\frac12\Lambda(\mathbf a,a_r \mathbf e_r+a_{r'}\mathbf e_{r'})}X_\ii^{a_r\mathbf e_r+a_{r'}\mathbf e_{r'}}
X_\ii^{\mathbf a-a_r\mathbf e_r-a_{r'}\mathbf e_{r'}}.
$$
If $a_r\ge a_{r'}$ then 
\begin{multline*}
X_\ii^{\mathbf a}=q^{\frac12(\Lambda(\mathbf a,a_r \mathbf e_r+a_{r'}\mathbf e_{r'})+\Lambda((a_r-a_{r'})\mathbf e_r,a_{r'}\mathbf e_{r'}))}
X_\ii^{\mathbf a-a_r\mathbf e_r-a_{r'}\mathbf e_{r'}}
X_\ii^{(a_r-a_{r'})\mathbf e_r}X_{\ii}^{a_{r'}(\mathbf e_r+\mathbf e_{r'})}
\\
=q^{\frac12 a_{r'}\Lambda(\mathbf a,\mathbf e_r+\mathbf e_{r'})}X_\ii^{\mathbf a-a_{r'}(\mathbf e_r+\mathbf e_{r'})}
X_{\ii}^{a_{r'}(\mathbf e_r+\mathbf e_{r'})}.
\end{multline*}
Note that $\Lambda(\mathbf e_r+\mathbf e_{r'},\mathbf n(r,r'))=0$. Then for $0\le k\le a'_r$
\begin{multline*}
q^{\frac12 a_{r'}\Lambda(\mathbf a,\mathbf e_r+\mathbf e_{r'})}X_\ii^{\mathbf a-a_{r'}(\mathbf e_r+\mathbf e_{r'})}
X_\ii^{k\mathbf n(r,r')}Y_\ii ^{a_{r'}-k}\\
=q^{\frac12 \Lambda(\mathbf a,a_{r'}(\mathbf e_r+\mathbf e_{r'})-k\mathbf n(r,r'))}X_\ii^{\mathbf a-a_{r'}(\mathbf e_r+\mathbf e_{r'})+
k\mathbf n(r,r')}Y_\ii ^{a_{r'}-k}
\\
=q^{\frac12 k\Lambda(\mathbf a,\mathbf e_r+\mathbf e_{r'}-\mathbf n(r,r'))}b(\mathbf a-a_{r'}(\mathbf e_r+\mathbf e_{r'})+k\mathbf n(r,r'),a_{r'}-k).
\end{multline*}
If $t<r$ or~$t>r'$ then
$$
\Lambda(\mathbf e_t,\mathbf e_r+\mathbf e_{r'}-\mathbf n(r,r'))=\pm(\alpha_i^{(t)},\alpha_\ii^{(r)}+\alpha_\ii^{(r')}-|\mathbf n(r,r')|_\ii)=0
$$
by~\eqref{eq:deg-eq}. For $r<t<r'$ it follows from~\eqref{eq:n(r,r')><} that 
$$
\Lambda(\mathbf e_t,\mathbf e_r+\mathbf e_{r'}-\mathbf n(r,r'))=(\alpha_\ii^{(t)},\alpha_\ii^{(r')}-|\mathbf n(r,r')_{>t}|_\ii
-\alpha_\ii^{(r)}+|\mathbf n(r,r')_{<t}|_\ii)=0.
$$
Since by Lemma~\ref{lem:Lambda pairing}
$$
\Lambda(\mathbf e_r,\mathbf e_r+\mathbf e_{r'}-\mathbf n(r,r'))=(\alpha_\ii^{(r)},\alpha_\ii^{(r')}-\mathbf n(r,r'))=
-(\alpha_\ii^{(r)},\alpha_\ii^{(r)})=-(\alpha_i,\alpha_i)
$$
while 
$$
\Lambda(\mathbf e_{r'},\mathbf e_r+\mathbf e_{r'}-\mathbf n(r,r'))=(\alpha_\ii^{(r')},-\alpha_\ii^{(r)}+\mathbf n(r,r'))=(\alpha_\ii^{(r')},\alpha_\ii^{(r')})
=(\alpha_i,\alpha_i)
$$
we conclude that $\Lambda(\mathbf a,\mathbf e_r+\mathbf e_{r'}-\mathbf n(r,r'))=(\alpha_i,\alpha_i)(a_{r'}-a_r)$.
Thus, by Lemma~\ref{lem:special case trans matrix} we have 
$$
X_\ii^{\mathbf a}=\sum_{k+l=a_{r'}} q_i^{-k(k+a_r-a_{r'})}\qbinom[q_i^{-2}]{a_{r'}}{k} b(\mathbf a-a_{r'}(\mathbf e_r+\mathbf e_{r'})+
k\mathbf n(r,r'),l).
$$

For $a_r\le a_{r'}$ we obtain in a similar way
$$
X_\ii^{\mathbf a}=
q^{-\frac12\Lambda(\mathbf a,a_r \mathbf e_r+a_{r'}\mathbf e_{r'})}X_\ii^{a_r\mathbf e_r+a_{r'}\mathbf e_{r'}}
X_\ii^{\mathbf a-a_r\mathbf e_r-a_{r'}\mathbf e_{r'}}
=q^{-\frac12a_r \Lambda(\mathbf a,\mathbf e_r+\mathbf e_{r'})}
X_\ii^{a_r(\mathbf e_r+\mathbf e_{r'})}
X_\ii^{\mathbf a-a_r(\mathbf e_r+\mathbf e_{r'})}.
$$
Since for $0\le k\le a_r$
\begin{multline*}
q^{-\frac12a_r \Lambda(\mathbf a,\mathbf e_r+\mathbf e_{r'})}
X_\ii^{k\mathbf n(r,r')}Y_\ii^{a_r-k}
X_\ii^{\mathbf a-a_r(\mathbf e_r+\mathbf e_{r'})}
\\
=q^{-\frac12 k\Lambda(\mathbf a,\mathbf e_r+\mathbf e_{r'}-\mathbf n(r,r'))}b(\mathbf a-a_r(\mathbf e_r+\mathbf e_{r'})+k\mathbf n(r,r'),a_r-k),
\end{multline*}
it follows that
$$
X_\ii^{\mathbf a}=\sum_{k+l=a_{r}} q_i^{-k(k+a_{r'}-a_{r})}\qbinom[q_i^{-2}]{a_{r}}{k} b(\mathbf a-a_{r}(\mathbf e_r+\mathbf e_{r'})+
k\mathbf n(r,r'),l).
$$
Proposition~\ref{prop:trans-matrix} is proved.
\end{proof}

By Lemma~\ref{lem:comm-rels Y}, $\overline{\mathbf b(\mathbf n,l)}=\mathbf b(\mathbf n,l)$ provided that 
$\min(n_r,n_{r'})=0$. Then Proposition \ref{prop:trans-matrix} and Theorems \ref{thm:bas-from-Lusztig-lemma}, \ref{thm:inclusion of base} imply that 
$b(\mathbf a-\min(a_r,a'_r)(\mathbf e_r+\mathbf e_{r'}),\min(a_r,a'_r))=b_{\ii,\mathbf a}\in\mathbf B(w)$. Clearly this gives the $b_{\ii,\mathbf a}$
for all $\mathbf a\in\ZZ_{\ge 0}^m$, which completes the proof of Theorem~\ref{thm:single rep basis}.
\end{proof}
\subsection{Bases for type \texorpdfstring{$A_3$}{A\_3}}\label{subs:ex A_3}
Let $w_\circ$ be the longest element in~$W$. 
We have $E_{ij}=T_i(E_j)$, $\{i,j\}=\{1,2\}$ or $\{2,3\}$,
$E_{123}=T_1T_2(E_3)=T_3^{-1}T_2^{-1}(E_1)$,
$E_{321}=T_3T_2(E_1)=T_1^{-1}T_2^{-1}(E_3)$, $E_{132}=T_1T_3(E_2)$, $E_{213}=T_2T_1T_3(E_2)=E_{132}^*=T_1^{-1}T_3^{-1}(E_2)$
and $E_{2132}=Y_{(2,1,3,2)}=E_2 E_{213}-q^{-1}E_{21}E_{23}$ as defined in Theorem~\ref{thm:single rep basis}. 
The following was essentially proved in~\cite{BZ1}, although with a slightly different definition of $\bar\cdot$ and hence with
different powers of~$q$ (see also Theorems~1.4.1 and~3.1.3 in a recent work~\cite{Qin}).
\begin{theorem}
$\mathbf B^{up}=\mathbf B(w_\circ)$ consists of monomials
$$
q^{\frac12 f(\mathbf a)} E_1^{m_1}E_2^{m_2}E_3^{m_3}E_{12}^{m_{12}}E_{21}^{m_{21}}
E_{23}^{m_{23}}E_{32}^{m_{32}}E_{213}^{m_{213}}E_{132}^{m_{132}}E_{123}^{m_{123}}E_{321}^{m_{321}}E_{2132}^{m_{2132}}
$$
where 
\begin{multline*}
f(\mathbf a)=(m_1-m_2)(m_{12}-m_{21})+(m_3-m_2)(m_{32}-m_{23})+(m_1+m_3)(m_{132}-m_{213})\\
+(m_1+m_{12}+m_{21}-m_3-m_{23}-m_{32})(m_{123}-m_{321})
-(m_{12}+m_{32})m_{132}+(m_{21}+m_{23})m_{213},
\end{multline*}
and $\min(m_\alpha,m_\beta)=0$ if  
$E_\alpha$, $E_\beta\notin\{ E_{123},E_{321},E_{2132}\}$ and 
are not connected by an edge in the following graph (see~\cite{BZ1}*{\S9.4, Fig~2})
$$
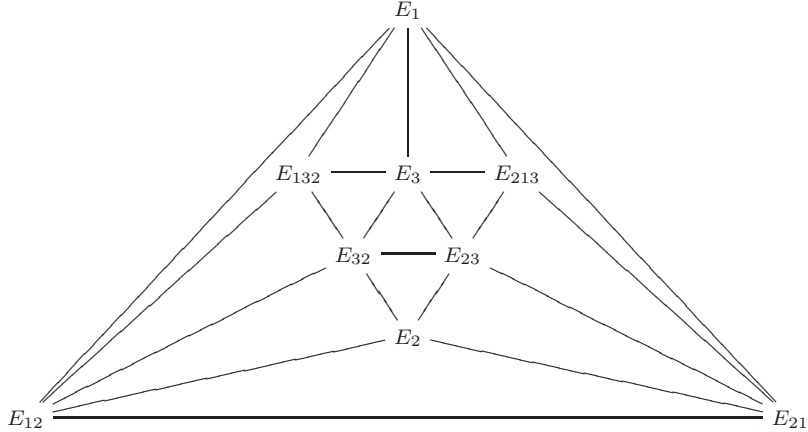
\xymatrix@!0@R=6ex@C=4ex{
&&&&&&&E_1\ar@{-}[dd]\ar@{-}[lldd]\ar@{-}[rrdd]&&&\\
\\
&&&&&E_{132}\ar@{-}[rr]\ar@{-}[rd]&&E_3\ar@{-}[ld]\ar@{-}[rd]&&E_{213}\ar@{-}[ld]\ar@{-}[ll]\\
&&&&&&E_{32}\ar@{-}[rd]\ar@{-}[rr]&&E_{23}\ar@{-}[ld]\\
&&&&&&&E_2
\\
E_{12}\ar@{-}[rrrrrrrrrrrrrr]
\ar@{-}[rrrrruuu]\ar@{-}[rrrrrrruuuuu]\ar@{-}[rrrrrruu]\ar@{-}[rrrrrrru]&&&&&&&&&&&&&&E_{21}\ar @{-}[llllluuu]\ar @{-}[llllllluuuuu]\ar @{-}[lllllluu]\ar @{-}[lllllllu]}
$$
\end{theorem}
We have the following table for the action of the $T_i^{-1}$, $1\le i\le 3$ on the~$E_\alpha$
$$
\begin{array}{c|cccccccccccc}
&E_1&E_2&E_3&E_{12}&E_{21}&E_{23}&E_{32}&E_{132}&E_{213}&E_{123}&E_{321}&E_{2132}\\\hline
T_1^{-1}& &E_{21}&E_3&E_2&&E_{213}&E_{321}&E_{32}&&E_{23}&&E_{2132}\\
T_2^{-1}&E_{12}&&E_{32}&&E_1&E_3&&&E_{132}&E_{123}&E_{321}&\\
T_3^{-1}&E_1&E_{23}&&E_{123}&E_{213}&&E_2&E_{12}&&&E_{21}&E_{2132}\\
\end{array}
$$
where the entry is empty if $T_i^{-1}(E_\alpha)\notin U_q(\lie n^+)$.
Using Theorem~\ref{thm:T_w(B(w'))} we conclude that $\mathbf B(s_1w_\circ)$ (respectively, $\mathbf B(s_2w_\circ)$) consists of monomials of the form 
\begin{multline*}
q^{\frac12(m_2m_{21}+(m_3-m_2)(m_{32}-m_{23})-(m_{21}-m_3-m_{23}-m_{32})m_{321}
-m_3m_{213}+(m_{21}+m_{23})m_{213})}\times\\ E_2^{m_2}E_3^{m_3}E_{21}^{m_{21}}
E_{23}^{m_{23}}E_{32}^{m_{32}}E_{213}^{m_{213}}E_{321}^{m_{321}}E_{2132}^{m_{2132}}
\end{multline*}
and, respectively
\begin{multline*}
q^{\frac12(m_1m_{12}+m_3m_{32}+(m_1+m_{12}-m_3-m_{32})(m_{123}-m_{321})
+(m_1+m_3-m_{12}+m_{32})m_{132})}\times\\ E_1^{m_1}E_3^{m_3}E_{12}^{m_{12}}
E_{32}^{m_{32}}E_{132}^{m_{132}}E_{123}^{m_{123}}E_{321}^{m_{321}}
\end{multline*}
where $\min(m_\alpha,m_\beta)=0$ if $E_\alpha,E_\beta\notin\{E_{123},E_{321},E_{2132}\}$ are not connected by an edge in the following respective graphs
$$
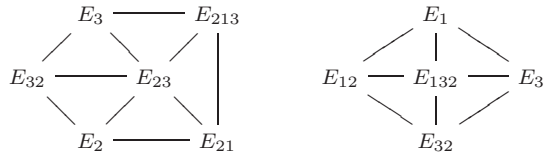
\xymatrix@!0{
&E_3\ar@{-}[ld]\ar@{-}[rd]&&E_{213}\ar@{-}[ld]\ar@{-}[ll]\\
E_{32}\ar@{-}[rd]\ar@{-}[rr]&&E_{23}\ar@{-}[ld]\\
&E_2&&E_{21}\ar@{-}[ll]\ar@{-}[lu]\ar@{-}[uu]}
\qquad 
\xymatrix@!0@C=7ex{&E_1\ar@{-}[rd]\ar@{-}[d]\\E_{12}\ar@{-}[r]\ar@{-}[ru]\ar@{-}[rd]&E_{132}\ar@{-}[r]\ar@{-}[d]&E_3\ar@{-}[dl]\\&E_{32}
}
$$
The basis $\mathbf B(s_3w_\circ)$ is easy to obtain from $\mathbf B(s_1w_\circ)$ using the diagram automorphism of~$U_q(\lie n^+)$ which interchanges 
$E_1$ and~$E_3$, $E_{12}$ and~$E_{32}$, $E_{21}$ and $E_{23}$ and $E_{123}$, $E_{321}$ and fixes all other elements $E_\alpha$.

Thus, $U_q(s_1w_\circ)$ is generated by $E_2$, $E_3$, $E_{21}$ subject to the relations 
\begin{gather*}
[E_i,[E_i,E_j]_q]_{q^{-1}}=0,\quad [E_2,E_{21}]_{q^{-1}}=0,\quad [E_3,[E_3,E_{21}]_q]_{q^{-1}}=0=[E_{21},[E_{21},E_3]_q]_{q^{-1}},
\end{gather*}
where $[x,y]_t=xy-t yx$, $x,y\in U_q(\lie n^+)$, $t\in\kk^\times$ and
$\{i,j\}=\{2,3\}$,
while $U_q(s_2w_\circ)$ is generated by $E_1$, $E_3$, $E_{12}$, $E_{32}$ subject to the relations
$$
[E_1,E_3]=0,\, [E_i,E_{i2}]_{q^{-1}}=0,\, [E_{12},E_{32}]=0,\quad [E_i,[E_i,E_{j2}]_q]_{q^{-1}}=0,\,
[E_{i2},[E_{i2},E_j]_q]_{q^{-1}}=0,
$$
where $\{i,j\}=\{1,3\}$.

Since all elements $w\in W$ with $\ell(w)\le 4$ are either repetition free or with a single repetition, all remaining Schubert cells 
have already been described in~\S\ref{subs:rep-free} and~\S\ref{subs:single rep}. For example, 
$$
\mathbf B(s_2s_1s_3s_2)=\{ 
q^{\frac12(m_2+m_{213})(m_{21}+m_{23})} E_2^{m_2}E_{21}^{m_{21}}
E_{23}^{m_{23}}E_{213}^{m_{213}}E_{2132}^{m_{2132}}\,:\, \min(m_2,m_{213})=0\}
$$
and $U_q(s_2s_1s_3s_2)$ is generated by $E_2$, $E_{21}$, $E_{23}$, $E_{213}$ subject to the relations
$$
[E_2,E_{2i}]_{q^{-1}}=0,\quad [E_{21},E_{23}]=0, \quad [E_{2i},E_{213}]_{q^{-1}}=0,\quad [E_2,E_{213}]=(q^{-1}-q)E_{21}E_{23},\, i\in\{1,3\}
$$
and coincides with the algebra of quantum $2\times 2$-matrices.

\subsection{Bases for type \texorpdfstring{$C_2$}{C\_2}}
We have $E_{12}=T_2^{-1}(E_1)$, $E_{1^22}=T_1(E_2)$, $E_{21}=T_2(E_1)$, $E_{21^2}=T_1^{-1}(E_2)$, $E_{121}=Y_{(1,2,1)}$ and $E_{21^2 2}=Y_{(2,1,2)}$
as defined in Theorem~\ref{thm:single rep basis}.
The following is apparently well-known (and can be deduced for instance from~\cite{Qin}*{Theorems~1.4.1 and~3.1.3}).
\begin{theorem}
$\mathbf B^{up}$ consists of all monomials 
$$
q^{m_1 (m_{1^22}-m_{21^2})+m_2(m_{21}-m_{12})-m_{12}m_{1^22}+m_{21}m_{21^2}} E_1^{m_1}E_2^{m_2}E_{12}^{m_{12}}E_{21}^{m_{21}}E_{1^22}^{m_{1^22}}E_{21^2}^{m_{21^2}}
E_{121}^{m_{121}}E_{21^22}^{m_{21^22}}
$$
where $\min(m_\alpha,m_\beta)=0$ if $E_\alpha,E_\beta\notin\{E_{121},E_{21^22}\}$ 
are not connected by an edge in the following graph 
$$
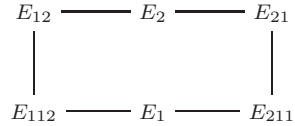
\xymatrix{ 
E_{12}\ar@{-}[d]\ar@{-}[r]&E_2&E_{21}\ar@{-}[l]\ar@{-}[d]\\
E_{112}&E_1\ar@{-}[r]\ar@{-}[l]&E_{211}
}
$$
\end{theorem}
All other Schubert cells have already been described in~\S\ref{subs:rep-free} and~\S\ref{subs:single rep}.

\subsection{Bases for bi-Schubert algebras}\label{subs:bi-schub}
Let $\lie g=\lie{sl}_4$. Using the computations from~\S\ref{subs:ex A_3}
we obtain
$$
\mathbf B(s_1w_\circ,s_1w_\circ)=
\{ q^{\frac12(m_3-m_2)(m_{32}-m_{23})}E_2^{m_2}E_3^{m_3}
E_{23}^{m_{23}}E_{32}^{m_{32}}E_{2132}^{m_{2132}}\,:\,\min(m_2,m_3)=0\}
$$
and $U_q(s_1w_\circ,s_1w_\circ)\cong U_q(\lie{sl}_3^+)\tensor \kk[E_{2132}]$,
\begin{multline*}
\mathbf B(s_1w_\circ,s_2w_\circ)=\{ 
q^{\frac12(-m_3(m_{23}+m_{213})-(m_{21}-m_3-m_{23})m_{321}+(m_{21}+m_{23})m_{213})}E_3^{m_3}E_{21}^{m_{21}}
E_{23}^{m_{23}}E_{213}^{m_{213}}E_{321}^{m_{321}}\,:\\\,\min(m_3,m_{21})=0\}
\end{multline*}
and $U_q(s_1w_\circ,s_2w_\circ)$ is generated by $E_3$, $E_{21}$ and $E_{23}$ subject to the relations
$$
[E_3,E_{23}]_q=0,\quad [E_3,[E_3,E_{21}]_q]_{q^{-1}}=0=[E_{21},[E_{21},E_3]_q]_{q^{-1}},
$$
$$
\mathbf B(s_1w_\circ,s_3w_\circ)=\{
q^{\frac12(m_2-m_{321})(m_{21}-m_{32})}E_2^{m_2}E_{21}^{m_{21}}
E_{32}^{m_{32}}E_{321}^{m_{321}}E_{2132}^{m_{2132}}\,:\,\min(m_{21},m_{32})=0\}
$$
and $U_q(s_1w_\circ,s_3w_\circ)$ is generated by $E_2$, $E_{21}$, $E_{32}$, $E_{321}$ subject to the relations 
$$
[E_2,E_{21}]_{q^{-1}}=[E_2,E_{32}]_q=[E_{21},[E_{21},E_{32}]]_{q^2}=[E_{32},[E_{32},E_{21}]]_{q^{-2}}=0
$$
and $[E_2,E_{321}]=[E_{21},E_{321}]_q=[E_{32},E_{321}]_{q^{-1}}=0$,
$$
\mathbf B(s_2w_\circ,s_2w_\circ)=\{q^{\frac12(m_1-m_3)(m_{123}-m_{321})}E_1^{m_1}E_3^{m_3}E_{123}^{m_{123}}E_{321}^{m_{321}}\}
$$
and $U_q(s_2w_\circ,s_2w_\circ)$ is a quantum plane. In particular, all these algebras are PBW in the sense 
of Remark~\ref{rem:intersect sc}\eqref{rem:intersect sc.a}.

\end{document}